\documentclass[11pt, leqno]{amsart}
\usepackage{amsthm, amsmath}
\usepackage{amssymb} 
\usepackage{amscd}
\usepackage[curve,matrix,arrow,frame,tips]{xy}
\usepackage{verbatim}	
\usepackage{color}
\usepackage{graphicx}	
\newtheorem{thm}{Theorem}[section]
\newtheorem{prop}[thm]{Proposition}
\newtheorem{lemma}[thm]{Lemma}
\newtheorem{defn}[thm]{Definition}

\newtheorem{rmks}[thm]{Remarks}
\newtheorem{rmk}[thm]{Remark}

\newtheorem{cor}[thm]{Corollary}

\newcommand{\cat}[1]{\mathcal #1}
\newcommand{\x}{\times}
\newcommand{\hs}[1]{\hspace{#1pt}}
\newcommand{\vs}[1]{\vspace{#1pt}}
\newcommand{\tab}{\hs{10}}
\newcommand{\vtab}{\vs{15}}
\newcommand{\id}{\mathbb{I}}
\newcommand{\imply}{\tab \Rightarrow \tab}
\newcommand{\defeq}{\stackrel{def}{=}}
\renewcommand{\subset}{\subseteq}
\renewcommand{\supset}{\supseteq}
\newcommand{\Z}{\mathbb{Z}}

\newcommand{\embed}{\hookrightarrow}
\renewcommand{\O}{\cat{O}}
\newcommand{\listing}[3]{#1,\ #2, \ldots, #3}
\newcommand{\iso}{\ \tilde{\to}\ }

\newcommand{\squarediagram}[4]{
\begin{center}
$\begin{CD}
#1    @>>> #2   \\
@VVV       @VVV \\
#3    @>>> #4
\end{CD}$
\end{center}
}
\newcommand{\squarediagramword}[8]{
\begin{center}
$\begin{CD}
#1      @>#5>>  #2     \\
@V#6VV          @VV#7V \\
#3      @>#8>>  #4
\end{CD}$
\end{center}
}
\newcommand{\A}{\mathbb{A}}
\renewcommand{\P}{\mathbb{P}}

\newcommand{\limit}[1]{\displaystyle \lim_{#1}}
\newcommand{\repn}{representation}
\newcommand{\ga}{\alpha}
\newcommand{\gb}{\beta}
\newcommand{\gph}{\phi}
\newcommand{\gp}{\pi}
\newcommand{\gx}{\xi}
\newcommand{\gps}{\psi}
\newcommand{\go}{\omega}
\newcommand{\gs}{\sigma}
\renewcommand{\gg}{\gamma}

\newcommand{\gm}{\mu}
\renewcommand{\ge}{\eta}

\newcommand{\gPS}{\Psi}
\newcommand{\gPH}{\Phi}
\newcommand{\gD}{\Delta}
\newcommand{\gO}{\Omega}
\newcommand{\gc}{\chi}
\newcommand{\nbundle}[2]{\cat{N}_{#1 \embed #2}}

\newcommand{\claimend}{\tab $\triangle$}
\newcommand{\dual}[1]{#1^{\vee}}	
\newcommand{\Q}{\mathbb{Q}}
\newcommand{\R}{\mathbb{R}}
\newcommand{\C}{\mathbb{C}}
\newcommand{\pt}{{\rm Spec}\, k}
\newcommand{\morp}{morphism}

\newcommand{\etale}{\'etale}
\newcommand{\wrt}{with respect to}

\newcommand{\resp}{respectively}
\newcommand{\withoutlog}{ithout loss of generality}
\newcommand{\longto}{\longrightarrow}
\newcommand{\closure}[2]{{\rm clos}_{#2}{#1}}
\newcommand{\singular}[1]{{\rm Sing}(#1)}
\newcommand{\blowup}[2]{{\rm Blow}_{#2}{#1}}
\newcommand{\disjoint}{\amalg}
\newcommand{\spec}[1]{{\rm Spec}\, #1}
\newcommand{\divisor}[1]{{\rm div}\, #1}
\newcommand{\morpset}[3]{{\rm Mor}_{#1}(#2,#3)}
\newenvironment{statementslist}{\begin{tabular}[t]{p{15pt}p{380pt}}}{\end{tabular}}
\newcommand{\kernel}[1]{{\rm kernel}\, #1}

\newcommand{\stricttransform}[1]{\left\langle #1 \right\rangle}

\newcommand{\gradient}[1]{\nabla #1}
\newcommand{\picard}[2]{{\rm Pic}^{#1}(#2)}
\newcommand{\Endo}[1]{{\rm End}\,(#1)}
\newcommand{\homo}{homomorphism}
\renewcommand{\char}[1]{{\rm char\,#1}}
\renewcommand{\span}[2]{#1{\rm -span}\{#2\}}

\newcommand{\equi}{equivariant}
\newcommand{\proj}{projective}
\newcommand{\qproj}{quasi-projective}
\newcommand{\sm}{smooth}
\newcommand{\smcat}{\textit{Sm}}
\newcommand{\gsmcat}[1]{\textit{#1-Sm}}

\newcommand{\lazard}{\mathbb{L}}
\newcommand{\glin}[1]{$#1$-linearized}
\newcommand{\cobg}[3]{\cat{U}_{#1}^{#2}(#3)}
\newcommand{\cob}[1]{\cat{U}_{#1}}
\newcommand{\coblocaltheory}[2]{\cat{U}_{#1}[1/{#2}]}
\newcommand{\coblocal}[3]{\cat{U}_{#1}(#2)[1/{#3}]}
\newcommand{\cobglow}[3]{\cat{U}_{#2}^{#1}(#3)}

\renewcommand{\L}{\cat{L}}
\newcommand{\ggen}{globally\ generated}

\newcommand{\nice}{nice}

\newcommand{\adinvsh}{admissible\ invertible\ shea}
\newcommand{\adtower}{admissible\ tower}
\newcommand{\qadtower}{quasi-admissible\ tower}

\newcommand{\rsncd}{reduced\ strict\ normal\ crossing\ divisor}
\newcommand{\girred}{$G$-irreducible}

\newcommand{\rou}{root\ of\ unity}
\newcommand{\fgldiv}[1]{F^{1/#1}}

\newcommand{\good}{good}
\newcommand{\lfsheaf}[1]{locally free sheaf of rank $#1$}
\newcommand{\fixptmap}{\cat{F}}
\newcommand{\dpr}{ouble point relation}
\newcommand{\edpr}{xtended double point relation}
\newcommand{\gdpr}{eneralized double point relation}
\newcommand{\gdprdivisors}{$A_1, \ldots, A_n$, $B_1, \ldots, B_m$}
\newcommand{\gdprdivisorsequi}{A_1 + \cdots + A_n \sim B_1 + \cdots + B_m}
\newcommand{\gdprdivisorssum}{A_1 + \cdots + A_n + B_1 + \cdots + B_m}
\newcommand{\generalterm}{X_i \cdots U^p_k \cdots}

\newcommand{\fgl}{formal group law}
\newcommand{\apicard}[2]{{\rm A}\picard{#1}{#2}}
\newcommand{\inv}{invariant}
\newcommand{\equidim}{equidimensional}

\begin{document}

%Title
\title{Equivariant Algebraic Cobordism and Double Point Relations}
\author{C. L. Liu}
\begin{abstract}For a reductive connected group or a finite group over a field of characteristic zero, we define an equivariant algebraic cobordism theory by a generalized version of the double point relation of Levine-Pandharipande. We prove basic properties and the well-definedness of a canonical fixed point map. We also find explicit generators of the algebraic cobordism ring of the point when the group is finite abelian.
\end{abstract}

\address{Department of Mathematics, Michigan State University, East Lansing, MI 48824-1027, USA} 
\email{liuchun@math.msu.edu}
\date{\today}

\maketitle

\medskip

%{\small \noindent ABSTRACT.}\tab For a reductive connected group or a finite group over a field of characteristic zero, we define an equivariant algebraic cobordism theory by a generalized version of the double point relation of Levine-Pandharipande. We prove basic properties and the well-definedness of a canonical fixed point map. We also find explicit generators of the algebraic cobordism ring of the point when the group is finite abelian.

\bigskip
%\begin{comment}
\centerline{\sc Contents}
\medskip

\noindent \S 1. Introduction\\
\S 2. Notations and assumptions\\
\S  3. Geometric \equi\ algebraic cobordism $\cob{G}$\\
 \S 4. The Chern class operator $c(\L)$\\
\S 5. More properties for $\cob{G}$\\
\S 6. Generators for the \equi\ algebraic cobordism ring\\
\S 7. Fixed point map\\
\noindent  \ \ \  \ \ References\\

\bigskip
\medskip

\medskip

\section{Introduction}

Cobordism is a deep and well-developed theory in topology. According to Thom's definition, two dimension $d$ \sm\ oriented manifolds $M,N$ are said to be cobordant if there exists a dimension $d+1$ \sm\ oriented manifold with boundary $M \disjoint (-N)$ (Negative sign means opposite orientation). By definition, the set of all cobordism classes, with addition given by disjoint union and multiplication given by product, is called the oriented bordism ring $U_*$ (grading given by dimension). This ring was well-studied. For instance, Thom showed that the torsion free part can be described by $U_* \otimes_{\Z} \Q \cong \Q[x_{4k}\ |\ k \geq 1]$. In addition,  Milnor and Wall showed that all torsion has order 2. The main technique involved was the use of the Thom spectrum which we will briefly explain below. 

Consider a $SO(n)$-bundle $E$ over a manifold $X$. Let $D$ be the set of all vectors (fiberwise) with length $\leq 1$ and $S$ be the set of all vectors with length 1. Then, the Thom space is defined as the quotient space $D/S$ and denoted by $T(E)$. Now consider the classifying space $BSO(n)$ with universal $SO(n)$-bundle $E_n$. Denote the Thom space $T(E_n)$ by $MSO(n)$. Notice that $E_n \x \R^1$ becomes a $SO(n+1)$-bundle over $BSO(n)$ and hence induces the classifying map $BSO(n) \to BSO(n+1)$ and $E_n \x \R^1 \to E_{n+1}$. Apply the Thom space construction on both sides of the second map, we get 
$$MSO(n) \wedge S^1 \cong T(E_n \x \R^1) \to T(E_{n+1}) = MSO(n+1).$$
That defines the Thom spectrum $MSO$. We can then consider the homotopy groups of $MSO$, namely $\gp_k(MSO) \defeq \limit{\stackrel{\longto}{n}}\ \gp_{n+k}(MSO(n))$. The importance of the Thom spectrum comes from the isomorphism $U_k \iso \gp_k(MO)$ which is given by the Pontrjagin-Thom construction (see \cite{notesoncobordism}).

More generally, for a \sm\ oriented manifold $X$, we say two maps $f_1 : Y_1 \to X$ and $f_2 : Y_2 \to X$, where $Y_1,Y_2$ are both dimension $d$ \sm\ oriented manifolds, are cobordant if there exists a map $F : Z \to X$ such that $Z$ is a dimension $d+1$ \sm\ oriented manifold with boundary and $F|_{\partial Z} = f_1 \disjoint (-f_2)$ (Negative sign means opposite orientation on domain). The set of all cobordism classes with addition given by disjoint union is denoted by $U_*(X)$ (grading given by dimension of the domain of the map). That is the oriented bordism group.

Other than oriented bordism theory, there are other bordism (or cobordism) theories. For example, for a stably complex manifold $X$, we define 
$$MU_k(X) \defeq \limit{\stackrel{\longto}{n}}\ [S^{2n+k}, MU(n) \wedge X]$$ 
and
$$MU^k(X) \defeq \limit{\stackrel{\longto}{n}}\ [S^{2n-k} \wedge X, MU(n)]$$  
where $MU(n)$ is the Thom space of the universal $U(n)$-bundle over the classifying space $BU(n)$. This way, one defines the complex bordism theory (given by $MU_*(X)$) and the complex cobordism theory (given by $MU^*(X)$). Milnor showed \cite{complexcobordism} that the complex bordism ring $MU_*$ is just a polynomial ring $\Z[x_{2k}\ |\ k \geq 1]$ and $MU^* \cong MU_{-*}$. 

Moreover, this complex cobordism theory is equipped with Chern classes and it leads to what is called the formal group law. More precisely, for each complex vector bundle $E$ over $X$ of rank $r$, there are Chern classes $c_i(E) \in MU^{2i}(X)$ for $1 \leq i \leq r$ associated to it (see \cite{cobordismandktheory}). It turns out the complex cobordism group $MU^*(\C \P^{\infty})$ is given by the power series ring $MU^*[[s]]$ and the tensor product map $\C \P^{\infty} \x \C \P^{\infty} \to \C \P^{\infty}$ will define a Hopf-algebra structure on $MU^*(\C \P^{\infty})$. Thus, we obtain a map
$$MU^*[[s]] \cong MU^*(\C \P^{\infty}) \to MU^*(\C \P^{\infty} \x \C \P^{\infty}) \cong MU^*(\C \P^{\infty}) \hat{\otimes} MU^*(\C \P^{\infty}) \cong MU^*[[u,v]].$$
Denote the image of $s$ by $F \in MU^*[[u,v]]$. Since $\C \P^{\infty}$ is the classifying space for $U(1)$ and the elements $s,u$ and $v$ correspond to $c_1(\O(1))$, $c_1(\O(1,0))$ and $c_1(\O(0,1))$ respectively, we obtain the following relation for any pairs of complex line bundles $L_1,L_2$ over $X$ :
$$c_1(L_1 \otimes L_2) = F(c_1(L_1), c_1(L_2))$$
as elements inside $MU^*(X)$. This power series $F$ is called a formal group law over $MU^*$ (see \cite{formalgrouplaw}).

Unfortunately, because of the lack of the notion of boundary in the category of algebraic varieties, an algebraic version of cobordism theory can not be defined in a similar manner. There is a naive approach which turns out to be unsuccessful. We may define two dimension $d$ \sm\ \proj\ varieties $X$, $X'$ to be cobordant if there exists a \morp\ $Y \to \P^1$ where $Y$ is a dimension $d+1$ \sm\ \proj\ variety such that $X$, $X'$ are the fibers over $0$, $1$ respectively. This approach was also addressed by M. Levine and F. Morel (see remark 1.2.9 in \cite{LeMo} for more detail). Consider the case when $d=1$. Since the genus and the number of connected components are invariant under this concept of cobordism, we can not decompose a \sm\ genus $g$ curve. Hence, the cobordism group of curves will be much bigger than $\Z$, which is what we expect from the theory of complex cobordism.

Nevertheless, in \cite{LeMo}, Levine and Morel managed to define an algebraic cobordism theory $\gO$, which is an analog of the complex cobordism theory, in spite of the absence of notion of boundary. However, the definition is relatively complicated. Roughly speaking, if $X$ is a separated scheme of finite type over the ground field $k$, then we consider elements of the form $(f : Y \to X, \L_1, \ldots, \L_r)$ where $f$ is \proj, $Y$ is an irreducible \sm\ variety over $k$ and the sheaves $\L_i$ are line bundles over $Y$ (the order of $\L_i$ does not matter and the number $r$ of line bundles can be zero). The dimension of $(f : Y \to X, \L_1, \ldots, \L_r)$ is defined to be $\dim Y - r$. There is a natural notion of isomorphism on elements of this form. Denote the free abelian group generated by isomorphism classes of elements of this form by $\cat{Z}(X)$. Let $\underline{\gO}(X)$ be the quotient of $\cat{Z}(X)$ by the subgroup corresponding to imposing the axioms \textbf{(Dim)} and \textbf{(Sect)} (following the notations in \cite{LeMo}). The algebraic cobordism group $\gO(X)$ is defined to be the quotient of $\underline{\gO}(X) \otimes_{\Z} \lazard$, where $\lazard$ is the Lazard ring, by the $\lazard$-submodule corresponding to imposing the formal group law \textbf{(FGL)}.

This cobordism theory satisfies a number of basic properties, \textbf{(D1)-(D4)}, \textbf{(A1)-(A8)}, \textbf{(Dim), (Sect)} and \textbf{(FGL)} (following the notation in \cite{LeMo}). It also satisfies some more advanced properties, for example, the localization property and the homotopy invariance property. Moreover, the cobordism ring $\gO(\pt)$ will be isomorphic to the Lazard ring $\lazard$ when the characteristic of $k$ is 0, which is what we expect from the complex cobordism theory (see Corollary 1.2.11 and Theorem 4.3.7 in \cite{LeMo}).

One may wonder if it is possible to construct an algebraic cobordism theory via a more geometric approach. Suppose $X$ is a \sm\ variety over $k$. We may consider the abelian group $M(X)^+$ generated by isomorphism classes over $X$ of \proj\ \morp s $f : Y \to X$ where $Y$ is a \sm\ variety over $k$. The hope is that by imposing some reasonable relations, we will obtain an algebraic cobordism theory that also satisfies some previously mentioned properties. Such a construction was introduced by M. Levine and R. Pandharipande in \cite{LePa}. A relation called ``double point relation'' was introduced and it was shown that the theory $\go$ obtained by imposing this relation is canonically isomorphic to the theory $\gO$ under the assumption that the characteristic of $k$ is 0 (see Theorem 1 of \cite{LePa}). 

More precisely, let $\gph : Y \to X \x \P^1$ be a \proj\ \morp\ where $Y$ is an equidimensional \sm\ variety over $k$. Consider the fibers for the composition $Y \to X \x \P^1 \to \P^1$. Suppose the fiber $Y_{\gx}$ is a \sm\ divisor on $Y$ and the fiber $Y_0$ can be expressed as the union of two \sm\ divisors $A,B$ such that $A$ intersects $B$ transversely. Then, the double point relation is
\begin{eqnarray}
[Y_\gx \embed Y \to X] &=& [A \embed Y \to X] + [B \embed Y \to X] \nonumber\\
&& -\ [\P(\O \oplus \nbundle{A \cap B}{A}) \to A \cap B \embed Y \to X]. \nonumber
\end{eqnarray}

The objective of the current paper is to develop an algebraic cobordism theory of varieties with group action that assembles the theories of Levine-Morel and Levine-Pandharipande. For this, we go back to topology for inspiration. In topology, for a compact Lie group $G$, the concept of $G$-\equi\ bordism was first studied by Conner and Floyd (see \cite{periodicmap} or \cite{geovshomo}). In their approach, for a $G$-space $X$, we consider the set of maps $Y \to X$ where $Y$ is a stable almost complex $G$-manifold. Define the notion of $G$-bordism similarly to form the geometric unitary bordism group of $X$, denoted by $U^G_*(X)$. Another approach was pursued by Tom Dieck \cite{equibordism}. Let $\gx^G_n \to BU(n,G)$ be the universal unitary $n$-dimensional $G$-bundle and $MU(n,G)$ be its Thom space. Then, the homotopy theoretic unitary $G$-bordism group of $X$ is defined by
$$MU^G_{2k}(X) \defeq \limit{\stackrel{\longto}{V}}\ [\,S^V, \,MU(\dim_{\C} V - k,G) \wedge X\,]^G$$ 
and 
$$MU^G_{2k+1}(X) \defeq \limit{\stackrel{\longto}{V}}\ [\,S^V \wedge S^1, \,MU(\dim_{\C} V - k,G) \wedge X\,]^G$$ 
where $V$ runs through all unitary $G$-\repn s (see \cite{equibordism}). Inspired by the isomorphism between the principal $G$-bordism group over a point and the oriented bordism group $MSO(BG)$ (where $EG \to BG$ is the universal $G$-bundle) when $G$ is finite (see \cite{periodicmap}), there is also a third $G$-\equi\ bordism theory defined by the following equation :
$$MU^{G,h}_*(X) \defeq MU_*((X \x EG)/G).$$

In the case when $X$ is a point, there are some maps relating the three theories.
$$U^G_* \stackrel{a}{\to} MU^G_* \stackrel{b}{\to} MU^{G,h}_*$$
The map $a$ is given by the same Pontrjagin-Thom construction, but an inverse can not be constructed in the same manner due to the lack of transversality when there is group action. Indeed, the map $a$ is never surjective (unless the group $G$ is trivial) because there are non-trivial elements in the negative degrees of $MU^G_*$. However, the injectivity of the map $a$ was shown by Loffler and Comezana when $G$ is abelian (see \cite{equiunitarycobor} and \cite{equihomotopy}). On the other hand, when $G$ is abelian, the map $b$ identifies the $I$-adic completion of $MU^G_*$, where $I$ is the augmentation ideal, to $MU^{G,h}_*$ (see \cite{completionthmforMUspectra}). 

There are some computational results on different versions of \equi\ bordism ring. In \cite{generatorofbordismring}, Kosniowski gave a list of $G$-spaces which multiplicatively generate the geometric unitary bordism ring $U^G_*$ over $MU_*$ when $G$ is a cyclic group of prime order. When $G$ is an abelian compact Lie group, Sinha gave a list of elements and relations that describe the structure of the homotopy theoretic unitary bordism ring $MU^G_*$ as a $MU_*$-algebra (see \cite{computationofequibordring}). Since $MU^{G,h}_*$ can be identified with the $I$-adic completion of $MU^G_*$ when $G$ is abelian, we also obtain the structure of $MU^{G,h}_*$.

Following this pattern, we can expect to also have several different approaches to \equi\ algebraic cobordism theory. In order to define an analog of the homotopy theoretic bordism theory $MU^G$ in the algebraic geometry setup, one possible way is through Voevodsky's machinery of $\A^1$-homotopy theory (see \cite{A1homotopy}). A (non-\equi) algebraic cobordism theory defined this way is discussed in \cite{A1homotopy2}, but, to our knowledge, an \equi\ version of this has not yet been considered.

To define an analog of the theory $MU^{G,h}$, one can employ Totaro's approximation of $EG$. In \cite{EdGr}, Edidin and Graham successfully defined an equivariant Chow ring following this line of thought. For a given dimension $n$ algebraic space $X$ with $G$-action and for a fixed integer $i$, pick a $G$-\repn\ $V$ and an invariant open set $U$ inside such that $G$ acts freely on $U$ and the codimension of $V-U$ is larger than $n-i$. Then, $X \x U \to (X \x U)/G$ will be a principle $G$-bundle. Moreover, the Chow group $A_{i + \dim V - \dim G}((X \x U)/G)$ is indeed independent of the choice of the pair $(V,U)$. Hence, the equivariant Chow group $A^G_i(X)$ is defined to be $A_{i + \dim V - \dim G}((X \x U)/G)$. 

Unfortunately, since the independence of choice relies on the fact that the negative (cohomological) degrees of Chow groups are always zero, i.e. $A^i = 0$ whenever $i < 0$, equivariant algebraic cobordism theory can not be defined in the exact same manner. One approach is by considering a whole system of good pairs $\{(V,U)\}$ and define the equivariant algebraic cobordism group $\gO_G^i(X)$ to be the inverse limit of $\gO^i((X \x U)/G)$ (see \cite{HeLo} for more details). Another, possibly equivalent, approach was pursued by Krishna \cite{Kr}.

Aside from these two homotopical approaches, one can also define an \equi\ algebraic cobordism theory analogously to the geometric bordism theory $U^G$, namely by considering varieties with $G$-action and imposing the $G$-action also on the d\dpr. Suppose $G$ is an algebraic group over $k$ and $X$ is a \sm\ $G$-variety over $k$. This is what we do in this paper. We can consider the abelian group $M_G(X)^+$ generated by isomorphism classes of $G$-\equi\ \proj\ \morp\ $f : Y \to X$ where $Y$ is also a \sm\ $G$-variety. For a \morp\ $\gph : Y \to X \x \P^1$ where $Y$ is a \sm\ $G$-variety, $\P^1$ is equipped with the trivial action and $\gph$ is \proj\ and $G$-\equi\ satisfying the same conditions on the fibers $Y_{\gx}$ and $Y_0$ as before, we impose the exact same equation with all objects involved equipped with their naturally inherited $G$-actions. Then, all \morp s involved will also be naturally \equi.

For technical reasons, we focus on the case when the characteristic of $k$ is zero and $G$ is either a finite group or a connected reductive group. Observe that if there is a \proj, $G$-\equi\ \morp\ $Y \to X$ and \sm\ $G$-invariant divisors $Y_{\gx}$, $A$, $B$ on $Y$ satisfying the conditions in the d\dpr, then $Y_{\gx}$ is \equi ly linearly equivalent to $A + B$ and $Y_{\gx} + A + B$ is a \rsncd. Suppose we are given a \sm, $G$-invariant, very ample divisor $C$ on $Y$. Due to the lack of transversality in the \equi\ setting, the choice of the pairs of \sm\ $G$-invariant divisors $A$, $B$ such that $C \sim A + B$ and $A + B + C$ is a \rsncd\ may become seriously limited, if not impossible. To remedy this, it is preferable to impose a more general relation which we call g\gdpr. 
 
More precisely, suppose $X$, $Y$ are both \sm\ varieties with $G$-action and $\gph : Y \to X$ is an \equi\ \proj\ \morp. Assume there are \sm\ invariant divisors \gdprdivisors\ on $Y$ such that $A_1 + \cdots + A_n$ is \equi ly linearly equivalent to $B_1 + \cdots + B_m$ and $\gdprdivisorssum$ is a \rsncd. Then, the g\gdpr\ $GDPR(n,m)$ we will impose is of the form
\begin{eqnarray}
&&[A_1 \embed Y \to X] + [A_2 \embed Y \to X] + \cdots + [A_n \embed Y \to X] + \text{extra terms}  \nonumber\\
&=& [B_1 \embed Y \to X] + [B_2 \embed Y \to X] + \cdots + [B_m \embed Y \to X] + \text{extra terms} \nonumber
\end{eqnarray}
where the extra terms are of the form $[\P \to C \embed Y \to X]$ such that $C$ is the intersection of some of the divisors \gdprdivisors\ and $\P \to C$ is an \adtower\ (see subsection \ref{reductionoftowersubsection} for the definition). Denote the left hand side of the above equation by $L(\gph, \text{\gdprdivisors})$ and the right hand side by $R(\gph, \text{\gdprdivisors})$. Hence, we define the (geometric) \equi\ algebraic cobordism group, denoted by $\cobg{G}{}{X}$, to be the quotient of $M_G(X)^+$ by the abelian subgroup generated by 
$$L(\gph, \text{\gdprdivisors}) - R(\gph, \text{\gdprdivisors})$$ 
for all \equi\ \proj\ \morp s $\gph : Y \to X$ and all possible set of invariant divisors \gdprdivisors\ satisfying the conditions described above. We conjecture that the g\gdpr\ is indeed stronger than the d\dpr\ (See Remark \ref{redundantexceptionalobj} in the text).

An important observation is that the g\gdpr\ actually holds in the non-\equi\ theory $\go$. In other words, our \equi\ algebraic cobordism theory in the case when $G$ is trivial coincides with the non-\equi\ algebraic cobordism theory, i.e . $\cobg{\{1\}}{}{X} \cong \go(X)$ for all \sm\ varieties $X$. That means this theory $\cat{U}_G$ can be thought as a generalization of $\go$. In addition, although the g\gdpr\ may look tedious, it is actually easier to use because of the freedom of the number of divisors involved. 

Using this theory, we are able to define a ``fixed-point map'' which is similar to a well-known construction in topology. Recall the definition of the fixed point map in topology (see \cite{equibordism}). For simplicity, suppose $G$ is a finite group of prime order $p$. Then, there are exactly $p$ non-isomorphic irreducible complex $G$-\repn s. Denote them by $V_1, \ldots, V_p$. For a unitary $G$-manifold $M$, let $F$ be a component of the fixed point set $M^G$. The normal bundle of $F$ inside $M$ can be written as $\oplus_{i = 1}^p\ V_i \otimes N_i$ for some complex vector bundles $N_i$ over $F$ with no $G$-action. Compose the classifying map of $N_i$ with the natural map $BU(\text{rank of } N_i) \to BU$. We get a map $F \to BU$ which we will denote by $f_i$. Thus, the fixed point map
$$\gph : U^G_* \to \bigoplus_{i=1}^p MU_*(BU)$$
is given by sending $[M]$ to the sum of 
$$([f_1 : F \to BU], \ldots, [f_p : F \to BU])$$
over all components $F$. If we add up the elements $[f_i]$ and push them down to the bordism ring, we obtain a map
$$U^G_* \to \bigoplus_{i=1}^p MU_*(BU) \to MU_*(BU) \to MU_*$$
given by 
$$[M] \mapsto \sum_F ([f_1], \ldots, [f_p]) \mapsto \sum_{F,i} [f_i] \mapsto \sum_{F,i} [F] = p\,[M^G].$$

Assume the ground field $k$ has characteristic 0 as before. If the group $G$ is finite, then the fixed point locus of any \sm\ variety over $k$ is again \sm\ (Proposition 3.4 in \cite{Ed}). The same statement also holds when $G$ is reductive (by Proposition \ref{smoothfixedpointlocus}). So, for a \sm\ variety $X$, we have an abelian group homomorphism from $M_G(X)^+$ to $M(X^G)^+$ defined by sending $[Y \to X]$ to $[Y^G \to X^G]$, which we will also call fixed point map. One of our main results is the following Theorem (Corollary \ref{fixedpointmap2} in the text) which can be considered as an analog of the topological fixed point map.

\vtab

\textbf{Theorem 1.}
\textit{For any \sm\ $G$-variety $X$, sending $[Y \to X]$ to $[Y^G \to X^G]$ defines an abelian group homomorphism}
$$\fixptmap : \cobg{G}{}{X} \to \go(X^G).$$

\vtab

We also managed to find a set of generators for the \equi\ algebraic cobordism ring of the point $\pt$ when $G$ is a finite abelian group with exponent $e$ and $k$ contains a primitive $e$-th \rou. We can naturally embed the non-\equi\ algebraic cobordism ring 
$$\lazard \cong \go(\pt) \cong \cobg{\{1\}}{}{\pt}$$ 
inside the \equi\ algebraic cobordism ring $\cobg{G}{}{\pt}$ (by assigning trivial $G$-action) (see Corollary \ref{embedlazardcor}). This construction provides $\cobg{G}{}{\pt}$ with a $\lazard$-algebra structure. Then, the following Theorem describes a set of generators of $\cobg{G}{}{\pt}$ (see Theorem \ref{setofgenerator} for more detail).

\vtab

\textbf{Theorem 2.}
\textit{
Suppose $G$ is a finite abelian group with exponent $e$ and $k$ contains a primitive $e$-th \rou. Then, the \equi\ algebraic cobordims ring $\cobg{G}{}{\pt}$ is generated by the set of exceptional objects
$$\{E_{n, H, H'}\ |\ n \geq 0 \text{ and } G \supset H \supset H'\}$$
and the set of \adtower s over $\pt$ as a $\lazard$-algebra.}

\vtab

Here is the definition of the exceptional objects. For an integer $n \geq 0$ and a pair of subgroups $G \supset H \supset H'$, since $G$ is abelian, we can write
$$H/H' \cong H_1 \x \cdots \x H_a$$
where $H_i$ is a cyclic group of order $p_i^{m_i}$ for a prime $p_i.$ Define an $(H/H')$-action on 
$${\rm Proj}\ k[x_0, \ldots, x_n,v_1, \ldots, v_a]$$
by assigning $H_i$ to act faithfully on $\span{k}{v_i}$ and trivially on other generators, for all $1 \leq i \leq a$. Then, the exceptional object is defined as, with the natural $G$-action,
$$E_{n,H,H'} \defeq G/H \x {\rm Proj}\ k[x_0, \ldots, x_n,v_1, \ldots, v_a]\, / \, (v_1^{p_1^{m_1}} - g_1, \ldots, v_a^{p_a^{m_a}} - g_a)$$
where $g_i \in k[x_0, \ldots, x_n]$ is homogeneous with degree $p_i^{m_i}$ such that $E_{n,H,H'}$ is \sm\ with dimension $n$ ($[E_{n,H,H'}] \in \cobg{G}{}{\pt}$ is independent of the choice of $\{g_i\}$).

Let us now give a brief outline of this paper. In section 2, we state some basic notions and assumptions we will be using throughout the paper. In section 3, we give the precise definition of g\gdpr\ and also the definition of our \equi\ algebraic cobordism theory $\cat{U}_G$. We also show that the g\gdpr\ holds in the non-\equi\ theory $\go$. Then, a number of basic properties, namely \textbf{(D1)-(D4)} and \textbf{(A1)-(A8)} that does not involve the first Chern class operator, will be stated and proved. The last subsection will be devoted to the investigation of the case when the action is free. In this case, we show an isomorphism $\go(X/G) \cong \cobg{G}{}{X}$.

In section 4, we handle the (first) Chern class operator. We first define the notion of ``\nice'' \glin{G}\ invertible sheaves. Then, we define the Chern class operator $c(\L)$ for all such sheaves and prove the most important property of this operator : formal group law \textbf{(FGL)}. Next, we extend the definition to arbitrary \glin{G}\ invertible sheaves with stronger assumptions on $G$ and $k$ (in particular, $G$ is a finite abelian group). 

In section 5, we will first prove the rest of the list of basic properties, i.e. \textbf{(D1)-(D4)} and \textbf{(A1)-(A8)} that involve the Chern class operator. Then, we will show the properties \textbf{(Dim)} and \textbf{(FGL)}.

The whole section 6 will be devoted to proving the Theorem about the set of generators of the \equi\ cobordism ring $\cobg{G}{}{\pt}$ as a $\lazard$-algebra. The first subsection in section 6 will be dedicated to an interesting general technique which we will call splitting principle by blowing up along invariant smooth centers. Finally, in the last section, we will prove the well-definedness of the fixed point map, i.e. Theorem \ref{fixedpointmap}.

\medskip

\begin{center}
\textbf{Acknowledgments}
\end{center}

I would like to thank my advisor G. Pappas for useful conversations and helpful comments. This research was partially supported by NSF grants
DMS-1102208 and DMS-0802686.

\bigskip
\bigskip

\section{Notations and assumptions}
\label{notationsection}

Throughout this paper, we work over a field $k$ with characteristic $0$. We will denote by $\smcat$ the category of \sm\ \qproj\ schemes over $k$. We will often refer to this as varieties even though they do not have to be irreducible. The identity \morp\ will be denoted by $\id_X : X \to X$. The groups which act on varieties are either reductive connected groups or finite groups over $k$. So, they are always affine over $\pt$. We will often use the symbol $\gp_k$ to denote the structure \morp\ $X \to \pt$ and the symbol $\gp_i$ to denote the projection of $X_1 \x \cdots \x X_n$ onto its $i$-th component $X_i$. 

As in \cite{Mu}, an action of a group scheme $G$ on a variety $X$ is by definition a \morp\ $\gs : G \x X \to X$ such that

\begin{statementslist}
1. & The two \morp s $\gs \circ (\id_G \x \gs)$ and $\gs \circ (\gm \x \id_X)$ from $G \x G \x X$ to $X$ agree, where $\gm : G \x G \to G$ is the group law of $G$. \\
2. & The composition 

\begin{center}
$X \iso \pt \x X \stackrel{e \x \id_X}{\longto} G \x X \stackrel{\gs}{\longto} X$
\end{center}

\noindent is equal to $\id_X$, where $e$ is the identity \morp.
\end{statementslist}

For any $\ga \in G$ and $x \in X$, we will denote $\gs(\ga,x)$ by $\ga \cdot x$, or simply $\ga x$ if there is no confusion. We will say that the action is proper if the \morp\ $G \x X \to X \x X$ given by $(\ga,x) \mapsto (\ga \cdot x, x)$ is proper. Similarly, we will say the action is free if the above map is a closed immersion. This notion is stronger than the notion ``set-theoretically free''. According to Lemma 8 of \cite{EdGr}, set-theoretically free and proper implies free. In the case when $G$ is a finite group scheme, the two \morp s $\gs, \gp_2 : G \x X \to X$ are both proper. That means the \morp\ $G \x X \to X \x X$ above is proper. Hence, in this case, ``set-theoretically free'' is equivalent to free. Morphisms between $G$-varieties are always assumed to be $G$-\equi\ unless specified otherwise. We will denote by $\gsmcat{G}$ the category with objects in $\smcat$ with $G$ action and 
$$\morpset{\gsmcat{G}}{X}{Y} = \{f : X \to Y\ |\ f \text{ is $G$-\equi} \}.$$ 

If $X$ is in $\gsmcat{G}$ and $\cat{E}$ is a locally free coherent sheaf on $X$ with rank $r$, then a $G$-linearization of $\cat{E}$ is a collection of isomorphisms $\{\gph_{\ga} : \ga^* \cat{E}\iso\cat{E}\ |\ \ga \in G \}$ that satisfies the cocycle condition :
$$\gph_{\ga \gb} = \gph_{\gb} \circ (\gb^* \gph_{\ga}),$$
as isomorphisms from $(\ga \gb)^* \cat{E}$ to $\cat{E}$, for all $\ga, \gb \in G.$ There is a natural definition of isomorphism associated to it. The set of isomorphism classes of invertible sheaves on $X$ with a $G$-linearization will be denoted by $\picard{G}{X}$.

Recall the definition of transversality from \cite{LePa}. For objects $A,B,C \in \smcat$ and \morp s $f : A \to C$ and $g : B \to C$, we say $f,g$ are transverse if $A \x_C\ B$ is \sm\ and for all irreducible components $A' \subset A$ and $B' \subset B$ such that $f(A'), g(B')$ are both contained in the same irreducible component $C' \subset C$, we have either
$$\dim A' \x_{C'}\ B' = \dim A' + \dim B' - \dim C'$$
or $A' \x_{C'}\ B' = \emptyset$. If $A,B$ are both subschemes of $C$, we say that $A,B$ are transverse if the inclusion \morp s\ are transverse. If $f : A \to C$ and $x$ is point in $C$, we say that $x$ is a regular value of $f$ if the inclusion \morp\ $x \embed C$ and $f$ are transverse.

Also recall the definition of principal $G$-bundle from \cite{EdGr}. A \morp\ $f : X \to Y$ is called a principal $G$-bundle if $G$ acts on $X$, the \morp\ $f$ is flat, surjective, $G$-\equi\ for the trivial $G$-action on $Y$ and the \morp
$$G \x X \to X \x_Y X,$$ 
defined by $(\ga,x) \mapsto (\ga \cdot x, x)$, is an isomorphism.

If $X$, $Y$ are two objects in $\gsmcat{G}$, then $X \x Y$ is considered to be in $\gsmcat{G}$ with $G$ acting diagonally.

For a \morp\ $f : X \to Y$ and a point $y \in Y$, we denote the fiber product $\spec{k(y)} \x_Y X$ by $f^{-1}(y)$ where $k(y)$ is the residue field of $y$ and $\spec{k(y)} \to Y$ is the \morp\ corresponding to $y$. Similarly, if $Z$ is a subscheme of $Y$, then we denote $Z \x_{Y} X$ by $f^{-1}(Z)$. If $A, B$ are both subschemes of $X$, then we denote $A \x_{X} B$ by $A \cap B$.

An object $Y \in \gsmcat{G}$ is called \girred\ if there exists an irreducible component $Y'$ of $Y$ such that $G \cdot Y' = Y$.

For a locally free sheaf $\cat{E}$ of rank $r$ over a $k$-scheme $X$, the corresponding vector bundle $E$ over $X$ will be given by $$E \defeq \spec{{\rm Sym}\ \dual{\cat{E}}}.$$
The same applies on the case when $X$ is a $G$-scheme over $k$ and $\cat{E}$ is \glin{G}.

%\newpage

\bigskip

\bigskip

\section{Geometric \equi\ algebraic cobordism $\cob{G}$}
\label{basicpropertysection}

\subsection{Preliminaries}

Before digging into the \equi\ algebraic cobordism theory, we need to understand more about $G$-invariant divisors and \glin{G}\ invertible sheaves.

\vtab

Weil Divisors :

Let $X$ be a \girred\ object in $\gsmcat{G}$. A $G$-invariant, \girred\ reduced closed subscheme $D \subset X$ with codimension 1 will be called a prime $G$-invariant Weil divisor. A $G$-invariant Weil divisor is a finite $\Z$-linear combination of prime divisors, i.e. $D = \sum n_i D_i$. A $G$-invariant Weil divisor $D$ is called effective if $n_i$ are all non-negative. Let $\cat{K}$ be the sheaf of total quotient rings of $\O_X$, which has its natural $G$-action. We say that two $G$-invariant Weil divisors $D$, $D'$ are $G$-\equi ly linearly equivalent, denoted by $D \sim D'$, if there is an element $f \in {\rm H}^0(X, \cat{K}^*)^G$ such that $D - D' = \divisor{f}$ where $\divisor{f}$ is defined in the usual way.

\vtab

Cartier Divisors :

Similar to the definition of Cartier divisors in Ch II, section 6 in \cite{Ha}, a $G$-invariant Cartier divisor is an element in ${\rm H}^0(X, \cat{K}^*/\O^*)^G$. We say two $G$-invariant Cartier divisors $D$, $D'$ are $G$-\equi ly linearly equivalent if $D - D'$ is in the image of 
$${\rm H}^0(X, \cat{K}^*)^G \to {\rm H}^0(X, \cat{K}^*/\O^*)^G.$$ 
As usual, we will represent a $G$-invariant Cartier divisor by $\{(U_i, f_i)\}$ where $\{U_i\}$ is an open cover of $X$ and $f_i \in {\rm H}^0(U_i, \cat{K}^*)$. The (left) $G$-action on the sheaf $\cat{K}$ (or the sheaf $\O$) is given explicitly by 
$$(\ga \cdot f)(x) = f(\ga^{-1} \cdot x)$$ 
for any $f \in \cat{K}$ (or in $\O$) and $\ga \in G$. Then, the Cartier divisor $D$ being $G$-invariant implies 
$$\{(U_i, f_i)\} = \{(\ga \cdot U_i, \ga \cdot f_i)\}$$ 
as elements in ${\rm H}^0(X, \cat{K}^*/\O^*)$ for all $\ga \in G$. In other words, $(\ga \cdot f_i) / f_j$ is a unit in $\O_{(\ga \cdot U_i) \cap U_j}$ for all $i,j$. Since $X$ is \sm, we have a one-to-one correspondence between the set of $G$-invariant Weil divisors and the set of $G$-invariant Cartier divisors by the same construction as in \cite{Ha}. Moreover, the notion of $G$-\equi ly linearly equivalent is also preserved. Hence, from now on, we will use the two notions interchangeably. Furthermore, divisors are always assumed to be $G$-invariant unless specified otherwise and linear equivalence means $G$-\equi\ linear equivalence. 

\vtab

\glin{G}\ invertible sheaves :

For a given $G$-invariant divisor $D$ on a \sm\ $G$-variety $X$, we can construct a \glin{G}\ invertible sheaf naturally. We will denote it by $\O_X(D)$. Here is the construction.

The underlying invertible sheaf structure is given by the natural definition as in Ch II, section 6 in \cite{Ha} : if $D$ is represented by $\{(U_i, f_i)\}$, then we define $\O_X(D)$ by the following equation : 
$$\O_X(D)|_{U_i} \defeq \O_{U_i} f_i^{-1}$$
for all $i$.

The $G$-linearization of $\O_X(D)$ can be defined as the following. Consider the case when $D$ is a prime $G$-invariant divisor. Then, it defines an ideal sheaf $\cat{I}$ which is naturally \glin{G}. Then, the natural isomorphism $\O_X(-D) \cong \cat{I}$ induces a $G$-linearization on $\O_X(-D)$. Hence, we can define the $G$-linearization by taking the dual, namely, $\O_X(D) \defeq \dual{\O_X(-D)}$. In general, if $D = \sum n_i D_i$ for some prime $G$-invariant divisors $D_i$, then we define $\O_X(D) \defeq \otimes\, O_X(D_i)^{\otimes n_i}$.

The $G$-linearization structure on $\O_X(D)$ can be explicitly given as the following. For a given $\ga \in G$, we will define an isomorphism $\gph_{\ga} : \ga^* \O(D) \to \O(D)$. Let us consider the restriction on $U_i$, the domain becomes
\begin{eqnarray}
(\ga^* \O(D)) |_{U_i} & = & \ga^* (\O(D)|_{\ga U_i}) \nonumber\\
         & = & \ga^* (\O_{U_j \cap \ga U_i} f_j^{-1}) \nonumber\\
         && \text{(further restricted on $U_j \cap \ga U_i$)} \nonumber\\
         & = & \O_{U_i \cap \ga^{-1} U_j}\, \ga^{-1} \cdot f_j^{-1}. \nonumber
\end{eqnarray}
On the other hand, the codomain becomes $\O_{U_i \cap \ga^{-1} U_j}\, f_i^{-1}$ when restricted on $U_i \cap \ga^{-1} U_j$. Then, we define 
$$\gph_{\ga}|_{U_i \cap \ga^{-1} U_j} : \O\, \ga^{-1} \cdot f_j^{-1} \to \O f_i^{-1}$$
by sending $\ga^{-1} \cdot f_j^{-1}$ to $(f_i\, / \,(\ga^{-1} \cdot f_j))\, f_i^{-1}$. Since $\gph_{\ga}|_{U_i \cap \ga^{-1} U_j}$ is an identity map, $\gph_{\ga}$ is well-defined and is an isomorphism.

We need to check the cocycle condition 
$$\gph_{\ga \gb} = \gph_{\gb} \circ (\gb^* \gph_{\ga}) : (\ga \gb)^* \O(D) \to \O(D).$$ 
For simplicity, we will denote $\O_X$ (or other base) by simply $\O$. Notice that, by the above definition, $\gph_{\gb} : \gb^* \O(D) \to \O(D)$ corresponds to $\O\, \gb^{-1} \cdot f_j^{-1} \to \O f_i^{-1}$ and $\gph_{\ga} : \ga^* \O(D) \to \O(D)$ corresponds to $\O\, \ga^{-1} \cdot f_k^{-1} \to \O f_j^{-1}.$ So, the \morp\ $\gb^* \gph_{\ga} : \gb^* \ga^* \O(D) \to \gb^* \O(D)$ corresponds to $\O\, \gb^{-1} \ga^{-1} \cdot f_k^{-1} \to \O\, \gb^{-1} \cdot f_j^{-1}.$ On the other hand, $\gph_{\ga \gb} : (\ga \gb)^* \O(D) \to \O(D)$ corresponds to $\O\, \gb^{-1} \ga^{-1} \cdot f_k^{-1} \to \O f_i^{-1}$. Thus, the domains and codomains of $\gph_{\ga \gb}$ and $\gph_{\gb} \circ (\gb^* \gph_{\ga})$ are represented by the same generators and all the \morp s are identities. Hence, they commute.

It remains to check its independence of the choice of representations $\{(U_i,f_i)\}$ of the Cartier divisor. In other words, if $D$ is represented by $\{(U_i,f_i)\}$ where $f_i \in {\rm H}^0(U_i, \O^*)$, then the \glin{G}\ invertible sheaf it defined will be $G$-\equi ly isomorphic to the structure sheaf. To see this, we define a \morp\ from $\O(D)$ to $\O$ by patching the \morp s $\O f_i^{-1} \to \O$ in which we send $f_i^{-1}$ to $f_i^{-1}$. Then, it is a well-defined isomorphism. The commutativity of the following diagram implies that this \morp\ is $G$-\equi.
\squarediagram{\O\, \ga^{-1} \cdot f_j^{-1}}{\O}{\O f_i^{-1}}{\O}

This natural construction also takes $G$-\equi ly linearly equivalent divisors to isomorphic \glin{G}\ invertible sheaves, i.e. if $f$ is in ${\rm H}^0(X,\cat{K}^*)^G$, then $\O f^{-1}\iso\O$ by sending $f^{-1}$ to $1$.

Unfortunately, we do not have the one-to-one correspondence between the set of $G$-invariant divisor classes and the set of isomorphism classes of \glin{G} invertible sheaves. Here is a simple reason. If the $G$-action on $X$ is trivial, then the $G$-action on any $G$-invariant divisor will be trivial too. Hence, the $G$-action on the line bundle corresponding to $\O(D)$ must be trivial. But, there are certainly $G$-\equi\ line bundles over $X$ with non-trivial fiberwise $G$-actions.

\vtab

The following are some basic properties of $G$-\inv\ divisors.

\begin{prop} Suppose $X,Y$ are objects in $\gsmcat{G}$.

\begin{statementslist}
{\rm (1)} & If $f : X \to Y$ is a \morp\ in \gsmcat{G}\ and $D$ is a $G$-\inv\ divisor on $Y$ such that $f^*D$ is a $G$-\inv\ divisor on $X$, then $f^* \O(D) \cong \O(f^*D)$. \\
{\rm (2)} & If $D$ is a $G$-\inv\ divisor on $X$ and $Z$ is a closed subscheme of $X$ such that $Z \cap {\rm Supp} D$ is empty, then $\O_X(D)|_Z \cong \O_Z$. \\
{\rm (3)} & If $D$ is a $G$-\inv\ divisor on $X$, then $\O_X(D) \cong \O_X$ if and only if $D \sim 0$.
\end{statementslist}
\end{prop}

\begin{proof}
{\rm (1)}\tab Suppose $D$ is represented by $\{(U_i, g_i)\}$. Then, the $G$-\inv\ divisor $f^*D$ can be represented by $\{(f^{-1}(U_i), f^* g_i)\}$. Thus, 
$$(f^* \O(D)) |_{f^{-1}(U_i)} = f^* (\O_{U_i} g_i^{-1}) = \O_{f^{-1}(U_i)}\ f^* g_i^{-1}.$$
On the other hand, $\O(f^*D) |_{f^{-1}(U_i)} = \O_{f^{-1}(U_i)}\, f^* g_i^{-1}$. So they are isomorphic. The compatibility of the $G$-action is easy to check.

{\rm (2)}\tab Suppose $D$ is represented by $\{(U_i, g_i)\}$ and $i : Z \embed X$ is the closed immersion. Since $Z \cap {\rm Supp} D = \emptyset$, by refinement, we can assume $U_i$ either has empty intersection with $Z$ or, otherwise, $g_i$ is a unit in $\O_{U_i}$. Thus, $i^*D$ is a $G$-\inv\ divisor on $Z$ and can be represented by $\{(U_i \cap Z, g_i|_Z)\}$, or simply $\{(Z,1)\}$ by the independence of representation. That means 
$$\O_X(D)|_Z \cong \O_Z(i^*D) \cong \O_Z.$$

{\rm (3)}\tab As mentioned before, if $D$ and $D'$ are $G$-\equi ly linearly equivalent, then they define the same \glin{G}\ invertible sheaf, i.e. $\O_X(D) \cong \O_X(D')$. So, it is enough to show if $\O_X(D) \cong \O_X$, then $D \sim 0$. Suppose $D$ is represented by $\{(U_i,g_i)\}$. Then, the isomorphism $\O_X \to \O_X(D)$ over $U_i$ is given by sending 1 to $a_i g_i^{-1}$ for some $a_i \in \O_{U_i}^*$. The fact that the isomorphism is globally defined implies that $a_i (g_j / g_i) = a_j$. Thus, 
$$h \defeq \frac{a_i}{g_i} = \frac{a_j}{g_j} \in {\rm H}^0(X, \cat{K}^*).$$
Since $a_i g_i^{-1}$ corresponds to 1, the $G$-action on $h$ is trivial. Hence, the two $G$-\inv\ divisors $\{(U_i,g_i)\}$ and $\{(U_i,a_i)\}$ are $G$-\equi ly linearly equivalent. The result then follows from the fact that $a_i \in \O_{U_i}^*$.
\end{proof}

\begin{rmk}
\label{divisopicrmk}
\rm{
By property (3), we can consider the set of $G$-\inv\ divisor classes on $X$ as a natural subgroup of $\picard{G}{X}$.
}
\end{rmk}

We will also use the following fact about \proj\ bundles from time to time.

\begin{prop}
\label{isoprojbundle}
For an object $X \in \gsmcat{G}$, suppose $\L$ is in $\picard{G}{X}$ and $\cat{E}$ is a \glin{G}\ locally free sheaf of rank $r$ over $X$. Then $\P(\cat{E})$ and $\P(\cat{E} \otimes \L)$ are naturally isomorphic as $G$-\equi\ \proj\ bundles over $X$.
\end{prop}

\begin{proof}
First of all, we define a \morp\ from $\P(\cat{E} \otimes \L)$ to $\P(\cat{E})$ without considering the group action. Let $\{U_i\}$ be an open cover of $X$ such that $\cat{E}|_{U_i}$ is trivial and $\L|_{U_i} \cong \O_{U_i} l_i$. Then, we define a \morp\ 
$$\gg : \cat{E}|_{U_i} \to \cat{E} \otimes \L|_{U_i}$$ 
by $e \mapsto e \otimes l_i$. This induces a \morp\ 
$$f|_{U_i} : \P(\cat{E} \otimes \L|_{U_i}) = {\rm Proj}\ {\rm Sym}\ \cat{E} \otimes \L|_{U_i} \to {\rm Proj}\ {\rm Sym}\ \cat{E}|_{U_i} = \P(\cat{E}|_{U_i}).$$ 
We claim that $\{f|_{U_i}\}$ will patch together to form a \morp\ from $\P(\cat{E} \otimes \L)$ to $\P(\cat{E})$ and it will be an isomorphism of \proj\ bundles over $X$.

Let $\gs_{\cat{E}}$, $\gs_{\L}$ be the transition functions of $\cat{E}$, $\L$ \resp\ from $U_i$ to $U_j$. Then, we have $\gs_{\L}(l_i) = a\, l_j$ for some $a \in \O_{U_j}^*$ and the transition function for $\cat{E} \otimes \L$ will be $\gs_{\cat{E}} \otimes \gs_{\L}$. Then, 
\begin{eqnarray}
(\gs_{\cat{E}} \otimes \gs_{\L}) \circ \gg(e) & = & (\gs_{\cat{E}} \otimes \gs_{\L})(e \otimes l_i) \nonumber\\
& = & \gs_{\cat{E}}(e) \otimes \gs_{\L}(l_i) \nonumber\\
& = & \gs_{\cat{E}}(e) \otimes a\, l_j \nonumber\\
& = & a\, (\gs_{\cat{E}}(e) \otimes l_j). \nonumber
\end{eqnarray}
On the other hand,
$$\gg \circ \gs_{\cat{E}}(e) = \gs_{\cat{E}}(e) \otimes l_j.$$
If we consider $\gs_{\cat{E}}(e) \otimes l_j$ and $a\, (\gs_{\cat{E}}(e) \otimes l_j)$ as elements in ${\rm Sym}\ \cat{E} \otimes \L$, then they agree, in homogeneous coordinates. Hence, $\{f|_{U_i}\}$ patch together to form a \morp\ $f$. Moreover, it is obviously an isomorphism and a \proj\ bundle \morp.

It remains to check if $f$ is $G$-\equi. The $G$-linearization on $\L$ is described by a set of iso\morp s $\{\gph_{\L,\ga} : \ga^* \L\iso\L\}$. When restricted on $U_i \cap \ga^{-1}U_j$, $\gph_{\L,\ga}$ defines an isomorphism from $\O l_j$ to $\O l_i$. So, $\gph_{\L,\ga}(l_j) = b_{\ga} l_i$ for some $b_{\ga} \in \O_{U_i \cap \ga^{-1}U_j}^*$. Similarly, if $\{\gph_{\cat{E},\ga}\}$ and $\{\gph_{\cat{E} \otimes \L,\ga}\}$ defines the $G$-linearizations on $\cat{E}$ and $\cat{E} \otimes \L$ respectively, then
$$\gg \circ \gph_{\cat{E},\ga}(e) = \gph_{\cat{E},\ga}(e) \otimes l_i.$$
On the other hand, 
\begin{eqnarray}
\gph_{\cat{E} \otimes \L,\ga} \circ \gg(e) & = & \gph_{\cat{E} \otimes \L,\ga} (e \otimes l_j) \nonumber\\
& = & \gph_{\cat{E},\ga}(e) \otimes \gph_{\L,\ga}(l_j) \nonumber\\
& = & \gph_{\cat{E},\ga}(e) \otimes b_{\ga} l_i \nonumber\\
& = & b_{\ga} (\gph_{\cat{E},\ga}(e) \otimes l_i). \nonumber
\end{eqnarray}
For the same reason, they agree in homogeneous coordinates and hence, $f$ is $G$-\equi.
\end{proof}

\vtab

\subsection{Generalized double point relation}
\label{gdprsubsection}

In \cite{LePa} (Definition 0.2), $\go_*(X)$ is defined as the quotient of the free abelian group generated by symbols $[f : Y \to X]$ where $Y$ is an object in $\smcat$ and $f$ is a \proj\ \morp, by an equivalence relation called double point relation. More precisely, suppose $Y \in \smcat$ is equidimensional and there is a \proj\ \morp\ $\gph : Y \to X \x \P^1$ such that a closed point $\gx \in \P^1$ is a regular value of $\gp_2 \circ \gph$ (In other words, $Y_{\gx} \defeq (\gp_2 \circ \gph)^{-1}(\gx)$ is a \sm\ divisor on $Y$), the fiber $Y_0 = A \cup B$ for some \sm\ divisors $A$, $B$ and $A + B$ is a \rsncd. Then, the double point relation is
\begin{eqnarray}
[Y_{\gx} \embed Y \to X] &=& [A \embed Y \to X] + [B \embed Y \to X] \nonumber\\
&& -\ [\P(\O \oplus \O(A)) \to A \cap B \embed Y \to X]. \nonumber
\end{eqnarray}

In addition, in section 5.2 in \cite{LePa}, a relation called e\edpr\ is also introduced. Suppose $Y \in \smcat$ is equidimensional and there is a \proj\ \morp\ $\gph : Y \to X$. In addition, we have divisors $A$, $B$, $C$ on $Y$ such that $A + B + C$ is a \rsncd\ and $C \sim A + B$. Then, the e\edpr\ is
\begin{eqnarray}
[C \embed Y \to X] &=& [A \embed Y \to X] + [B \embed Y \to X] \nonumber\\
&& -\ [\P(\O \oplus \O(A)) \to A \cap B \embed Y \to X] \nonumber\\
&& +\ [\P(\O \oplus \O(1)) \to \P(\O(-B) \oplus \O(-C)) \to A \cap B \cap C \embed Y \to X] \nonumber\\
&& -\ [\P(\O \oplus \O(-B) \oplus \O(-C)) \to A \cap B \cap C \embed Y \to X]. \nonumber
\end{eqnarray}

On one hand, if we assume $C$ does not intersect $A \cup B$, then this is the same as the d\dpr. On the other hand, Lemma 5.2 in \cite{LePa} showed that the e\edpr\ holds in the theory $\go$ defined by the d\dpr. In other words, the theory with double point relation imposed and the one with extended double point relation imposed are equivalent. One may then expect the existence of similar formulas when $Y_0 = A_1 \cup A_2 \cup A_3$ in the d\dpr\ setup, or when $B \sim A_1 + A_2 + A_3$ in the e\edpr\ setup. Indeed, it is possible to write a formula for arbitrary number of divisors. For induction purpose, we will consider the e\edpr\ setup.

More precisely, suppose $X$ is a separated scheme of finite type over $k$ and $\gph : Y \to X$ is a \proj\ \morp\ with $Y \in \smcat$ such that $Y$ is \equidim. Moreover, suppose there are divisors $A_1, \ldots, A_n,$ $B_1, \ldots, B_m$ on $Y$ such that 
$$\gdprdivisorsequi$$ 
and $\gdprdivisorssum$ is a \rsncd. Then, we expect a formula of the form
$$[A_1 \to X] + \cdots + [A_n \to X] + \text{extra terms} = [B_1 \to X] + \cdots + [B_m \to X] + \text{extra terms}.$$
We will give such a formula inductively. For this purpose, we will consider the following.

\begin{defn}
{\rm
Define a polynomial ring over $\Z$ :
$$\cat{R} \defeq \Z[\{X_i,Y_j,U^p_k,V^q_l\}]$$ 
where $i,j,k,l \geq 1$ and $1 \leq p,q \leq 3$.
}\end{defn}

Then, we define some elements in $\cat{R}$ inductively.

\begin{defn}
{\rm
Let $E_1^X$, $F_1^X \defeq 0$. For $n \geq 2$, define
\begin{eqnarray}
E_n^X &\defeq& E_{n-1}^X - (X_1 + \cdots + X_{n-1} + E_{n-1}^X) X_n U^1_{n-1} - X_n F_{n-1}^X \nonumber\\
F_n^X &\defeq& F_{n-1}^X + (X_1 + \cdots + X_{n-1} + E_{n-1}^X) X_n (U^2_n - U^3_n). \nonumber
\end{eqnarray}
Similarly, define $E_n^Y$, $F_n^Y$ by replacing $X$ by $Y$ and $U$ by $V$ in $E_n^X$, $F_n^X$ respectively, namely,
\begin{eqnarray}
E_1^Y,\ F_1^Y &\defeq& 0 \nonumber\\
E_n^Y &\defeq& E_{n-1}^Y - (Y_1 + \cdots + Y_{n-1} + E_{n-1}^Y) Y_n V^1_{n-1} - Y_n F_{n-1}^Y \nonumber\\
F_n^Y &\defeq& F_{n-1}^Y + (Y_1 + \cdots + Y_{n-1} + E_{n-1}^Y) Y_n (V^2_n - V^3_n) \nonumber
\end{eqnarray}
for $n \geq 2$. Also, for $n,m \geq 1$, define the elements $G^X_{n,m}$ as the following :
$$G^X_{n,m} \defeq X_1 + \cdots + X_n + E^X_n + (Y_1 + \cdots + Y_m)F^X_n + E^Y_mF^X_n.$$
Finally, define $G^Y_{n,m}$ by interchanging $X$ and $Y$ in $G^X_{n,m}$, namely,
$$G^Y_{n,m} \defeq Y_1 + \cdots + Y_n + E^Y_n + (X_1 + \cdots + X_m)F^Y_n + E^X_mF^Y_n.$$
}\end{defn}

For a \proj\ \morp\ $\gph : Y \to X$, such that $Y$ is \equidim, and divisors \gdprdivisors\ on $Y$ such that $\gdprdivisorsequi$ and $\gdprdivisorssum$ is a \rsncd, we define an abelian group homomorphism 
$$\cat{G} : \cat{R} \to \go(X)$$
as the following.

First of all, any term with $X_i$ such that $i > n$, or $Y_j$ such that $j > m$, or $U^p_k$ such that $k > n$, or $V^q_l$ such that $l > m$ will be sent to 0. (In fact, they have no effect on our formula. This part of definition is just for formality.) Then, we send
\begin{eqnarray}
1 & \mapsto & [Y \to X] \nonumber\\
X_i & \mapsto & [A_i \to Y \to X] \nonumber\\
Y_j & \mapsto & [B_j \to Y \to X] \nonumber\\
U^1_k & \mapsto & [\P(\O \oplus \O(D)) \to Y \to X] \nonumber\\
&& \text{where $D \defeq A_1 + \cdots + A_k$. Denote it by $[P^1_k \to X]$ for simplicity.} \nonumber\\
U^2_k &\mapsto& [\P(\O \oplus \O(1)) \to \P(\O(-A_k) \oplus \O(-D)) \to Y \to X] \nonumber\\
&& \text{Denote it by $[P^2_k \to X]$.} \nonumber\\
U^3_k &\mapsto& [\P(\O \oplus \O(-A_k) \oplus \O(-D)) \to Y \to X] \nonumber\\
&& \text{Denote it by $[P^3_k \to X]$.} \nonumber\\
V^q_l &\mapsto& [Q^q_l \to Y \to X] \nonumber\\
&& \text{where $Q^q_l$ is defined in the same manner with divisors $B_l$ and } \nonumber\\
&& \text{$D = B_1 + \cdots + B_l$ instead.} \nonumber
\end{eqnarray}
Finally, for a general term $X_i \cdots Y_j \cdots U^p_k \cdots V^q_l \cdots$, we send it to 
$$[A_i \x_Y \cdots \x_Y B_j \x_Y \cdots \x_Y P^p_k \x_Y \cdots \x_Y Q^q_l \x_Y \cdots \to Y \to X].$$
Then, the g\gdpr\ $GDPR(n,m)$ is the equality :
$$\cat{G}(G^X_{n,m}) = \cat{G}(G^Y_{m,n}).$$

\begin{rmk}
\label{gdprremark}
\rm{
In order for the homo\morp\ $\cat{G}$ to be well-defined, we need to check that, in general, the \morp
$$A_i \x_Y \cdots \x_Y B_j \x_Y \cdots \x_Y P^p_k \x_Y \cdots \x_Y Q^q_l \x_Y \cdots \to Y \to X$$
is \proj\ and its domain is \sm. Notice that 
$$\cat{G}(X_i^n) = [A_i \x_Y \cdots \x_Y A_i \to X] = [A_i \to X],$$ 
which is \proj\ and $A_i$ is \sm. Since $\gdprdivisorssum$ is a \rsncd, the same is true for $\cat{G}(\text{term with $X_i$, $Y_j$ only})$. In addition, the \morp s $P^p_k \to Y$ and $Q^q_l \to Y$ are both \proj\ and \sm. That means $\cat{G} : \cat{R} \to \go(X)$ is well-defined.

Observe that for any $n,m \geq 1$, the terms in $G^X_{n,m}$ or $G^Y_{m,n}$ are always of the form
$$X_i \cdots Y_j \cdots U^p_k \cdots V^q_l \cdots$$
where the powers for $X_i$, $Y_j$ are either 0 or 1. In other words, self intersection will never happen in any $GDPR(n,m)$. Moreover, $1 \leq i, k \leq n$ and $1 \leq j, l \leq m$. In addition, $\cat{G}(G^X_{n,m})$, $\cat{G}(G^Y_{m,n})$ are both in $\go_{\dim Y-1}(X)$.

The g\gdpr\ is indeed a generalization of d\dpr\ and e\edpr. For example, if we apply the definition on the setup $\id_X : Y = X \to X$ with $A_1 + A_2 \sim B_1$, we get
\begin{eqnarray}
E^X_2 &=& -\ X_1X_2U^1_1 \nonumber\\
F^X_2 &=& X_1X_2(U^2_2-U^3_2) \nonumber\\
G^X_{2,1} &=& X_1 + X_2 + E^X_2 + Y_1F^X_2 \nonumber\\
&=& X_1 + X_2 - X_1X_2U^1_1 + Y_1X_1X_2(U^2_2-U^3_2) \nonumber\\
G^Y_{1,2} &=& Y_1 \nonumber
\end{eqnarray}
Hence, the $GDPR(2,1)$ is the equality
\begin{eqnarray}
&&[A_1 \embed X] + [A_2 \embed X] \nonumber\\
&& -\ [\P(\O \oplus \O(A_1)) \to A_1 \cap A_2 \embed X] \nonumber\\
&& +\ [\P(\O \oplus \O(1)) \to \P(\O(-A_2) \oplus \O(-A_1 - A_2)) \to B_1 \cap A_1 \cap A_2 \embed X] \nonumber\\
&& -\ [\P(\O \oplus \O(-A_2) \oplus \O(-A_1 - A_2)) \to B_1 \cap A_1 \cap A_2 \embed X] \nonumber\\
&=& [B_1 \embed X], \nonumber
\end{eqnarray}
which is exactly the e\edpr\ as Lemma 5.2 in \cite{LePa}. If we further assume that $B_1$ is disjoint from $A_1$, $A_2$, then we get
$$[A_1 \embed X] + [A_2 \embed X] - [\P(\O \oplus \O(A_1)) \to A_1 \cap A_2 \embed X] = [B_1 \embed X],$$
which is the d\dpr\ in \cite{LePa} (because $\nbundle{A_1 \cap A_2}{A_2} \cong \O_{A_1 \cap A_2}(A_1)$).
}
\end{rmk}

\vtab

Our first goal is to prove $GDPR(n,m)$ holds in the theory $\go$. In other words, we will show that imposing the g\gdpr\ is equivalent to imposing the d\dpr.

To be more precise, suppose $X$ is a separated scheme of finite type over $k$ and $\gph : Y \to X$ is a \proj\ \morp\ such that $Y \in \smcat$ is \equidim. Moreover, suppose there are divisors \gdprdivisors\ on $Y$ such that $\gdprdivisorsequi$ and $\gdprdivisorssum$ is a \rsncd. We want to show 
$$\cat{G}(G^X_{n,m}) = \cat{G}(G^Y_{m,n})$$ 
where $\cat{G} : \cat{R} \to \go(X)$ is the corresponding group homo\morp.

First of all, observe that $\cat{G}(G^X_{n,m}) = \gph_*\circ \cat{G}'(G^X_{n,m})$ where $\cat{G}'$ is the map defined by the setup $\id : Y \to Y$ with the same set of divisors on $Y$. Similarly, $\cat{G}(G^Y_{m,n}) = \gph_*\circ \cat{G}'(G^Y_{m,n})$. Hence, we reduce to the case when $\gph = \id_X$. In particular, we may assume $X$ is in $\smcat$ and is \equidim.

Recall that in \cite{LePa}, suppose $X$ is a separated scheme of finite type over $k$ and $\L$ is an invertible sheaf over $X$. There is a corresponding operator $\tilde{c}_1(\L) \in \Endo{\go(X)}$ which is called the first Chern class operator. For simplicity, we will denote it by $c(\L)$ and call it Chern class operator for the rest of this paper.

We are going to prove $GDPR(n,m)$ by induction. For this purpose, we need to modify the definition of $\cat{G}$. Suppose $X \in \smcat$ is \equidim\ and there are divisors \gdprdivisors\ on $X$ such that $\gdprdivisorsequi$. Then, we define a ring homomorphism $\cat{G} : \cat{R} \to \Endo{\go(X)}$ by
\begin{eqnarray}
X_i & \mapsto & c(\O(A_i)) \nonumber\\
Y_j & \mapsto & c(\O(B_j)) \nonumber\\
U^a_k & \mapsto & {p_a}_* p_a^* \nonumber\\
&& \text{where $p_a : P^a_k \to X$} \nonumber\\
V^b_l & \mapsto & {q_b}_* q_b^* \nonumber\\
&& \text{where $q_b : Q^b_l \to X$} \nonumber
\end{eqnarray}
if $1 \leq i,k \leq n$ and $1 \leq j,l \leq m$ (The \morp s ${p_a}_*$, $p_a^*$, ${q_b}_*$, $q_b^*$ are all well-defined because $p_a$, $q_b$ are both \sm\ and \proj.). Otherwise, send them to zero.

For well-definedness of $\cat{G}$, we need to check the commutativity of some endomorphisms. Axiom \textbf{(A5)} in $\go$ implies that $c(\L)c(\L') = c(\L')c(\L)$. In addition, for $p : P^a_k \to X$, we have
$$c(\L)p_* p^* = p_* c(p^*\L) p^* = p_*  p^* c(\L)$$ 
and same for $q$. For the commutativity between $p_* p^*$ and $(p')_* (p')^*$ where $p : P \defeq P^a_k \to X$ and $p' : P' \defeq P^{a'}_{k'} \to X$, consider the following commutative diagram :

\squarediagramword{P \x_X P'}{P}{P'}{X}{\overline{p'}}{\overline{p}}{p}{p'}

\noindent By axiom \textbf{(A2)} in $\go$, 
$$p_* p^* (p')_* (p')^* = p_* {\overline{p'}}_* {\overline{p}}^* (p')^* = (p')_* {\overline{p}}_* {\overline{p'}}^* p^* = (p')_* (p')^* p_* p^*.$$ 
The commutativity between $q$ and $q'$, $p$ and $q$ follow from similar arguments. Hence, the ring homo\morp\ $\cat{G} : \cat{R} \to \Endo{\go(X)}$ is well-defined.

The statement we are going to prove is $\cat{G}(G^X_{n,m}) = \cat{G}(G^Y_{m,n})$ as elements in $\Endo{\go(X)}$. Notice that we do not assume $\gdprdivisorssum$ to be a \rsncd\ in the setup anymore. Moreover, if $A_i$ is a \sm\ divisor, then 
$$\cat{G}(X_i)[\id_X] = c(\O(A_i))[\id_X] = [A_i \embed X]$$ 
by the \textbf{(Sect)} axiom in \cite{LePa}. So, the statement corresponding to this modified $\cat{G}$ is actually stronger than what we aimed to prove at the beginning (we will make this more precise later). For simplicity, we will still call this statement $GDPR(n,m)$ within this proof.

\vtab

Here is the outline of the proof. We will prove that $GDPR(n,m)$ holds by assuming $GDPR(n,1)$. Then, for $n \geq 3$, we will prove $GDPR(n,1)$ by assuming $GDPR(n-1,1)$ and $GDPR(2,1)$. Thus, we reduce the proof of $GDPR(n,m)$ to the proof of $GDPR(2,1)$, which is essentially the e\edpr\ in \cite{LePa}. But since the definition of $\cat{G}$ is modified, $GDPR(2,1)$ becomes a stronger statement. Hence, there is still some works needed to be done.

Suppose $GDPR(n,1)$ holds. Then, for a given \equidim\ $X \in \smcat$ and divisors \gdprdivisors\ on $X$, let $C \defeq A_1 + \cdots + A_n$. Consider the setup corresponding to $A_1 + \cdots + A_n \sim C$. From $GDPR(n,1)$, we get $\cat{G}(G^X_{n,1}) = \cat{G}(G^Y_{1,n})$. That means, as elements in $\Endo{\go(X)}$,
$$c(\O(C)) = \cat{G}(X_1 + \cdots + X_n + E^X_n) + \cat{G}(F^X_n)c(\O(C)).$$
Similarly, by considering the setup $C \sim B_1 + \cdots + B_m$, we get
$$c(\O(C)) = \cat{G}'(Y_1 + \cdots + Y_m + E^Y_m) + \cat{G}'(F^Y_m)c(\O(C))$$
with corresponding $\cat{G}'$.

Now, consider the map $\cat{G}''$ corresponding to the setup $\gdprdivisorsequi$. Then, by observing that $\cat{G}=\cat{G}''$ on terms without $Y$ or $V$ and $\cat{G}'=\cat{G}''$ on terms without $X$ or $U$, we have
\begin{eqnarray}
&&c(\O(C)) \nonumber\\
&=& \cat{G}(X_1 + \cdots + X_n + E^X_n) + \cat{G}(F^X_n)c(\O(C)) \nonumber\\
&=& \cat{G}''(X_1 + \cdots + X_n + E^X_n) + \cat{G}''(F^X_n)\ (\cat{G}'(Y_1 + \cdots + Y_m + E^Y_m) + \cat{G}'(F^Y_m)c(\O(C))) \nonumber\\
&=& \cat{G}''(X_1 + \cdots + X_n + E^X_n) + \cat{G}''(F^X_n)\ (\cat{G}''(Y_1 + \cdots + Y_m + E^Y_m) + \cat{G}''(F^Y_m)c(\O(C))) \nonumber\\
&=& \cat{G}''(G^X_{n,m})+ \cat{G}''(F^X_nF^Y_m)c(\O(C)). \nonumber
\end{eqnarray}
On the other hand,
\begin{eqnarray}
&&c(\O(C)) \nonumber\\
&=& \cat{G}'(Y_1 + \cdots + Y_m + E^Y_m) + \cat{G}'(F^Y_m)c(\O(C)) \nonumber\\
&=& \cat{G}''(G^Y_{m,n})+ \cat{G}''(F^X_nF^Y_m)c(\O(C)). \nonumber
\end{eqnarray}
Then, the result follows from cancelling $\cat{G}''(F^X_nF^Y_m)c(\O(C))$ on both sides. That means it is enough to show $GDPR(n,1)$.

\vtab

Assume it is true for $n-1$ case and $n = 2$ case. Now, we start with a setup $A_1 + \cdots + A_n \sim B$. Let $C \defeq A_1 + \cdots + A_{n-1}$. Consider the setup $C + A_n \sim B$. Define $\gs \defeq {p_1}_* p_1^*$ and $\gs' \defeq {p_2}_* p_2^* - {p_3}_* p_3^*$ where 
\begin{eqnarray}
p_1 &:& \P(\O \oplus \O(C)) \to X \nonumber\\
p_2 &:& \P(\O \oplus \O(1)) \to \P(\O(-A_n) \oplus\O(-B)) \to X \nonumber\\
p_3 &:& \P(\O \oplus \O(-A_n) \oplus\O(-B)) \to X. \nonumber
\end{eqnarray}
Then, by $GDPR(2,1)$, we get
\begin{eqnarray}
c(\O(B)) &=& c(\O(C)) + c(\O(A_n)) \label{eqn1} \\
&& -\ c(\O(C)) c(\O(A_n)) \gs + c(\O(B))c(\O(C))c(\O(A_n))\gs'. \nonumber
\end{eqnarray}
By $GDPR(n-1,1)$ corresponding to the setup $A_1 + \cdots + A_{n-1} \sim C$, we have $\cat{G}'(G^X_{n-1,1}) = \cat{G}'(G^Y_{1,n-1})$ where $\cat{G}'$ is the corresponding ring homo\morp. That implies
\begin{eqnarray}
c(\O(C)) = \cat{G}'(X_1 + \cdots + X_{n-1} + E^X_{n-1}) + c(\O(C))\cat{G}'(F^X_{n-1}). \label{eqn2} 
\end{eqnarray}
Now, consider the setup $A_1 + \cdots + A_n \sim B$ and call the corresponding ring homo\morp\ $\cat{G}$. Then, $\cat{G}=\cat{G}'$ on terms involving only $X_i$, $U^p_k$, if $1 \leq i,k \leq n-1$. Also, we have $\cat{G}(X_n) = c(\O(A_n))$. 

For simplicity, we will drop the notation $\cat{G}$. Hence, as elements in $\Endo{\go(X)}$,
\begin{eqnarray}
&&c(\O(B)) \nonumber\\
&=& c(\O(C)) + X_n - c(\O(C)) X_n \gs + c(\O(B))c(\O(C))X_n\gs' \nonumber\\
&& \text{(by equation (\ref{eqn1}))} \nonumber\\
&=& (X_1 + \cdots + X_{n-1} + E^X_{n-1} + c(\O(C))F^X_{n-1}) + X_n \nonumber\\
&& -\ X_n \gs\ (X_1 + \cdots + X_{n-1} + E^X_{n-1} + c(\O(C))F^X_{n-1}) \nonumber\\
&& +\ c(\O(B))X_n\gs'\ (X_1 + \cdots + X_{n-1} + E^X_{n-1} + c(\O(C))F^X_{n-1}) \nonumber\\
&& \text{(by equation (\ref{eqn2}))} \nonumber\\
&=& X_1 + \cdots + X_{n-1} + X_n \nonumber\\
&& +\ E^X_{n-1} - (X_1 + \cdots + X_{n-1} + E^X_{n-1}) X_n \gs \nonumber\\
&& +\ c(\O(B))\gs'X_n(X_1 + \cdots + X_{n-1} + E^X_{n-1}) \nonumber\\
&& +\ c(\O(C))F^X_{n-1}  - c(\O(C))F^X_{n-1} X_n \gs + c(\O(B))c(\O(C))\gs'X_nF^X_{n-1} \nonumber\\
&=& X_1 + \cdots + X_n \nonumber\\
&& +\ (E^X_n + X_nF^X_{n-1}) \nonumber\\
&& +\ c(\O(B))(F^X_n - F^X_{n-1}) \nonumber\\
&& +\ c(\O(C))F^X_{n-1}  - c(\O(C))F^X_{n-1} X_n \gs + c(\O(B))c(\O(C))\gs'X_nF^X_{n-1}. \nonumber\\
&& \text{(by definition of $E^X_n$ and $F^X_n$ and the fact that $\gs = \cat{G}(U^1_{n-1})$ and $\gs' = \cat{G}(U^2_n - U^3_n$))} \nonumber
\end{eqnarray}
Notice that the last three terms are
\begin{eqnarray}
&&c(\O(C))F^X_{n-1}  - c(\O(C))F^X_{n-1} X_n \gs + c(\O(B))c(\O(C))\gs'X_nF^X_{n-1} \nonumber\\
&=&(c(\O(B))- X_n + c(\O(C)) X_n \gs - c(\O(B))c(\O(C))X_n\gs')\ F^X_{n-1} \nonumber\\
&& -\ c(\O(C))F^X_{n-1} X_n \gs + c(\O(B))c(\O(C))\gs'X_nF^X_{n-1} \nonumber\\
&& \text{(by equation (\ref{eqn1}))} \nonumber\\
&=& (c(\O(B)) - X_n)F^X_{n-1}. \nonumber
\end{eqnarray}
Hence,
\begin{eqnarray}
c(\O(B)) &=& X_1 + \cdots + X_n + E^X_n + X_nF^X_{n-1} + c(\O(B))(F^X_n - F^X_{n-1}) \nonumber\\
&& +\ (c(\O(B)) - X_n)F^X_{n-1} \nonumber\\
&=& X_1 + \cdots + X_n + E^X_n + c(\O(B))F^X_n, \nonumber
\end{eqnarray}
which is exactly $\cat{G}(G^Y_{1,n}) = \cat{G}(G^X_{n,1})$. That means it is enough to show $GDPR(2,1)$.

Suppose $X \in \smcat$ is \equidim\ and $\L$, $\cat{M}$ are two invertible sheaves over $X$. Define an element $H(\L,\cat{M}) \in \Endo{\go(X)}$ as the following :
\begin{eqnarray}
H(\L,\cat{M}) &\defeq& c(\L) + c(\cat{M}) - c(\L)c(\cat{M}){p_1}_*{p_1}^* \nonumber\\
&& +\ c(\L) c(\cat{M}) c(\L \otimes \cat{M}) ({p_2}_*{p_2}^*-{p_3}_*{p_3}^*) - c(\L \otimes \cat{M}) \nonumber\\
&& \text{where $p_1 : \P(\O \oplus \L) \to X$} \nonumber\\
&& p_2 : \P(\O \oplus \O(1)) \to \P(\dual{\cat{M}} \oplus \dual{(\L \otimes \cat{M})}) \to X \nonumber\\
&& p_3 : \P(\O \oplus \dual{\cat{M}} \oplus \dual{(\L \otimes \cat{M})}) \to X. \nonumber
\end{eqnarray}
Observe that if $A$, $B$, $C$ are divisors on $X$ such that $A + B \sim C$, then 
$$H(\O_X(A), \O_X(B)) \defeq \cat{G}(G^X_{2,1}) - \cat{G}(G^Y_{1,2})$$
where $\cat{G}$ is the ring homo\morp\ corresponding to the setup $A + B \sim C$. In other words, it is enough to show $H(\L,\cat{M})=0$ for any \equidim\ $X \in \smcat$ and invertible sheaves $\L$, $\cat{M}$ over $X$. For this purpose, we need the following two Lemmas.

\begin{lemma}
Suppose $f : X' \to X$ is a \morp\ in $\smcat$ such that $X$, $X'$ are both \equidim\ and $\L$, $\cat{M}$ are two invertible sheaves over $X$.

\begin{tabular}[t]{ll}
1. & If $f$ is \proj, then $H(\L,\cat{M})f_* = f_*H(f^*\L,f^*\cat{M})$. \\
2. & If $f$ is \sm, then $H(f^*\L,f^*\cat{M})f^* = f^*H(\L,\cat{M})$.
\end{tabular}
\end{lemma}

\begin{proof}
1.\tab Axiom \textbf{(A3)} in $\go$ implies that $c(\L)f_*=f_*c(f^*\L)$. For $p_i$, consider the commutative diagram

\squarediagramword{P^i \x_X X'}{X'}{P^i}{X}{p_i'}{f'}{f}{p_i}

\noindent By axiom \textbf{(A2)}, we have ${p_i}_*{p_i}^*f_* = {p_i}_*f'_*{p_i}'^* = f_*{p_i}'_*{p_i}'^*$ and the \morp s $p_i'$ are
\begin{eqnarray}
P^1 \x_X X' &=& \P(\O \oplus f^*\L) \to X' \nonumber\\
P^2 \x_X X' &=& \P(\O \oplus \O(1)) \to \P(f^*\dual{\cat{M}} \oplus f^* \dual{(\L \otimes \cat{M})}) \to X' \nonumber\\
P^3 \x_X X' &=& \P(\O \oplus f^*\dual{\cat{M}} \oplus f^* \dual{(\L \otimes \cat{M})}) \to X'. \nonumber
\end{eqnarray}

2.\tab Similarly, axiom \textbf{(A4)} implies that $c(f^*\L)f^*=f^*c(\L)$. For $p_i$, we can consider the same diagram above and we get ${p_i}'_*{p_i}'^*f^* = {p_i}'_*f'^*{p_i}^* = f^*{p_i}_*{p_i}^*.$
\end{proof}

\begin{lemma}
\label{sectaxiomlemma}
Suppose $X$ is a \sm\ $k$-scheme, $\L_1$, $\L_2, \ldots, \L_n$ are invertible sheaves over $X$ and $L_1$, $L_2, \ldots, L_n$ are the corresponding line bundles over $X$. Let $\tilde{X} \defeq L_1 \x_X L_2 \x_X \cdots \x_X L_n$ and $\gp : \tilde{X} \to X$ be the projection. Then, there are canonically defined global sections $s_i \in {\rm H}^0(\tilde{X},\gp^*\L_i)$ such that, for each $i$, the section $s_i$ will cut out a \sm\ divisor $D_i$ on $\tilde{X}$ and $D_1 + \cdots + D_n$ is a \rsncd.
\end{lemma}

\begin{proof}
Define $s_i : \tilde{X} \to \tilde{X} \x_X L_i$ by $(x,v_1, \ldots ,v_n) \mapsto (x,v_1, \ldots ,v_n, v_i)$. This is a canonically defined global section of $\gp^*\L_i$. It cuts out the divisor $D_i \defeq \{(x,v_1, \ldots ,v_n)\ |\ v_i=0\}$. Moreover, the intersection of $D_{i_1}, \ldots , D_{i_j}$ is just $\{(x,v_1, \ldots ,v_n)\ |\ v_{i_1} = v_{i_2}= \cdots = v_{i_j}=0\}$, which is \sm\ and has codimension $j$ in $\tilde{X}$.
\end{proof}

We are now ready to prove $H(\L,\cat{M}) = 0$. First of all, 
$$H(\L,\cat{M})[f] = H(\L,\cat{M})f_*[\id] = f_*H(f^*\L,f^*\cat{M})[\id].$$ 
So, it is enough to consider the element $[\id]$. Let $\L_1$, $\L_2$, $\L_3$ be the invertible sheaves $\L$, $\cat{M}$, $\L \otimes \cat{M}$ \resp\ and $\gp : \tilde{X} \to X$ as in Lemma \ref{sectaxiomlemma}. Then, we have  
$$\gp^*H(\L,\cat{M})[\id] = H(\gp^*\L,\gp^*\cat{M})\gp^*[\id] = H(\gp^*\L,\gp^*\cat{M})[\id].$$ 
By the extended homotopy property in $\go$, $\gp^* : \go(X) \to \go(\tilde{X})$ is an isomorphism. That means it is enough to prove $H(\L,\cat{M})[\id] = 0$ when there are divisors $A$, $B$, $C$ on $X$ such that $\L \cong \O_X(A)$, $\cat{M} \cong \O_X(B)$, $C \sim A + B$ and $A + B + C$ is a \rsncd. In this case,
\begin{eqnarray}
&& H(\L,\cat{M})[\id] \nonumber\\
& = & c(\O(A))[\id] + c(\O(B))[\id] - c(\O(A))c(\O(B)){p_1}_*{p_1}^*[\id]\nonumber\\
&& +\ c(\O(A))\, c(\O(B))\, c(\O(C))\, ({p_2}_*{p_2}^*-{p_3}_*{p_3}^*)[\id] - c(\O(C))[\id]\nonumber\\
&=& [A \embed X] + [B \embed X] - {p_1}_*{p_1}^*c(\O(A))c(\O(B))[\id]\nonumber\\
&& +\ ({p_2}_*{p_2}^*-{p_3}_*{p_3}^*)\, c(\O(A))\, c(\O(B))\, c(\O(C))\, [\id] - [C \embed X] \nonumber\\
&& (\text{by \textbf{(Sect)} axiom in $\go$}) \nonumber
\end{eqnarray}
\begin{eqnarray}
&=& [A \embed X] + [B \embed X] - {p_1}_*{p_1}^*[A \cap B \embed X] \nonumber\\
&& +\ ({p_2}_*{p_2}^*-{p_3}_*{p_3}^*)[A \cap B \cap C \embed X] - [C \embed X] \nonumber\\
&& (\text{by \textbf{(Sect)} axiom}) \nonumber\\
&=& [A \embed X] + [B \embed X] - [\P(\O \oplus \O(A)) \to A \cap B \embed X]\nonumber\\
&& +\ [\P(\O \oplus \O(1)) \to \P(\O(-B) \oplus\O(-C)) \to A \cap B \cap C \embed X] \nonumber\\
&& -\ [\P(\O \oplus \O(-B) \oplus\O(-C))\to A \cap B \cap C \embed X] - [C \embed X] \nonumber\\
&=& 0 \nonumber
\end{eqnarray}
by the extended double point relation in \cite{LePa} (Lemma 5.2). Hence, we proved the following Proposition.

\begin{prop}
\label{gdprholdprop}
Suppose $X \in \smcat$ is \equidim\ and \gdprdivisors\ are divisors on $X$ such that $\gdprdivisorsequi$. Let $\cat{G} : \cat{R} \to \Endo{\go(X)}$ be the corresponding map constructed before. Then, $\cat{G}(G^X_{n,m}) = \cat{G}(G^Y_{m,n})$.
\end{prop}

We can now apply this statement to prove that the g\gdpr\ holds in $\go$.

\begin{cor} (G\gdpr\ holds in $\go$)
\label{gdprholdcor}

Suppose $X$ is a separated scheme of finite type over $k$ and there is a \proj\ \morp\ $\gph : Y \to X$ such that $Y$ is in $\smcat$ and is \equidim. Moreover, suppose \gdprdivisors\ are divisors on $Y$ such that $\gdprdivisorsequi$ and $\gdprdivisorssum$ is a \rsncd. Let $\cat{G} : \cat{R} \to \go(X)$ be the corresponding map constructed before. Then, $\cat{G}(G^X_{n,m}) = \cat{G}(G^Y_{m,n})$.
\end{cor}

\begin{proof}
By definition, $\cat{G}(G^X_{n,m}) = \gph_* \circ \cat{G}'(G^X_{n,m})$ and $\cat{G}(G^Y_{m,n}) = \gph_* \circ \cat{G}'(G^Y_{m,n})$ where $\cat{G}'$ is the map corresponding to the setup $\id_Y : Y \to Y$ with the same set of divisors. So, we may assume $\gph = \id_X$. Then, it follows from the fact that 
\begin{eqnarray}
&& \cat{G}(G^X_{n,m})[\id_X] \nonumber\\
&& \text{(the modified definition $\cat{G} : \cat{R} \to \Endo{\go(X)}$)} \nonumber\\
&=& \cat{G}(G^X_{n,m}) \nonumber\\
&& \text{(the original definition $\cat{G} : \cat{R} \to \go(X)$)} \nonumber
\end{eqnarray}
and similarly for $G^Y_{m,n}$.
\end{proof}

\begin{rmk}
\label{zerodivisorrmk}
\rm{
Notice that in the g\gdpr\ setup $\gph : Y \to X$ with $\gdprdivisorsequi$ on $Y$, we do not assume $A_i$ or $B_j$ to be nonempty. If $\cat{G}$ is the map corresponding to $\gdprdivisorsequi$ and $\cat{G}'$ is the map corresponding to 
$$A_1 + \cdots + A_n + \sum_{i=n+1}^N C_i \sim B_1 + \cdots + B_m + \sum_{j=m+1}^M D_j$$ 
where $\{C_i, D_j\}$ are zero divisors, then
$$\cat{G}(G^X_{n,m}) = \cat{G}'(G^X_{n,m}) = \cat{G}'(G^X_{N,M}) \text{\tab and\tab } \cat{G}(G^Y_{m,n}) = \cat{G}'(G^Y_{m,n}) = \cat{G}'(G^Y_{M,N}).$$

\begin{proof}
Notice that if a general term $\generalterm$ in $\cat{R}$ contains $X_i$ or $Y_j$ with $n+1 \leq i \leq N$ or $m+1 \leq j \leq M$, then $\cat{G}'(\generalterm) = 0$. By definition,
$$E^X_{n+1} = E^X_n + \sum \text{ terms with } X_{n+1}.$$
Inductively,
$$E^X_N = E^X_n + \sum \text{ terms with } X_i \text{ where } n+1 \leq i \leq N.$$
Similarly,
\begin{eqnarray}
F^X_N &=& F^X_n + \sum \text{ terms with } X_i \text{ where } n+1 \leq i \leq N \nonumber\\
E^Y_M &=& E^Y_m + \sum \text{ terms with } Y_j \text{ where } m+1 \leq j \leq M \nonumber\\
F^Y_M &=& F^Y_m + \sum \text{ terms with } Y_j \text{ where } m+1 \leq j \leq M. \nonumber
\end{eqnarray}
Hence,
\begin{eqnarray}
G^X_{N,M} &=& X_1 + \cdots + X_N + E^X_N + (Y_1 + \cdots + Y_M)F^X_N + E^Y_M F^X_N \nonumber\\
&=& X_1 + \cdots + X_n + E^X_n + (Y_1 + \cdots + Y_m)F^X_n + E^Y_m F^X_n \nonumber\\
&& +\ \sum \text{ terms with } X_i \text{ where } n+1 \leq i \leq N \nonumber\\
&& +\ \sum \text{ terms with } Y_j \text{ where } m+1 \leq j \leq M \nonumber\\
&=& G^X_{n,m} + \text{redundant terms}. \nonumber
\end{eqnarray}
That means $\cat{G'}(G^X_{N,M}) = \cat{G}'(G^X_{n,m}) = \cat{G}(G^X_{n,m})$. Similarly, $\cat{G}'(G^Y_{M,N}) = \cat{G}'(G^Y_{m,n}) = \cat{G}(G^Y_{m,n})$.
\end{proof}
}
\end{rmk}

\vtab

\subsection{Definition and basic properties}

Now we will define our \equi\ algebraic cobordism theory using the g\gdpr. 

\begin{defn}
{\rm
For an object $X$ in $\gsmcat{G}$, let $M_G(X)$ be the set of isomorphism classes over $X$ of \proj\ \morp s $f : Y \to X$ in $\gsmcat{G}$. Then, $M_G(X)$ is a monoid under disjoint union of domains, i.e.
$$[Y \to X] + [Y' \to X] \defeq [Y \disjoint Y' \to X].$$
We define the abelian group $M_G(X)^+$ as the group completion of $M_G(X)$.
}\end{defn}

The $i$-th graded piece (cohomological grading) : $(M_G(X)^+)^i$, when $X$ is equidimensional, is given by $[Y \to X]$ where $Y$ is equidimensional and $i = \dim X - \dim Y$. We also have homological grading $M_G(X)_i^+$ where $i$ denotes the dimension of $Y$, if $Y$ is \equidim.

\begin{rmk}
\rm{
The main reason for focusing on \qproj\ $X$ instead of just separated scheme of finite type over $k$ as in \cite{LePa} is because we will sometimes consider the quotient $X/G$ and the operation of taking quotient works better in the \qproj\ category.
}
\end{rmk}

Next, we will define the notion of \equi\ g\gdpr\ which is the \equi\ analog of the g\gdpr\ we just defined in section \ref{gdprsubsection}. To be more precise, we will consider the following setup.

Let $\gph : Y \to X$ be a \proj\ \morp\ in $\gsmcat{G}$ such that $Y$ is \equidim. In addition, \gdprdivisors\ are $G$-invariant divisors on $Y$ such that $\gdprdivisorsequi$ ($G$-\equi ly linearly equivalent) and $\gdprdivisorssum$ is a \rsncd. In this setup, we construct a corresponding abelian group homo\morp\ $\cat{G} : \cat{R} \to M_G(X)^+$ by the exact same definition as in section \ref{gdprsubsection}. Notice that all objects involved are \sm\ varieties with natural $G$-action and all \morp s involved are naturally $G$-\equi. We will call the collection of $\gph : Y \to X$ together with the divisors as above a g\gdpr\ setup over $X$, or GDPR setup.

\begin{defn}
{\rm
The \equi\ algebraic cobordism group $\cobg{G}{}{X}$ is defined as the quotient of $M_G(X)^+$ by the subgroup generated by all expressions $\cat{G}(G^X_{n,m}) - \cat{G}(G^Y_{m,n})$ where $\cat{G}$ corresponds to some GDPR setup over $X$.
}\end{defn}

\begin{rmk}
\rm{
As pointed out in remarks \ref{gdprremark}, if $\gph : Y \to X$ is the \morp\ defining $\cat{G}$, then $\cat{G}(G^X_{n,m})$, $\cat{G}(G^Y_{m,n})$ both lie in $M_G(X)_{\dim Y-1}^+$. Hence, if $X$ is equidimensional, we can define a homological (cohomological) grading on $\cobg{G}{}{X}$, namely
$$\cobg{G}{}{X} = \bigoplus_i\ \cobg{G}{i}{X} = \bigoplus_i\ \cobglow{G}{i}{X}$$
where $\cobglow{G}{i}{X}$ is defined as the quotient of $M_G(X)^+_i$ by the subgroup generated by all expressions $\cat{G}(G^X_{n,m}) - \cat{G}(G^Y_{m,n})$ such that $\cat{G}$ corresponds to some GDPR setup over $X$ where the dimension of the domain of $\gph$ is $i+1$. Similarly, the group $\cobg{G}{i}{X}$ is the quotient of $(M_G(X)^+)^i$ with GDPR setups over $X$ when the dimension of the domain of $\gph$ is $\dim X - i + 1$.
}
\end{rmk}

G\gdpr\ is a generalization of the d\dpr\ in the \equi\ configuration.

\begin{prop}
\label{equidprholds}
Suppose $\gph : Y \to X \x \P^1$ is a \proj\ \morp\ in $\gsmcat{G}$ (with trivial $G$-action on $\P^1$) such that $Y$ is \equidim. Let $\gx \in \P^1$ be a closed point. Assume  that the fiber $Y_{\gx} \defeq (\gp_2 \circ \gph)^{-1}(\gx)$ is a \sm\ $G$-invariant divisor on $Y$ and there exist \sm\ $G$-invariant divisors $A$, $B$ on $Y$ such that $Y_0 = A \cup B$ and $A$, $B$ intersect transversely, then
$$[Y_{\gx} \to X] = [A \to X] + [B \to X] - [\P(\O \oplus \O(A)) \to A \cap B \to X]$$
as elements in $\cobg{G}{}{X}$.
\end{prop}

\begin{proof}
Since $Y_{\gx}$ is disjoint from $A$, $B$ and $A$, $B$ intersect transversely, $Y_{\gx} + A + B$ is a \rsncd\ on $Y$. In addition, since $\P^1$ has trivial $G$-action, $Y_{\gx} \sim A + B$. That defines a g\gdpr\ setup $\gp_1 \circ \gph : Y \to X$ with $Y_{\gx} \sim A + B$. Thus, we obtain the equality $\cat{G}(G^X_{1,2}) = \cat{G}(G^Y_{2,1})$ in $\cobg{G}{}{X}$ which is exactly
$$[Y_{\gx} \to X] = [A \to X] + [B \to X] - [\P(\O \oplus \O(A)) \to A \cap B \to X].$$
\end{proof}

\vtab

In \cite{LePa}, M. Levine and R. Pandharipande listed several natural axioms and properties that an algebraic cobordism theory should satisfy. Here, we will show the \equi\ version of some of them.

\textbf{(D1)} If $f : X \to X'$ in $\gsmcat{G}$ is \proj, then there is an abelian group homo\morp 
$$f_* : \cobglow{G}{*}{X} \to \cobglow{G}{*}{X'}.$$ 
Moreover, if $f,g$ are both \proj, then $(g \circ f)_* = g_* \circ f_*.$

\begin{proof}
As in the $\go_*$ theory of \cite{LePa}, the push-forward $f_*$ is given by sending $[h : Y \to X]$ to $[f \circ h : Y \to X']$. We need to check that it preserves the g\gdpr.

Suppose a g\gdpr\ on $X$ is defined by a \proj\ \morp\ $\gph : Y \to X$ in $\gsmcat{G}$ with $\gdprdivisorsequi$. It defines a homo\morp\ $\cat{G} : \cat{R} \to M_G(X)^+$. We can then consider the g\gdpr\ on $X'$ given by $f \circ \gph : Y \to  X'$ with the same set of divisors. This will also define a homo\morp\ $\cat{G}' : \cat{R} \to M_G(X')^+$. Thus, for a general term $X_i \cdots Y_j \cdots U^p_k \cdots V^q_l \cdots$ in $\cat{R}$,
\begin{eqnarray}
&&f_* \circ \cat{G}(X_i \cdots Y_j \cdots U^p_k \cdots V^q_l \cdots) \nonumber\\
&=& f_*[A_i \x_Y \cdots \x_Y  B_j \x_Y \cdots \x_Y P^p_k \x_Y \cdots \x_Y Q^q_l \x_Y \cdots \to X] \nonumber\\
&=& [A_i \x_Y \cdots \x_Y  B_j \x_Y \cdots \x_Y P^p_k \x_Y \cdots \x_Y Q^q_l \x_Y \cdots \to X \to X']. \nonumber
\end{eqnarray}
On the other hand,
\begin{eqnarray}
&&\cat{G}'(X_i \cdots Y_j \cdots U^p_k \cdots V^q_l \cdots) \nonumber\\
&=& [A_i \x_Y \cdots \x_Y  B_j \x_Y \cdots \x_Y P^p_k \x_Y \cdots \x_Y Q^q_l \x_Y \cdots \to X \to X']. \nonumber
\end{eqnarray}
That implies $f_* \circ \cat{G} = \cat{G}'$. In particular, $f_* \circ \cat{G}(G^X_{n,m}) = \cat{G}'(G^X_{n,m})$ and $f_* \circ \cat{G}(G^Y_{m,n}) = \cat{G}'(G^Y_{m,n})$, which means $f_* \circ \cat{G}(G^X_{n,m}) = f_* \circ \cat{G}(G^Y_{m,n})$ in $\cobg{G}{}{X'}$. So, the group homo\morp\ $f_* : \cobglow{G}{}{X} \to \cobglow{G}{}{X'}$ is well-defined. Clearly, it preserves the homological grading and $(g \circ f)_* = g_* \circ f_*$.
\end{proof}

\vtab

\textbf{(D2)} If $f : X' \to X$ in $\gsmcat{G}$ is \sm\ such that $X,X'$ are both equidimensional, then there is an abelian group homo\morp
$$f^* : \cobg{G}{*}{X} \to \cobg{G}{*}{X'}.$$

\begin{proof}
Let $[Y \to X]$ be an element $\cobg{G}{}{X}$, then we define the pull-back $f^*[Y \to X]$ as $[Y \x_X X' \to X']$. First of all, $Y \x_X X'$ is a \sm\ variety with natural diagonal $G$-action and the \morp\ $Y \x_X X' \to X'$ is \proj\ and $G$-\equi.

Consider a GDPR setup over $X$ given by $\gph : Y \to X$ with divisors \gdprdivisors\ on $Y$ and $\cat{G}$ be the corresponding map. We have the following commutative diagram :

\squarediagramword{Y' \defeq Y \x_X X'}{Y}{X'}{X}{f'}{\gph'}{\gph}{f}

\noindent We obtain a g\gdpr\ setup over $X'$ given by $\gph' : Y' \to X'$ with divisors $f'^*A_1, \ldots, f'^*A_n, f'^*B_1, \ldots, f'^*B_m$ on $Y'$. Let $\cat{G}'$ be the corresponding homo\morp. The smoothness of $f'$ implies that $f'^*A_1 + \cdots + f'^*A_n + f'^*B_1 + \cdots + f'^*B_m$ is still a \rsncd. Observe that if $P^1_k = \P(\O \oplus \O(D))$ is a $G$-\equi\ \proj\ bundle over $Y$, then $P^1_k \x_Y Y' \cong \P(\O \oplus \O(f'^*D))$, as $G$-\equi\ \proj\ bundles over $Y'$. So, 
$$\cat{G}'(U^1_k) = [P^1_k \x_Y Y' \to Y'] = f^*[P^1_k \to Y] = f^* \circ \cat{G}(U^1_k).$$ 
Similar statements \wrt\ $U^p_k$ and $V^q_l$ also hold. For a general term,
\begin{eqnarray}
&&f^* \circ \cat{G}(X_i \cdots U^p_k \cdots) \nonumber\\
&=& f^*[A_i \x_Y \cdots \x_Y P^p_k \x_Y \cdots \to X] \nonumber\\
&=& [(A_i \x_Y \cdots \x_Y P^p_k \x_Y \cdots) \x_X X' \to X']. \nonumber
\end{eqnarray}
On the other hand,
\begin{eqnarray}
&&\cat{G}'(X_i \cdots U^p_k \cdots) \nonumber\\
&=& [(A_i \x_Y Y') \x_{Y'} \cdots \x_{Y'} (P^p_k \x_Y Y') \x_{Y'} \cdots \to X']. \nonumber\\
&=& [(A_i \x_Y \cdots \x_Y  P^p_k \x_Y \cdots ) \x_Y Y' \to X']. \nonumber\\
&=& [(A_i \x_Y \cdots \x_Y P^p_k \x_Y \cdots ) \x_X X' \to X']. \nonumber
\end{eqnarray}
That shows the well-definedness of $f^* : \cobg{G}{}{X} \to \cobg{G}{}{X'}$. Since $f$ is \sm, taking fiber product with $f : X' \to X$ preserves codimension. Thus, $f^*$ preserves the cohomological grading.
\end{proof}

\vtab

\textbf{(D3)} In \cite{LePa}, there is a discussion of the first Chern class operator. This will be addressed in the next section.

\vtab

\textbf{(D4)} For each pair $(X, X')$ of objects in $\gsmcat{G}$, there is a bilinear, graded pairing
$$\x : \cobglow{G}{i}{X} \x \cobglow{G}{j}{X'} \to \cobglow{G}{i+j}{X \x X'}$$
which is commutative, associative and admits a distinguished element $1 \in \cobglow{G}{0}{\pt}$ as a unit.

\begin{proof}
The definition is standard. We define 
$$[f : Y \to X] \x [f' : Y' \to X'] \defeq [f \x f' : Y \x Y' \to X \x X'].$$

Suppose a GDPR setup over $X$ is given by $\gph : Z \to X$ with divisors \gdprdivisors\ on $Z$ and $\cat{G}$ be the corresponding homo\morp. We need to show $$\cat{G}(G^X_{n,m}) \x [f' : Y' \to X'] = \cat{G}(G^Y_{m,n}) \x [f' : Y' \to X'].$$ 
W\withoutlog, we may assume $Y'$ is \equidim. Consider the GDPR setup over $X \x X'$ given by $\gph \x f' : Z \x Y' \to X \x X'$ with divisors $\gp_1^*A_1, \ldots, \gp_1^*A_n$, $\gp_1^*B_1, \ldots, \gp_1^*B_m$ on $Z \x Y'$. Let $\cat{G}'$ be the corresponding \homo. 

Observe that if $P^1_k = \P(\O_Z \oplus \O_Z(D))$, then $P^1_k \x Y' = \P(\O_{Z \x Y'} \oplus \O_{Z \x Y'}(\gp_1^*D))$. So, 
$$\cat{G}'(U^1_k) = [P^1_k \x Y' \to X \x X'] = [P^1_k \to X] \x [Y' \to X'] = \cat{G}(U^1_k) \x [Y' \to X'].$$ 
Similar statements \wrt\ $U^p_k$ and $V^q_l$ also hold. For a general term,
\begin{eqnarray}
&&[f'] \x \cat{G}(X_i \cdots U^p_k \cdots) \nonumber\\
&=& [f'] \x [A_i \x_Z \cdots \x_Z P^p_k \x_Z \cdots \to X] \nonumber\\
&=& [(A_i \x_Z \cdots \x_Z P^p_k \x_Z \cdots) \x Y' \to X \x X']. \nonumber
\end{eqnarray}
On the other hand,
\begin{eqnarray}
&&\cat{G}'(X_i \cdots U^p_k \cdots) \nonumber\\
&=& [(A_i \x Y') \x_{Z \x Y'} \cdots \x_{Z \x Y'} (P^p_k \x Y') \x_{Z \x Y'} \cdots \to X \x X']. \nonumber\\
&=& [(A_i \x_Z \cdots \x_Z P^p_k \x_Z \cdots) \x Y' \to X \x X']. \nonumber
\end{eqnarray}
That shows the well-definedness of $\x$. It is not hard to see that this product is graded, associative and commutative. The unit in $\cobglow{G}{0}{\pt}$ is simply $[\id : \pt \to \pt]$.
\end{proof}

\begin{rmk}
\rm{
We will refer to 
$$\x : \cobglow{G}{i}{X} \x \cobglow{G}{j}{X'} \to \cobglow{G}{i+j}{X \x X'}$$ 
as the external product. This external product gives $\cobglow{G}{*}{\pt}$ a graded ring structure and $\cobglow{G}{*}{X}$ a graded $\cobglow{G}{*}{\pt}$-module structure. In addition, if $f : X \to X'$ is a \proj\ \morp\ in $\gsmcat{G}$, then the push-forward $f_* : \cobglow{G}{*}{X} \to \cobglow{G}{*}{X'}$ will be a graded $\cobglow{G}{*}{\pt}$-module homomorphism. Similarly, if $f : X \to X'$ in $\gsmcat{G}$ is \sm\ such that $X$, $X'$ are \equidim, then the pull-back $f^* : \cobg{G}{*}{X'} \to \cobg{G}{*}{X}$ will be a graded $\cobg{G}{*}{\pt}$-module homomorphism.
}
\end{rmk}

\vtab

The following two properties can be easily derived from the definitions, similarly to \cite{LePa}.

\vtab

\textbf{(A1)} If $f : X \to X'$ and $g : X' \to X''$ are both \sm\ and $X$, $X'$, $X''$ are all \equidim, then
$$(g \circ f)^* = f^* \circ g^*.$$
Moreover, $\id^*$ is the identity \homo. 
\begin{flushright} 
$\square$ 
\end{flushright}

\textbf{(A2)} If $f : X \to Z$ is \proj\ and $g : Y \to Z$ is \sm\ such that $X$, $Y$, $Z$ are all \equidim, then we have $g^*f_* = f'_*g'^*$ in the pull-back square
\squarediagramword{X \x_Z Y}{X}{Y}{Z}{g'}{f'}{f}{g}
\begin{flushright} 
$\square$ 
\end{flushright}

\vtab

\textbf{(A3), (A4), (A5)} in \cite{LePa} are properties involving the Chern class operator. Hence, they will be addressed in the next section.

\vtab

\textbf{(A6)} If $f,g$ are \proj, then
$$\x \circ (f_* \x g_*) = (f \x g)_* \circ \x.$$

\begin{proof}
Let $f : X \to X'$ and $g : Z \to Z'$. The statement follows from the commutativity of the following diagram, which is easy to check.
\squarediagramword{\cobg{G}{}{X} \x \cobg{G}{}{Z}}{\cobg{G}{}{X \x Z}}{\cobg{G}{}{X'} \x \cobg{G}{}{Z'}}{\cobg{G}{}{X' \x Z'}}{\x}{f_* \x g_*}{(f \x g)_*}{\x}
\end{proof}

\vtab

\textbf{(A7)} If $f,g$ are \sm\ with \equidim\ domains and codomains, then
$$\x \circ (f^* \x g^*) = (f \x g)^* \circ \x.$$

\begin{proof}
It follows from the commutativity of the previous diagram with vertical arrows reversed.
\end{proof}

\vtab

\subsection{Results for free action}
\label{freeactionisosubsection}

Before handling the definition of the first Chern class operator, let us mention some useful observations here first. 

Consider the set of objects $Y \in \gsmcat{G}$ such that the geometric quotient (definition 0.6 in \cite{Mu}\,) $Y/G$ exists as scheme over $k$, lies in $\smcat$ and the map $Y \to Y/G$ is a principal $G$-bundle. Denote this set of objects by $\cat{D}$. Notice that if $Y/G$ exists and $Y \to Y/G$ is a principal $G$-bundle, then $Y \to Y/G$ is locally trivial in the \etale\ topology (see \cite{EdGr}) and hence, $Y$ is \sm\ if and only if $Y/G$ is \sm\ (Proof in Proposition \ref{domainisind}). Therefore, an object $Y \in \gsmcat{G}$ lies in $\cat{D}$ if and only if $Y/G$ exists, is \qproj\ over $k$ and $Y \to Y/G$ is a principal $G$-bundle. We will consider $\cat{D}$ as a full subcategory of $\gsmcat{G}$. 

Suppose $X$ is a variety in $\cat{D}$, it turns out that there is a one-to-one correspondence between \morp s $Z \to X/G$ in the category $\smcat$ and $G$-\equi\ \morp s $Y \to X$ in the category $\gsmcat{G}$. This important observation will lead us to the proof of the isomorphism
$$\go(X/G)\iso\cobg{G}{}{X}$$
for any $X \in \cat{D}$.

Throughout this paper, we will call going from $X$ to $X/G$ ``descent'' and going from $X/G$ to $X$ ``ascent''. 

\begin{prop}
\label{domainisind}
If $f : Y \to X$ is a \morp\ in $\gsmcat{G}$ and $X$ is in $\cat{D}$, then $Y$ is also in $\cat{D}$.
\end{prop}

\begin{proof}
Recall that the group scheme $G$ we are working with is either a reductive connected group over $k$ or a finite group.

Consider the case when $G$ is connected and reductive. Since $Y$ is \qproj, the map $Y \to X$ is \qproj. Then, there exists an invertible sheaf $\L$ over $Y$ (may not be \glin{G}) which is very ample relative to $X$. By Theorem 1.6 in \cite{Su}, since $Y$ is normal, there exists a positive integer $m$ such that $\L^m$ (\,$\defeq \L^{\otimes m}$) admits a $G$-linearization. Then, by Proposition 7.1 in \cite{Mu}, we have the following commutative diagram in which $Y/G$ is \qproj\ and $Y \to Y/G$ is a principal $G$-bundle.

\squarediagram{Y}{X}{Y/G}{X/G}

\noindent Since $Y \to Y/G$ is a principal $G$-bundle, the \morp\ $Y \to Y/G$ is locally trivial in the \etale\ topology. That means that $Y/G$ can be covered by \etale\ neighborhoods $W$ for which we have the following commutative diagram :

\squarediagramword{W \x G}{Y}{W}{Y/G}{\text{\etale}}{}{}{\text{\etale}}

\noindent Hence, $Y$ is \sm\ if and only if $Y/G$ is \sm. 

For the case when $G$ is finite, just replace $\L^m$ by $\otimes_{\ga \in G}\ \ga^* \L$.
\end{proof}

The following is mostly a standard application of descent theory, but we need to make sure we preserve the \sm ness and \qproj ness assumptions.

\vtab

\begin{prop}
\label{ascentdescent}
For any object $X \in \cat{D}$,

\begin{statementslist}
{\rm (1)} & There is a one-to-one correspondence between the set of \morp s $f : Z \to X/G$ in $\smcat$ and the set of \morp s $g : Y \to X$ in $\gsmcat{G}$, given by sending $Z \to X/G$ to its fiber product with $X \to X/G$. Moreover, its inverse is given by sending $Y \to X$ to $Y/G \to X/G$. \\
{\rm (2)} & The above map defines a one-to-one correspondence between the set of \proj\ \morp s $f : Z \to X/G$ in $\smcat$ and the set of \proj\ \morp s $g : Y \to X$ in $\gsmcat{G}$. \\
{\rm (3)} & The above map defines a one-to-one correspondence between the set of vector bundles $E' \to X/G$ and the set of $G$-equivariant vector bundles $E \to X$.
\end{statementslist}
\end{prop}

\begin{proof}
{\rm (1)}\tab For ascent, consider the following commutative diagram :

\squarediagramword{Z \x_{X/G} X}{X}{Z}{X/G}{g}{}{}{f}

\noindent There is a natural $G$-action on $Z \x_{X/G} X$ and $g$ is $G$-\equi. Since $X, Z$ are \qproj, $Z \x_{X/G} X$ is \qproj.

\vtab

Claim 1 : If $X$ is an object in $\cat{D}$, then the \morp\ $X \to X/G$ is \sm.

Since $X \to X/G$ is a principal $G$-bundle, it is flat and locally trivial in the \etale\ topology. Thus, we have the following commutative diagram :

\squarediagramword{W \x G}{X}{W}{X/G}{\text{\etale}}{}{}{\text{\etale}}

\noindent Let $x$ be a point in $X/G$ and $K$ be the algebraic closure of $k(x)$. Then, by taking fiber product with $\spec{K} \to \spec{k(x)}$, we have the following commutative diagram :

\squarediagramword{W_K \x G}{X_K}{W_K}{\spec{K}}{\text{\etale}}{}{}{\text{\etale}}

\noindent Clearly, $\dim X_K = \dim W_K \x G = \dim G$ and $X_K$ is regular. The claim then follows from Theorem 10.2 in Ch III in \cite{Ha}. \claimend

\vtab

Since the \morp\ $X \to X/G$ is \sm\ and $Z$ is \sm, $Z \x_{X/G} X$ is \sm. That shows the well-definedness of ascent.

For descent, consider the following commutative diagram :

\squarediagramword{Y}{X}{Y/G}{X/G}{g}{}{}{f\ \defeq\ g/G}

\noindent By Proposition \ref{domainisind}, $Y$ is in $\cat{D}$. So, $Y/G$ is in $\smcat$. The fact that these two constructions are inverse to each other is standard and follows from descent theory.

\vtab

{\rm (2)}\tab Ascent clearly preserves \proj ness. For descent, it follows from the descent of properness (Proposition 2 of \cite{EdGr}) and the fact that $Y/G$ is \qproj.

\vtab

{\rm (3)}\tab Ascent clearly takes vector bundles to $G$-\equi\ vector bundles. For descent, it follows from Lemma 1 of \cite{EdGr}.
\end{proof}

We are now ready to prove the following Theorem.

\begin{thm}
\label{ascentdescent2}
Suppose $X$ is an object in $\cat{D}$. Sending $[Z \to X/G]$ to $[Z \x_{X/G} X \to X]$ defines an abelian group iso\morp
$$\gPS : \go^*(X/G) \to \cobg{G}{*}{X}.$$
\end{thm}

\begin{proof}
Define the inverse \homo\ $\gPS^{-1}$ by sending $[Y \to X]$ to $[Y/G \to X/G]$. We will call $\gPS$ ``ascent'' and $\gPS^{-1}$ ``descent''.

First of all, we need to prove that $\gPS$ is well-defined. By Proposition \ref{ascentdescent}, $\gPS$ is well-defined at the level of $M(X/G)^+$. In this proof, we will denote the fiber product with $X \to X/G$ by a star, i.e. $W^* \defeq W \x_{X/G} X$. We also denote by $\gp : X \to X/G$ the projection. Consider the following commutative diagram :

\squarediagramword{Y^*}{X \x \P^1}{Y}{X/G \x \P^1}{\gph^*}{}{}{\gph}

\noindent where $\gph$ corresponds to a d\dpr\ setup over $X/G$ (the fiber $Y_{\gx}$ is a \sm\ divisor, $Y_0 = A \cup B$ for some \sm\ divisors $A$, $B$ and $A$, $B$ intersect transversely).

We want to show that $\gph^*$ gives an \equi\ d\dpr\ setup over $X$. Notice that $Y^*$ is in $\cat{D}$ because $X$ is in $\cat{D}$ (Proposition \ref{domainisind}). So, $Y^*$ is \sm\ and the projection $Y^* \to Y$ is \sm\ (claim 1 in the proof of Proposition \ref{ascentdescent}). Then, $Y^*$ is \equidim, $(Y_{\gx})^* = (Y^*)_{\gx}$, $A^*$ and $B^*$ are $G$-\inv\ divisors on $Y^*$, 
$$A^* \cup B^* = (A \cup B)^* = (Y_0)^* = (Y^*)_0$$
and $A^*$, $B^*$ intersect transversely. Clearly, $\gph^*$ is \proj. Hence, that gives us an \equi\ d\dpr\ setup over $X$. By Proposition \ref{equidprholds}, we obtain the following equation in $\cobg{G}{}{X}$ :
\begin{eqnarray}
\label{eqn3} [Y_{\gx}^* \to X] = [A^* \to X] + [B^* \to X] - [\P(\O_{D^*} \oplus \O_{D^*}(A^*)) \to X]
\end{eqnarray}
\noindent where $D \defeq A \cap B$.

On the other hand, the d\dpr\ on $X/G$ corresponding to $\gph$ is
$$[Y_{\gx} \to X/G] = [A \to X/G] + [B \to X/G] - [\P(\O_D \oplus \O_D(A)) \to X/G].$$
If we apply $\gPS$ on this equation, we will get
\begin{eqnarray}
\label{eqn4} [Y_{\gx}^* \to X] = [A^* \to X] + [B^* \to X] - [\P(\O_D \oplus \O_D(A)) \x_{X/G} X \to X].
\end{eqnarray}
Since 
$$\P(\O_D \oplus \O_D(A)) \x_{X/G} X \cong \P(\,\gp^*(\O_D \oplus \O_D(A))\,) \cong \P(\O_{D^*} \oplus \O_{D^*}(A^*)),$$ 
equations (\ref{eqn3}) and (\ref{eqn4}) are equivalent. This finishes the first half of the proof : well-definedness of $\gPS$.

It remains to show the well-definedness of the inverse $\gPS^{-1}$. By Proposition \ref{ascentdescent}, it is well-defined at the level of $M_G(X)^+$. That means for a given GDPR setup $\gph : Y \to X$ with divisors \gdprdivisors\ on $Y$ and corresponding \homo\ $\cat{G}$, we need to show
$$\gPS^{-1} \circ \cat{G}(G^X_{n,m}) = \gPS^{-1} \circ \cat{G}(G^Y_{m,n})$$
as elements in $\go(X/G).$

First of all, $Y$ is in $\cat{D}$ (by Proposition \ref{domainisind}) implies that $Y/G$ is in $\smcat$ and is \equidim. In addition, for all $i$, the $G$-\inv\ divisor $A_i$ is in $\cat{D}$. So, $A_i/G$ is in $\smcat$. Moreover, $\dim A_i/G = \dim A_i - \dim G$ implies that $A_i/G$ is a \sm\ divisor on $Y/G$. By similar arguments, 
$$A_1/G + \cdots + A_n/G + B_1/G + \cdots + B_m/G$$ 
is a \rsncd\ on $Y/G$. On the other hand, by definition, there exists $f \in {\rm H}^0(Y, \cat{K}^*)^G$ such that 
$$A_1 + \cdots + A_n - B_1 - \cdots - B_m = \divisor{f}.$$ 
By the fact that ${\rm H}^0(Y, \cat{K}^*)^G \cong {\rm H}^0(Y/G, \cat{K}^*)$, we can consider $f$ as an element in ${\rm H}^0(Y/G, \cat{K}^*)$ and deduce that 
$$A_1/G + \cdots + A_n/G - B_1/G - \cdots - B_m/G = \divisor{f}.$$ 
By Proposition \ref{ascentdescent}, $\gph/G : Y/G \to X/G$ is \proj. Hence, we obtain a GDPR setup over $X/G$ given by $\gph/G : Y/G \to X/G$ with divisors $A_1/G, \ldots, A_n/G$, $B_1/G, \ldots, B_m/G$ on $Y/G$. Let $\cat{G}'$ be the corresponding \homo. By Corollary \ref{gdprholdcor}, 
$$\cat{G}'(G^X_{n,m}) = \cat{G}'(G^Y_{m,n})$$
in $\go(X/G)$. So, it will be enough to show $\cat{G}' = \gPS^{-1} \circ \cat{G}$. We will need the following claim first.

\vtab

Claim : For \morp s $Z \to X$ and $Z' \to X$ with $X$, $Z$, $Z' \in \cat{D}$, we have the following isomorphism :
$$(Z \x_X Z')/G \cong Z/G \x_{X/G} Z'/G.$$

Notice that
\begin{eqnarray}
(Z/G \x_{X/G} Z'/G) \x_{X/G} X & \cong & Z/G \x_{X/G} (Z'/G \x_{X/G} X) \nonumber\\
& \cong & Z/G \x_{X/G} Z' \nonumber\\
&& (\text{by Proposition \ref{ascentdescent}}) \nonumber\\
& \cong & Z/G \x_{X/G} X \x_X Z' \nonumber\\
& \cong & Z \x_X Z' \nonumber\\
&& (\text{by Proposition \ref{ascentdescent}}). \nonumber
\end{eqnarray}
Again, by Proposition \ref{ascentdescent}, we get 
$$Z/G \x_{X/G} Z'/G \cong ((Z/G \x_{X/G} Z'/G) \x_{X/G} X)\, /G \cong (Z \x_X Z')/G.$$ 
The proves the claim. \claimend

\vtab

Consider a general term $X_i \cdots U^p_k \cdots$ in $\cat{R}$. On one hand,
\begin{eqnarray}
\cat{G}'(X_i \cdots U^p_k \cdots) &=& [A_i/G \x_{Y/G} \cdots \x_{Y/G} (P^p_k)' \x_{Y/G} \cdots \to X/G] \nonumber\\
&& \text{where $(P^p_k)'$ is the corresponding tower defined by $\{A_i/G\}$.} \nonumber
\end{eqnarray}
On the other hand,
\begin{eqnarray}
\gPS^{-1} \circ \cat{G}(X_i \cdots U^p_k \cdots) &=& \gPS^{-1} [A_i \x_Y \cdots \x_Y P^p_k \x_Y \cdots \to X] \nonumber\\
&& \text{where $P^p_k$ is the corresponding tower defined by $\{A_i\}$} \nonumber\\
&=& [(A_i \x_Y \cdots \x_Y P^p_k \x_Y \cdots)/G \to X/G] \nonumber\\
&=& [A_i/G \x_{Y/G} \cdots \x_{Y/G} P^p_k/G \x_{Y/G} \cdots \to X/G] \nonumber\\
&& \text{(by the claim)}. \nonumber
\end{eqnarray}
Thus, it remains to show $(P^p_k)' \cong P^p_k /G$. Consider the case when $p = 1$. Let $D$ be the divisor $A_1 + \cdots + A_k$. Then, we have
\begin{eqnarray}
(P^1_k)' \x_{Y/G} Y & \cong & \P(\, \gp^*(\O_{Y/G} \oplus \O_{Y/G}(D/G)) \,) \nonumber\\
& \cong & \P(\O_Y \oplus \O_Y(D)) \nonumber\\
& = & P^1_k. \nonumber
\end{eqnarray}
By Proposition \ref{ascentdescent}, we have $(P^1_k)' \cong P^1_k /G$. Similarly, $(P^p_k)' \cong P^p_k /G$ for $p =2,3$.
\end{proof}

When $X$ is an object in $\cat{D}$, there are some natural formulas relating the push-forward, pull-back and external product with their non-\equi\ versions.

\begin{prop}
\label{formulaswhenind}
Suppose $f : X' \to X$ is a \morp\ in $\cat{D}$.

\begin{statementslist}
{\rm (1)} & If $f$ is \proj, then $f/G$ is \proj, we have push-forward 

\begin{center}
$(f/G)_* : \go(X'/G) \to \go(X/G)$
\end{center}

\noindent and

\begin{center}
$f_* = \gPS \circ (f/G)_* \circ \gPS^{-1}$
\end{center} 
 
\noindent as \morp s from $\cobg{G}{}{X'}$ to $\cobg{G}{}{X}$. \\
{\rm (2)} & If $f$ is \sm\ and $X$, $X'$ are both \equidim, then $f/G$ is \sm, we have pull-back 

\begin{center}
$(f/G)^* : \go(X/G) \to \go(X'/G)$
\end{center}

\noindent and

\begin{center}
$f^* = \gPS \circ (f/G)^* \circ \gPS^{-1}$
\end{center}

\noindent as \morp s from $\cobg{G}{}{X}$ to $\cobg{G}{}{X'}$.
\end{statementslist}
\end{prop}

\begin{proof}
{\rm (1)}\tab First of all, $f/G$ is \proj\ by Proposition \ref{ascentdescent}. Also, $X/G$, $X'/G$ are both in $\smcat$. Hence, the push-forward $(f/G)_* : \go(X'/G) \to \go(X/G)$ is well-defined. Moreover, by definition,
\begin{eqnarray}
\gPS \circ (f/G)_* \circ \gPS^{-1}\ [Y \to X'] & = & \gPS \circ (f/G)_*\ [Y/G \to X'/G] \nonumber\\
                                               & = & \gPS\ [Y/G \to X/G] \nonumber\\
                                               & = & [Y/G \x_{X/G} X \to X] \nonumber\\
                                               & = & [Y \to X]. \nonumber
\end{eqnarray}

{\rm (2)}\tab By the descent of \sm ness (Proposition 2 of \cite{EdGr}), the \morp\ $f/G$ is \sm. Also, $X/G$, $X'/G \in \smcat$ are both \equidim. Hence, the pull-back $(f/G)^* : \go(X/G) \to \go(X'/G)$ is well-defined. Moreover, 
\begin{eqnarray}
\gPS \circ (f/G)^* \circ \gPS^{-1}\ [Y \to X] & = & \gPS \circ (f/G)^*\ [Y/G \to X/G] \nonumber\\
                                              & = & \gPS\ [Y/G \x_{X/G} X'/G \to X'/G] \nonumber\\
                                              & = & [Y/G \x_{X/G} X'/G \x_{X'/G} X' \to X'] \nonumber\\
                                              & = & [Y/G \x_{X/G} X' \to X'] \nonumber\\
                                              & = & [Y/G \x_{X/G} X \x_X X' \to X'] \nonumber\\
                                              & = & [Y \x_X X' \to X'] \nonumber\\
                                              &  & \text{by Proposition \ref{ascentdescent}.} \nonumber
\end{eqnarray}
\end{proof}

There is a also similar formula for the external product, which is somewhat harder to state. We need some trivial facts first.

Let $\gg : G \to H$ be a group scheme homomorphism between the group schemes $G$, $H$. Then, for all $X \in \gsmcat{H}$, it induces a natural abelian group homomorphism
$$\gPH_{\gg} : \cat{U}_H(X) \to \cat{U}_G(X)$$
by sending $[Y \to X]$ with $H$-actions to $[Y \to X]$ with $G$-actions via $\gg$. This \homo\ obviously respects GDPR, so $\gPH_{\gg}$ is well-defined.

Denote the ascending \homo\ corresponding to $G$-action as $\gPS_G : \go(X/G) \to \cobg{G}{}{X}$.

\begin{prop}
Suppose $X$, $X'$ are two objects in $\cat{D}$. Then, the external product 
$$\x : \cobg{G}{}{X} \x \cobg{G}{}{X'} \to \cobg{G}{}{X \x X'}$$ 
of the element $(a,b) \in \cobg{G}{}{X} \x \cobg{G}{}{X'}$ can be given by
$$a \x b = \gPH_{\gD} \circ \gPS_{G \x G} (\gPS_G^{-1} a \x  \gPS_G^{-1} b)$$
where $\gD : G \to G \x G$ is the diagonal \morp.
\end{prop}

\begin{proof}
Follows from the definition.
\end{proof}

\bigskip

\bigskip

\section{The Chern class operator $c(\L)$}
\label{chernclasssection}

Suppose $X$ is an object in $\gsmcat{G}$ and $\L$ is a \glin{G}\ invertible sheaf over $X$. Our goal in this section is to define an abelian group homo\morp 
$$c(\L) : \cobglow{G}{*}{X} \to \cobglow{G}{*-1}{X}$$ 
which satisfies some natural properties.

Recall that in section 4 of \cite{LePa}, when $\L$ is a \ggen\ invertible sheaf over a $k$-scheme $X \in \smcat$, $c(\L) : \go_*(X) \to \go_{*-1}(X)$ is defined as follow. Let $[f : Y \to X]$ be an element in $\go(X)$ such that $Y$ is irreducible. Since $f^*\L$ is a \ggen\ invertible sheaf over $Y$, there is a \sm\ divisor $H$ on $Y$ such that $\O_Y(H) \cong f^*\L$. Then, we define $c(\L)[f : Y \to X] \defeq [H \embed Y \to X]$.

It is natural to try to give a similar version in our \equi\ setting. However, since there is no assumption on how the group $G$ acts on the scheme $X$, there is no guarantee that even a single non-zero invariant global section of $\L$ can be found. For example, if the action on $X$ is transitive, then no matter how nice a \glin{G}\ invertible sheaf $\L$ over $X$ is, there is no invariant global section that cuts out an invariant divisor. Hence, $c(\L)[\id : X \to X]$ can not be defined in a similar manner. 

Moreover, even if there is an invariant section cutting out a \sm\ \inv\ divisor, it may not be generic. For example, take $G \defeq GL(2)$ and $X \defeq \P^2$ with action 
$$\left(
\begin{array}{cc}
a & b \\
c & d
\end{array}
\right)
\cdot
\left(
\begin{array}{c}
x \\
y \\
z
\end{array}
\right)
\defeq
\left(
\begin{array}{ccc}
a & b & 0 \\
c & d & 0 \\
0 & 0 & 1
\end{array}
\right)
\left(
\begin{array}{c}
x \\
y \\
z
\end{array}
\right)$$
Consider the case when $\L = \O(1)$, which is naturally $G$-linearized. Then, there is only one invariant section $s \in {\rm H}^0(X, \L)^G$ that cuts out an \inv\ divisor, namely $s = z$. In this case, for a \proj\ map $f : Y \to X$, we can not define $c(\L)[f : Y \to X]$ by $f^*(s)$ because there is no reason to believe that $H_{f^* s}$ (the subscheme cut out by $f^*s$) will be \sm, or even a divisor. So, it is important that the choice of section is generic. Indeed, we will see later that this freedom of choice is essential for the well-definedness of our Chern class operator.

\vtab

\subsection{First approach}

As pointed out in the subsection \ref{freeactionisosubsection}, the theory $\cob{G}$ works nicely in the subcategory $\cat{D}$. Hence, our first approach is to restrict to this subcategory and define the Chern class operator. We first need a little lemma to ensure we stay inside the \qproj\ setup.

\begin{lemma}
\label{lemmaqprojvbundle}
If $X$ is \qproj\ over $k$ and $\pi : E \to X$ is a vector bundle, then $E$ is \qproj\ over $k$.
\end{lemma}

\begin{proof}
Consider $\P(\dual{\cat{E}} \oplus \O_X)$ where $\cat{E}$ is the locally free sheaf over $X$ corresponding to $E$. Since $\P(\dual{\cat{E}} \oplus \O_X) \to X$ is \proj, the scheme $\P(\dual{\cat{E}} \oplus \O_X)$ is \qproj. Then, $E$ can be considered as an open set inside $\P(\dual{\cat{E}} \oplus \O_X)$, hence is \qproj.
\end{proof}

Here is the natural definition of $c(\L)$ when $X$ is in $\cat{D}$.

\begin{defn}
{\rm
Suppose $X$ is an object in $\cat{D}$ and $\L$ is a sheaf in $\picard{G}{X}$. We define the Chern class operator $c(\L) : \cobglow{G}{*}{X} \to \cobglow{G}{*-1}{X}$ by 
$$c(\L) \defeq \gPS \circ c(\gp_* \L^G) \circ \gPS^{-1}$$
where $\gp : X \to X/G$ is the quotient map and $\gPS : \go(X/G)\iso\cobg{G}{}{X}$ is the ascent iso\morp\ defined in subsection \ref{freeactionisosubsection}.
}\end{defn}

Since $X$ is in $\cat{D}$, the sheaf $\gp_* \L^G$ over $X/G$ is invertible. Hence, the abelian group homo\morp\ $c(\gp_* \L^G) : \go_*(X/G) \to \go_{*-1}(X/G)$ is well-defined (see sections 4 and 9 in \cite{LePa} for more detail).

\begin{rmk}
\rm{
\label{rmkchernclass}
For $X \in \cat{D}$ and $\L \in \picard{G}{X}$ such that $\L$ is \ggen\ by invariant sections, we can construct $c(\L)[f : Y \to X]$ by following the definitions of $\gPS$ and $c(\gp_* \L^G)$.

First, descend $Y \to X$ to get $Y/G \to X/G$. Then, $(f/G)^*(\gp_* \L^G)$ will be a \ggen\ invertible sheaf over $Y/G$ (A \glin{G}\ invertible sheaf $\L$ being \ggen\ by invariant sections is equivalent to $\gp_*\L^G$ being \ggen). Pick a global section $s \in {\rm H}^0(Y/G, (f/G)^*(\gp_* \L^G))$ that cuts out a \sm\ divisor $H_s$ on $Y/G$. Then, ascend $H_s \to Y/G \to X/G$ to obtain $[H_s \x_{X/G} X \to X]$. Thus,
$$c(\L)[f : Y \to X] = [H_s \x_{X/G} X \to X].$$

It can be seen that $c(\L)[f : Y \to X]$ can also be obtained in the following way. Since $\L$ is \ggen\ by invariant sections, $f^*\L$ is also \ggen\ by invariant sections. Pick a section $s' \in{\rm H}^0(Y, f^*\L)^G$ that cuts out an invariant \sm\ divisor $H_{s'}$ on $Y$. Then, 
$$c(\L)[f : Y \to X] = [H_{s'} \to Y \to X].$$
}
\end{rmk}

\vtab

Because of the natural isomorphism between $\cobg{G}{}{X}$ and $\go(X/G)$ when $X$ is in $\cat{D}$, we can now easily show the \equi\ versions of some properties of the Chern class operator listed in \cite{LePa}, namely \textbf{(A3)-(A5)}, \textbf{(A8)}, \textbf{(Dim)}, etc.

\vtab

\subsection{Second approach}
\label{secondapproachsubsection}

Instead of imposing a restriction on $X$, we may impose a restriction on $\L$. Our second approach is to first define the notion of a ``\nice'' \glin{G}\ invertible sheaf. Then, we define the Chern class operator for ``\nice'' sheaves $\L$ and extend this definition to more general \glin{G}\ invertible sheaves through the formal group law.

Before proceeding to describe this second approach, let us recall the definition of the \fgl\ and some basic properties.

We denote the Lazard ring by $\lazard$ (see section 1.1 in \cite{LeMo}). Let $\{a_{ij}\}$ with $i,j \geq 0$ and $(i,j) \neq (0,0)$ be the standard set of generators of the Lazard ring, i.e. $\lazard = \Z[a_{ij}]$. Then, the \fgl\ $F$ is the power series in $\lazard[[u,v]]$ :
$$F(u,v) = \sum_{i,j \geq 0} a_{ij} u^i v^j = u + v + \sum_{i,j \geq 1} a_{ij} u^i v^j$$
(see section 2.4.3 in \cite{LeMo}). To help our intuition, we will think of the \fgl\ as giving ``addition''. By definition, we have 
\begin{eqnarray}
F(u,0)       & = & u. \nonumber\\
F(u,v)       & = & F(v,u). \nonumber\\
F(u, F(v,w)) & = & F(F(u,v),w) \nonumber
\end{eqnarray}
and the relations on $a_{ij}$ are the ones imposed by these equalities. 

Moreover, there is a power series $\gc(u) \in \lazard[[u]]$ that satisfies
$$F(u,\gc(u)) = 0.$$
The power series $\gc(u)$ can be regarded as giving the ``inverse'' of $u$. Hence, we can define ``subtraction'' by
$$F^-(u,v) \defeq F(u, \gc(v)).$$
For our purpose, we also need the notion of ``multiplication by a positive integer'' :
$$F^n(u) \defeq F(u, F(u, \cdots F(u,u) \cdots ))$$ 
\begin{center}
($n-1$ times application of $F$)
\end{center}
Finally, we will need the notion ``division by a positive integer''. For simplicity, denote $\lazard \otimes_{\Z} \Z[\frac{1}{n}]$ by $\lazard_n$. The Lazard's Theorem states that $\lazard$ is a polynomial algebra over integers with infinitely many generators (see \cite{lazardthm}). In particular, $\lazard$ has no torsion and $\lazard \embed \lazard_n$.

\begin{lemma}
For all $n \geq 1$, there exists a power series in $\lazard_n[[u]]$, denoted by $\fgldiv{n}(u)$, such that 
$$\fgldiv{n}(F^n(u)) = F^n(\fgldiv{n}(u)) = u.$$

\end{lemma}

\begin{proof}
Let $F^n(u) \defeq \sum_{i \geq 1} a_i u^i$ for some $a_i \in \lazard$.

\vtab

Claim : $a_1 = n$.

We proceed by induction on $n$. Obviously, the claim is true for $n = 1$. Suppose the claim is true for $n-1$. Notice that we can always ignore terms with degree of $u$ greater than 1. Hence,
\begin{eqnarray}
F^n(u) &=& F(u, F^{n-1}(u)) \nonumber\\
&=& u + F^{n-1}(u) + \text{higher degree terms} \nonumber\\
&=& u + (n-1)(u) + \cdots  \nonumber\\
&=& nu + \cdots. \nonumber
\end{eqnarray}
That proves the claim. \claimend

\vtab

Let $\fgldiv{n}(u) \defeq \sum_{i \geq 1} b_i u^i \in \lazard_n[[u]]$ with coefficients $\{b_i\}$ yet to be determined. The equality we want is 
$$u = \fgldiv{n}(F^n(u)) = b_1 (a_1 u + a_2 u^2 + \cdots) + b_2 (a_1 u + a_2 u^2 + \cdots)^2 + \cdots.$$
That gives us the following set of equations :
\begin{eqnarray}
1 & = & b_1 a_1 \nonumber\\
0 & = & b_1 a_2 + b_2 a_1^2 \nonumber\\
0 & = & b_1 a_3 + b_2 2 a_1 a_2 + b_3 a_1^3 \text{, etc.} \nonumber
\end{eqnarray}
Thus, we have $b_1 = 1 / a_1 = 1 / n \in \lazard_n$. After $b_1, \ldots, b_{i-1}$ are determined, we can define $b_i \in \lazard_n$ by the equation \wrt\ $u^i$ and the fact that the term corresponding to $b_i$ is just $b_i a_1^i = n^i b_i$. That gives us a power series $\fgldiv{n}(u) \in \lazard_n[[u]]$ such that $u = \fgldiv{n}(F^n(u))$.

To show the second equality $F^n(\fgldiv{n}(u)) = u$, let $F^n(\fgldiv{n}(u)) \defeq \sum_{i \geq 1} c_i u^i$. Then, 
\begin{eqnarray}
b_1 u + b_2 u^2 \cdots & = & \fgldiv{n}(u) \nonumber\\
                       & = & \fgldiv{n}(F^n(\fgldiv{n}(u))) \nonumber\\
                       & = & \fgldiv{n}(\sum_{i \geq 1} c_i u^i) \nonumber\\
                       & = & b_1 (c_1 u + c_2 u^2 + \cdots) + b_2 (c_1 u + c_2 u^2 + \cdots)^2 + \cdots. \nonumber
\end{eqnarray}
By comparing the coefficients, we obtain the following set of equations :
\begin{eqnarray}
b_1 & = & b_1 c_1 \nonumber\\
b_2 & = & b_1 c_2 + b_2 c_1^2 \nonumber\\
b_3 & = & b_1 c_3 + b_2 2 c_1 c_2 + b_3 c_1^3 \text{, etc.} \nonumber
\end{eqnarray}
Since $b_1 = \frac{1}{n}$, the first equation implies $c_1 = 1$. Substituting $c_1 = 1$ into the second equation implies that $c_2 = 0$. Inductively, $c_i = 0$ for all $i \geq 2$. Hence, $F^n(\fgldiv{n}(u)) = u$.
\end{proof}

\begin{rmk}
\rm{
By examining the proof carefully, it can be shown that if $\fgldiv{n}(u) = \sum_{i \geq 1}\ b_i u^i$, then $n^{i(i+1) / 2} b_i \in \lazard$.
}
\end{rmk}

\vtab

As mentioned at the beginning of this subsection, we will start by defining the notion of a \nice\ $G$-\equi\ invertible sheaf.

\begin{defn}
{\rm
Suppose $X$ is an object in $\gsmcat{G}$ and $\L$ is a sheaf in $\picard{G}{X}$. We say that $\L$ is \nice\ if there exists a \morp\ in \gsmcat{G}, $\gps : X \to \P^n$ (with trivial $G$-action on $\P^n$) such that $\L \cong \gps^* \O(1)$.
}\end{defn}

Here are some basic properties.

\begin{lemma}
Suppose $X$ is an object in $\gsmcat{G}$.

\begin{tabular}{ll}
1. & The structure sheaf $\O_X$ is \nice. \\
2. & If the sheaves $\L$, $\L' \in \picard{G}{X}$ are both \nice, then $\L \otimes \L'$ is also \nice. \\
3. & If $f : X \to Y$ is a \morp\ in $\gsmcat{G}$ and $\L \in \picard{G}{Y}$ is \nice, then $f^* \L$ is \nice.
\end{tabular}
\end{lemma}

\begin{proof}
1.\tab By considering the map $\gps : X \to \P^0 \cong \pt.$

2.\tab Suppose we have two \morp s $\gps : X \to \P^n$ and $\gps' : X \to \P^m$ such that $\gps^* \O(1) \cong \L$ and $\gps'^* \O(1) \cong \L'$. Let $\gps''$ be the following composition :
\begin{center}
$\begin{CD}
X @>{\gps \x \gps'}>> \P^n \x \P^m @>{Segre}>> \P^N.
\end{CD}$
\end{center}
Then, $\gps''^* \O(1) \cong \L \otimes \L'$.

3.\tab By definition.
\end{proof}

We will start with a definition of the Chern class operator which depends on $\gps$. Suppose that $\L$ is a sheaf in $\picard{G}{X}$ and there is a map $\gps : X \to \P^n$ such that $\gps^* \O(1) \cong \L$. We would like to define
$c_{\gps}(\L)[f : Y \to X]$ as $[Y \x_{\P^n} H \to Y \to X]$ where $H$ is a hyperplane in $\P^n$ such that $Y \x_{\P^n} H$ is a \sm\ \inv\ divisor on $Y$. Clearly, it is enough to consider the case when $Y$ is \girred. In what follows, we will show that this is well-defined, i.e. that such an $H$ exists, that this element is independent of the choice of $H$ and that the construction respects GDPR.

\begin{lemma}
\label{bertinilemma}
Denote the dual \proj\ space $\P( {\rm H}^0(\P^n, \O(1)) )$ by $(\P^n)^*$. Then, there is a non-empty open set $U$ in $(\P^n)^*$ such that for any section $s$ in $U$, the closed subscheme $Y \x_{\P^n} H \subset Y$, where $H$ is the hyperplane in $\P^n$ cut out by the section $s$, is a \sm\ \inv\ divisor on $Y$ .
\end{lemma}

\begin{proof}
This is a variation of the Bertini's Theorem when $\char{k} = 0$. We have $f : Y \to X$ and $\gps : X \to \P^n$ as above. Let $\cat{H}$ be the analog of the universal Cartier divisor, i.e.
$$\cat{H} \defeq \{\,(y, s)\ |\ s(\gps \circ f(y)) = 0\,\} \subset Y \x (\P^n)^*.$$

\vtab

Claim : $\cat{H}$ is \sm\ and of dimension $\dim Y + n - 1$.

Let $\P^n = {\rm Proj}\ k[x_0, \ldots, x_n]$ and $(\P^n)^* = {\rm Proj}\ k[c_0, \ldots, c_n]$. Let $D(x_i)$ be the affine open subscheme of $\P^n$ given by $x_i \neq 0$ and similarly for $D(c_i)$. Also, let $\spec{A}$ be an affine open subscheme of $(\gps \circ f)^{-1}(D(x_i))$. Then, $\gps \circ f$ is locally given by a map $\spec{A} \to \spec{k[x_0/x_i, \ldots, x_n/x_i]}$, which corresponds to sending the elements $x_j/x_i$ to some elements $a_j \in A$. So, the universal Cartier divisor $\cat{H}$ is locally given by the equation $\sum_{j \neq i} (c_j/c_i) a_j = 0$ inside $\spec{A} \x D(c_i)$. Hence, the claim is true. \claimend

\vtab

Consider the projection $\cat{H} \to (\P^n)^*$. For a section $s \in (\P^n)^*$, the fiber is exactly $Y \x_{\P^n} H$ where $H$ is the hyperplane cut out by $s$. Hence, the open set we want will be the set of regular values of this projection map.
\end{proof}

\begin{lemma}
Let $s$, $s'$ be two sections in $(\P^n)^*$, cutting out $H$, $H'$ respectively, such that $Y \x_{\P^n} H$ and $Y \x_{\P^n} H'$ are both \sm\ \inv\ divisors on $Y$. Then we have 
$$[Y \x_{\P^n} H \to X] = [Y \x_{\P^n} H' \to X]$$
as elements in $\cobg{G}{}{X}$.
\end{lemma}

\begin{proof}
Observe that $H$, $H'$ are \equi ly linearly equivalent divisors on $\P^n$. Thus,
$$Y \x_{\P^n} H = (\gps \circ f)^*H \sim (\gps \circ f)^*H' = Y \x_{\P^n} H'$$
as \inv\ divisors on $Y$. The result then follows from $GDPR(1,1)$.
\end{proof}
	
\begin{lemma}
Sending $[Y \to X]$ to $[Y \x_{\P^n} H \to X]$ defines an abelian group homo\morp\ from $\cobglow{G}{*}{X}$ to $\cobglow{G}{*-1}{X}$.
\end{lemma}

\begin{proof}
As before, let $\cat{G}$ be the map corresponding to a GDPR setup $Y \to X$ with divisors \gdprdivisors\ on $Y$. We need to show
$$c_{\gps}(\L) \circ \cat{G}(G^X_{n,m}) = c_{\gps}(\L) \circ \cat{G}(G^Y_{m,n}).$$

For simplicity, we will denote $X \x_{\P^n} H$ by $X_H$. Consider the \proj\ \morp\ $Y_H \to X_H$. By the freedom of choice of $H$, we may assume $X_H$ is a \sm\ \inv\ divisor on $X$ and the same for $Y_H$. In particular, $Y_H$, $X_H$ are both in $\gsmcat{G}$ and $Y_H$ is \equidim. Similarly, we may assume the same property holds for ${A_i}_H$ and ${B_j}_H$ and also, 
$${A_1}_H + \cdots + {A_n}_H + {B_1}_H + \cdots + {B_m}_H$$ 
is a \rsncd\ on $Y_H$. Since the divisors are given by pull-back along $Y_H \to Y$, we have 
$${A_1}_H + \cdots + {A_n}_H \sim {B_1}_H + \cdots + {B_m}_H.$$ 
Thus, we can define a map $\cat{G}' : \cat{R} \to \cobg{G}{}{X_H}$ by the GDPR setup $Y_H \to X_H$ with ${A_1}_H + \cdots + {A_n}_H \sim {B_1}_H + \cdots + {B_m}_H$. So, it is enough to show
$$c_{\gps}(\L) \circ \cat{G} = i_* \circ \cat{G}'$$
where $i : X_H \embed X$.

For a general term $\generalterm$,
\begin{eqnarray}
c_{\gps}(\L) \circ \cat{G}(\generalterm) &=& c_{\gps}(\L)[A_i \x_Y \cdots \x_Y P^p_k \x_Y \cdots \to X] \nonumber\\
&=& [(A_i \x_Y \cdots \x_Y P^p_k \x_Y \cdots )_H \to X] \nonumber\\
&=& [{A_i}_H \x_{Y_H} \cdots \x_{Y_H} {(P^p_k)}_H \x_{Y_H} \cdots \to X_H \to X]. \nonumber
\end{eqnarray}
Hence, it is enough to show ${(P^p_k)}_H$ is the same as the corresponding tower given by \inv\ divisors $\{{A_i}_H\}$. The $p=1$ case follows from the fact that 
$$\P(\O_Y \oplus \O_Y(D))_H \cong \P(\O_{Y_H} \oplus \O_{Y_H}(D_H))$$ 
and the $p=2,3$ cases can be proved similarly. That shows the well-definedness of the \homo. The fact that it sends $\cobglow{G}{*}{X}$ to $\cobglow{G}{*-1}{X}$ is clear.
\end{proof}

Hence, we have the following definition.

\begin{defn}
{\rm
Suppose that $\L$ is a sheaf in $\picard{G}{X}$ such that there exists an \equi\ \morp\ $\gps : X \to \P^n$ with $\gps^*\O(1) \cong \L$. We define the Chern class operator $c_{\gps}(\L) : \cobglow{G}{*}{X} \to \cobglow{G}{*-1}{X}$ by 
$$c_{\gps}(\L)[f : Y \to X] \defeq [Y \x_{\P^n} H \to Y \to X]$$
where $H$ is a hyperplane in $\P^n$ such that $Y \x_{\P^n} H$ is an invariant \sm\ divisor on $Y$.
}\end{defn}

We definitely do not want the definition of the Chern class operator to depend on the particular \morp\ $\gps : X \to \P^n$. 

\begin{lemma}
$c_{\gps}(\L)$ is independent of $\gps$.
\end{lemma}

\begin{proof}
Suppose we have two \equi\ \morp s $\gps_1 : X \to \P^n$ and $\gps_2 : X \to \P^m$ such that $\gps_1^* \O(1) \cong \L \cong \gps_2^* \O(1)$. Consider the pull-back of sections 
$$\gps_1^* : {\rm H}^0(\P^n,\O(1)) \to {\rm H}^0(X,\L).$$ 
Then, the image of $\gps_1^*$ will lie in ${\rm H}^0(X, \L)^G$ and the same for $\gps_2$. Let $\{s_{1i}\}$ be a $k$-basis for ${\rm H}^0(\P^n, \O(1))$ and $\{s_{2j}\}$ be a $k$-basis for ${\rm H}^0(\P^m, \O(1))$. Then, $\span{k}{\gps_1^* s_{1i},\ \gps_2^* s_{2j}}$ will be a finite dimensional vector space in ${\rm H}^0(X,\L)^G$. In addition, it is base-point free. This defines an \equi\ morphism $\gps_3 : X \to \P^N$ which can be factored as $X \to \P^n \embed \P^N$ or $X \to \P^m \embed \P^N$. Also, $\gps_3^*\O(1) \cong \L$. Thus, it is enough to show $c_{\gps_1}(\L) = c_{\gps_3}(\L)$.

Consider an element $[Y \to X]$ in $\cobg{G}{}{X}$. Pick a hyperplane $H \subset \P^N$ such that $\P^n \cap H$ is a hyperplane in $\P^n$ (this is equivalent to $\P^n \x_{\P^N} H$ being a \sm\ divisor on $\P^n$) and $Y \x_{\P^N} H$ is a \sm\ divisor on $Y$. Then,
\begin{eqnarray}
c_{\gps_1}(\L)[Y \to X] &=& [Y \x_{\P^n} (\P^n \cap H)] \nonumber\\
&=& [Y \x_{\P^N} H] \nonumber\\
&=& c_{\gps_3}(\L)[Y \to X]. \nonumber
\end{eqnarray}
\end{proof}

Hence, for a \nice\ \glin{G}\ invertible sheaf $\L$ over $X \in \gsmcat{G}$, we have a natural definition of the Chern class operator
$$c(\L) : \cobglow{G}{*}{X} \to \cobglow{G}{*-1}{X}.$$

\vtab

\subsection{Special pull-back and the formal group law}

Recall that in the $\go_*$ theory in \cite{LePa}, we have the following property (Proposition 9.4 in \cite{LePa}).

For any $X \in \smcat$ and invertible sheaves $\L$, $\cat{M}$ over $X$, we have
$$c(\L \otimes \cat{M}) = F(c(\L), c(\cat{M}))$$
as abelian group endo\morp s on $\go(X)$ where $F \in \lazard[[u,v]] \cong \go(\pt)[[u,v]]$ is the \fgl\ with $\go(X)$ considered as a $\go(\pt)$-module by the external product. Since the Chern class operator always cuts down the dimension of the domain by one, $F(c(\L), c(\cat{M}))$ indeed acts as a finite sum on any given element in $\go(X)$.

We will follow the notation in \cite{LePa} and denote this property by \textbf{(FGL)}. Our objective in this subsection is to prove it holds in our \equi\ setting, when all \glin{G}\ invertible sheaves involved are \nice. First of all, we will need some basic facts.

\begin{prop}
\label{equiprojmorp}
Suppose $f: Y \to X$ is a \morp\ in $\gsmcat{G}$. Then, there exists a $G$-\repn\ $V$ and an \equi\ immersion $i : Y \embed \P(V) \x X$ such that $f = \gp_2 \circ i$. If we further assume $f$ to be \proj, then $i$ will be a closed immersion.
\end{prop}

\begin{proof}
First, assume that $G$ is reductive and connected. Since $Y$ is \qproj, there exists an (not necessarily \equi) immersion $i_0 : Y \embed \P^n$. Define $\L \defeq i_0^* \O(1)$ as an (not necessarily \glin{G}) invertible sheaf over $Y$. By Theorem 1.6 in \cite{Su}, there exists an integer $m$ such that $\L^m$ is $G$-linearizable. Fix a $G$-linearization of $\L^m$. Since we have a \glin{G}\ very ample invertible sheaf $\L$ over $Y$, by Proposition 1.7 in \cite{Mu}, there exists an \equi\ immersion $i_1 : Y \embed \P(V)$ for some $G$-\repn\ $V$ such that $i_1^* O(1) \cong \L^m$. Then, the map $i_1 \x f : Y \to \P(V) \x X$ will be the \equi\ immersion we want.

Now assume that $G$ is finite. As above, $\L = i_0^* \O(1)$ is a very ample invertible sheaf over $Y$. Then, $\otimes_{\ga \in G}\ \ga^*\L$ will be a \glin{G}\ very ample invertible sheaf over $Y$, which gives us the \equi\ immersion $i_1$.

If $f$ is \proj, then the image of $i = i_1 \x f$ will be a closed subscheme of $\P(V) \x X$.
\end{proof}

Suppose $X$ is a scheme over $k$ and $U$ is a subscheme of $X$. We will denote the closure of $U$ in $X$ by $\closure{U}{X}$. Also denote the singular locus of $X$ by $\singular{X}$.

\begin{prop}
\label{equiembedwithsmclosure}
(Equivariant immersion with \sm\ closure)

\begin{statementslist}
{\rm (1)} & If $Y$ is an object in $\gsmcat{G}$, then there exists a $G$-\repn\ $V$ where $Y$ can be equivariantly embedded into $\P(V)$ such that its closure is \sm. \\
{\rm (2)} & Suppose $X$, $Y$ are objects in $\gsmcat{G}$ and $U \subset X$ is an invariant open subscheme. If a \morp\ $f : Y \to U$ in $\gsmcat{G}$ is \equi\ and \proj, then there exists a $G$-\repn\ $V$, an \equi\ closed immersion $i : Y \embed U \x \P(V)$ such that $f = \gp_1 \circ i$, and $\closure{Y}{X \x \P(V)}$ is \sm.
\end{statementslist}
\end{prop}

\begin{proof}
{\rm (1)}\tab By Proposition \ref{equiprojmorp}, we may assume there exists an \equi\ immersion \tab$Y \embed \P(V')$ for some $G$-\repn\ $V'$. By the canonical resolution of singularities (Theorem 1.6 in \cite{embeddeddesingularization}), for any variety $Z$ over $k$ ($\char{k} = 0$), there exists a \sm\ variety $Z^{res}$ and a \morp\ $Z^{res} \to Z$ which is given by a series of blowups along canonically chosen \sm\ centers. As pointed out in Remarks 4-1-1 in \cite{Ma}, since the blowups are canonical, $Z^{res}$ has a natural $G$-action and $Z^{res} \to Z$ will be $G$-\equi. Apply this on our case by setting $Z \defeq \closure{Y}{\P(V')}$, then we have an \equi\ \morp\ $\gp : Z^{res} \to Z$.

First of all, since $Y$ is \sm, $\gp$ is an isomorphism away from $\singular{Z} \subset Z - Y$. That implies the \equi\ immersion $Y \embed Z$ lifts to an \equi\ immersion $Y \embed Z^{res}$ and $\closure{Y}{Z^{res}} = Z^{res}$. Moreover, $Z^{res}$ is \proj\ because $\gp$ is \proj\ and $Z$ is \proj. By Proposition \ref{equiprojmorp}, $Z^{res}$ can be \equi ly embedded into $\P(V)$ for some $G$-\repn\ $V$. Hence, we have $Y \embed Z^{res} \embed \P(V)$ such that $\closure{Y}{\P(V)} = Z^{res}$ is \sm.

\vtab

{\rm (2)}\tab Since $f : Y \to U$ is \proj, by Proposition \ref{equiprojmorp}, there exists an \equi\ immersion $i' : Y \embed U \x \P(V')$ for some $G$-\repn\ $V'$ such that $f = \gp_1 \circ i'$. Consider $U \x \P(V')$ as an invariant open subscheme in $X \x \P(V')$ and let $Z \defeq \closure{Y}{X \x \P(V')}$. By canonical resolution of singularities as above, we have an \equi\ \proj\ \morp\ $Z^{res} \to Z$. By considering
$$Z^{res} \to Z \embed X \x \P(V') \to X,$$
we know that $Z^{res} \to X$ is \equi\ and \proj. By Proposition \ref{equiprojmorp}, there exists an \equi\ immersion $Z^{res} \embed X \x \P(V)$ for some $G$-\repn\ $V$. Again, the \equi\ immersion $Y \embed Z$ lifts to $Y \embed Z^{res}$ and we have $Y \embed Z^{res} \embed X \x \P(V)$ where $\closure{Y}{X \x \P(V)} = Z^{res}$ is \sm. Consider the following commutative diagram :

\begin{center}
$\begin{array}{ccccc}
Z^{res}    & \embed & X \x \P(V)  &          & \\
\downarrow &        &             & \searrow & \\
Z          & \embed & X \x \P(V') & \to      & X. \\
\end{array}$
\end{center}
Consider its restriction over $U$. Then, we obtain the following commutative diagram :

\begin{center}
$\begin{array}{ccccc}
Z^{res}|_U \cong Y & \embed & U \x \P(V)  &          & \\
\downarrow &                &             & \searrow & \\
Z|_U \cong Y           & \embed & U \x \P(V') & \to      & U. \\
\end{array}$
\end{center}
That gives us an \equi\ closed immersion $i : Y \embed U \x \P(V)$ such that the closure $\closure{Y}{X \x \P(V)} = Z^{res}$ is \sm. Moreover, the composition $\gp_1 \circ i$ is given by
$$Y \iso Z^{res}|_U \iso Z|_U \cong Y \embed U \x \P(V') \to U,$$
which is $\gp_1 \circ i' = f$.
\end{proof}

\vtab

In order to prove the \textbf{(FGL)} property , we need some reduction of arguments, which requires the following special type of pull-back.
 
Let $\gps : X \to \prod_i \P^{n_i}$ be a $G$-\equi\ \morp\ where $X \in \gsmcat{G}$ is \equidim\ and the $G$-action on $\prod_i \P^{n_i}$ is trivial. We are going to define $\gps^* : \cobg{G}{}{\prod_i \P^{n_i}} \to \cobg{G}{}{X}$. Our proof is basically the \equi\ version of Lemma 6.1 in \cite{LePa}. Let $Q$ be the group scheme $\prod_i GL(n_i + 1)$ which acts on $\prod_i \P^{n_i}$ naturally. We consider $Q$ as a variety with trivial $G$-action, so $Q$ is in $\gsmcat{G}$.

\begin{lemma}
Let $f : Y \to \prod_i \P^{n_i}$ be a \proj\ \morp\ in $\gsmcat{G}$ such that $Y$ is \girred.

\begin{statementslist}
{\rm (1)} & There exists a non-empty open subscheme $U(\gps,f) \subset Q$ such that, for all closed points $\gb \in U(\gps,f)$, the \morp s $\gb \cdot \gps$ and $f$ are transverse. \\
{\rm (2)} & For any two closed points $\gb, \gb' \in U(\gps,f)$, we have 

\begin{center}
$[X \x_{\gb \cdot \gps} Y \to X] = [X \x_{\gb' \cdot \gps} Y \to X]$ 
\end{center}

\noindent as elements in $\cobg{G}{}{X}.$ \\
\end{statementslist}
\end{lemma}

\begin{proof}
{\rm (1)}\tab First of all, $\gb \cdot \gps$ is $G$-\equi\ because $\gb : \prod_i \P^{n_i}\iso\prod_i \P^{n_i}$ is trivially $G$-\equi. Define a map $Q \x X \to \prod_i \P^{n_i}$ by $(\gb, x) \mapsto \gb \cdot \gps(x)$, which is clearly $G$-\equi. In addition, since $Q$ acts on $\prod_i \P^{n_i}$ transitively, the map 
$$T_{\gb} Q \oplus T_x X = T_{(\gb,x)} (Q \x X) \to T_{\gb \gps(x)} (\prod_i \P^{n_i})$$
is surjective ($T_x X$ means the tangent space of $X$ at $x$). Since the domain and codomain are both \sm, by Proposition 10.4 in Ch. III in \cite{Ha} ($\char{k} = 0$), the map $Q \x X \to \prod_i \P^{n_i}$ is \sm. That implies $(Q \x X) \x_{\prod_i \P^{n_i}} Y$ is \sm.

Let $(Q \x X) \x_{\prod_i \P^{n_i}} Y \to Q$ be the projection. If a closed point $\gb \in Q$ is a regular value, then $((Q \x X) \x_{\prod_i \P^{n_i}} Y)_{\gb} = X \x_{\gb \cdot \gps} Y$ is \sm\ and
\begin{eqnarray}
\dim X \x_{\gb \cdot \gps} Y &=& \dim ((Q \x X) \x_{\prod_i \P^{n_i}} Y)_{\gb} \nonumber\\
& = & \dim (Q \x X) \x_{\prod_i \P^{n_i}} Y - \dim Q \nonumber\\
                                                  & = & \dim Q \x X + \dim Y - \dim \prod_i \P^{n_i} - \dim Q \nonumber\\
                                                  & = & \dim X + \dim Y - \dim \prod_i \P^{n_i}. \nonumber
\end{eqnarray}
In other words, $f$ and $\gb \cdot \gps$ are transverse. Hence, the open set $U(\gps,f)$ we want is just the set of regular values of $(Q \x X) \x_{\prod_i \P^{n_i}} Y \to Q$. 

\vtab

{\rm (2)}\tab Consider the following commutative diagram :

\vtab

$\begin{CD}
(Q \x X) \x_{\prod_i \P^{n_i}} Y @>>> Q \x X @>>> Q                      @>>> \A^N \\
@AAA                               @AAA        @AAA                        @AAA \\
(U \x X) \x_{\prod_i \P^{n_i}} Y @>>> U \x X @>>> {U \defeq Q \cap \A^1} @>>> {\A^1 \defeq \text{line through }\gb, \gb'} 
\end{CD}$

\vtab

\noindent where the group scheme $Q = \prod_i GL(n_i +1)$ is considered as an open subscheme of $\A^N$ for some large $N$ (trivial $G$-action on $\A^N$). Notice that $U$ is a non-empty open subscheme of $\A^1$. All maps in the diagram are trivially $G$-\equi. The \morp\ $(Q \x X) \x_{\prod_i \P^{n_i}} Y \to Q \x X$ is \proj\ because it is an extension from $f$. By using a smaller $U$ (as long as $U \subset U(\gps,f)$), we can assume the projection map 
$$(U \x X) \x_{\prod_i \P^{n_i}} Y \to U$$ 
to be \sm. Hence, $(U \x X) \x_{\prod_i \P^{n_i}} Y$ is \sm. Notice that the fibers are 
$$((U \x X) \x_{\prod_i \P^{n_i}} Y)_{\gb} = X \x_{\gb \cdot \gps} Y.$$

Denote the map 
$$Z \defeq (U \x X) \x_{\prod_i \P^{n_i}} Y \to U \x X$$ 
by $g$. Then, $g$ is a \proj\ \morp\ in $\gsmcat{G}$. In addition, $Z$ is \equidim\ because $U$ is \equidim\ and $Z \to U$ is \sm. By Proposition \ref{equiembedwithsmclosure}, there exists a $G$-\equi\ closed immersion $i : Z \embed (U \x X) \x \P(V)$ for some $G$-\repn\ $V$ such that $g = \gp_1 \circ i$ and the closure of $Z$ in $(\P^1 \x X) \x \P(V)$ is \sm. Let us denote this closure by $\overline{Z}$. Thus, we obtain a \proj\ \morp\ $\overline{Z} \to \P^1 \x X \to X$ in $\gsmcat{G}$ such that the fibers of $\overline{Z}$ over $\gb,\gb' \in \P^1$ agree with the fibers of $Z$ over $\gb,\gb'$, namely $\overline{Z}_{\gb} = Z_{\gb}$ and $\overline{Z}_{\gb'} = Z_{\gb'}$. Since $\gb,\gb'$ can be considered as $G$-invariant divisors on $\P^1$ and they are $G$-\equi ly linearly equivalent, we have $\overline{Z}_{\gb} \sim \overline{Z}_{\gb'}$, as $G$-invariant divisors on $\overline{Z}$. Hence, by $GDPR(1,1)$,
$$[X \x_{\gb \cdot \gps} Y \to X] = [\overline{Z}_{\gb} \to X] = [\overline{Z}_{\gb'} \to X] = [X \x_{\gb' \cdot \gps} Y \to X].$$
\end{proof}

We will define the special pull-back $\gps^* : \cobg{G}{*}{\prod_i \P^{n_i}} \to \cobg{G}{*}{X}$ by sending the element $[f : Y \to \prod_i \P^{n_i}]$ to $[X \x_{\gb \cdot \gps} Y \to X]$ with $\gb \in U(\gps,f)$. Its well-definedness is given by the following Lemma.

\begin{lemma}
Sending $[f : Y \to \prod_i \P^{n_i}]$ to $[X \x_{\gb \cdot \gps} Y \to X]$ defines an abelian group homomorphism from $\cobg{G}{*}{\prod_i \P^{n_i}}$ to $\cobg{G}{*}{X}$.
\end{lemma}

\begin{proof}
This proof is roughly the same as the proof of the well-definedness of $c_{\gps}(\L)$. We need to show it respects GDPR. This can be achieved by using the fact that the choice of $\gb$ in the group $Q$ is generic which is similar to the generic choice of $H$ in $\P^n$ in the other proof.

As before, let $\cat{G}$ be the map corresponding to a GDPR setup $\gph : Y \to \prod_i \P^{n_i}$ with $G$-invariant divisors \gdprdivisors\ on $Y$. Consider the following commutative diagram :
\squarediagramword{Y' \defeq Y \x_{\prod_i \P^{n_i}} X}{Y}{X}{\prod_i \P^{n_i}}{(\gb \cdot \gps)'}{\gph'}{\gph}{\gb \cdot \gps}
By picking $\gb \in U(\gps, \gph)$, we may assume that $Y'$ is \sm\ and of dimension 
$$\dim X + \dim Y - \dim \prod_i \P^{n_i}.$$ 
Similarly, there is a non-empty open subscheme $U \subset Q$ such that $A_i' \defeq (\gb \cdot \gps)'^{-1}(A_i)$ is a \sm\ \inv\ divisor on $Y'$ for all $\gb \in U$. By taking intersection with some more open subschemes, we may assume $A_1' + \cdots + A_n' + B_1' + \cdots + B_m'$ is a \rsncd\ on $Y'$ for all $\gb$ in some non-empty open subscheme $U' \subset Q$. The divisors are given by pull-back, so $A_1' + \cdots + A_n' \sim B_1' + \cdots + B_m'$. Thus, $\gph' : Y' \to X$ together with  $A_1', \ldots, A_n', B_1', \ldots, B_m'$ defines a GDPR setup over $X$. Denote its corresponding map by $\cat{G}'$.

For a general term $\generalterm$,
\begin{eqnarray}
\gps^* \circ \cat{G}(\generalterm) &=& \gps^* [A_i \x_Y \cdots \x_Y P^p_k \x_Y \cdots \to \prod_i \P^{n_i}] \nonumber\\
&=& [X \x_{\gb \cdot \gps} (A_i \x_Y \cdots \x_Y P^p_k \x_Y \cdots) \to X] \nonumber\\
&=& [(X \x_{\gb \cdot \gps} A_i) \x_{Y'} \cdots \x_{Y'} (X \x_{\gb \cdot \gps} P^p_k) \x_{Y'} \cdots \to X]. \nonumber\\
&=& [A_i' \x_{Y'} \cdots \x_{Y'} (X \x_{\gb \cdot \gps} P^p_k) \x_{Y'} \cdots \to X]. \nonumber
\end{eqnarray}
On the other hand,
$$\cat{G}'(\generalterm) = [A_i' \x_{Y'} \cdots \x_{Y'} (P^p_k)' \x_{Y'} \cdots \to X].$$
Observe that $X \x_{\gb \cdot \gps} P^p_k = Y' \x_Y P^p_k \cong (P^p_k)'$. Hence, $\gps^* \circ \cat{G} = \cat{G}'$.
\end{proof}

Hence, for any $G$-\equi\ \morp\ $\gps : X \to \prod_i \P^{n_i}$ such that $X$ is \equidim, we obtain a special pull-back 
$$\gps^* :\cobg{G}{*}{\prod_i \P^{n_i}} \to \cobg{G}{*}{X}$$
which sends $[f : Y \to \prod_i \P^{n_i}]$ to $[X \x_{\gb \cdot \gps} Y \to X]$ where $\gb$ is a closed point in $Q$ such that $\gb \cdot \gps $ and $f$ are transverse.

\vtab

Now we can proceed to the proof of \textbf{(FGL)}. Here are a few simple properties we will need.

\begin{lemma}
\label{candpullcomm}
Suppose $\gps : X \to \P^n \x \P^m$ is a \morp\ in $\gsmcat{G}$ such that $X$ is \equidim. Denote the sheaves $\gp_1^* \O_{\P^n}(1)$, $\gp_2^* \O_{\P^m}(1)$ and $\gp_1^* \O_{\P^n}(1) \otimes \gp_2^* \O_{\P^m}(1)$ by $\O(1,0)$, $\O(0,1)$ and $\O(1,1)$ respectively.

\begin{statementslist}
{\rm (1)} & If $\L$ is either $\O(1,0)$, $\O(0,1)$ or $\O(1,1)$, then $\L$ is \nice\ and 

\begin{center}
$\gps^* \circ c(\L) = c(\gps^* \L) \circ \gps^*$
\end{center}
as \morp s from $\cobg{G}{}{\P^n \x \P^m}$ to $\cobg{G}{}{X}$. \\
{\rm (2)} & The special pull-back $\gps^*$ is a $\cobg{G}{}{\pt}$-module homomorphism.
\end{statementslist}

\end{lemma}

\begin{proof}
{\rm (1)}\tab The sheaves $\O(1,0), \O(0,1)$ and $\O(1,1)$ are \nice\ by definition. The equalities follow immediately from our construction.

{\rm (2)}\tab Same reason as the usual \sm\ pull-back.
\end{proof}

\begin{lemma}
\label{candpushcomm}
Suppose $f : X \to X'$ is a \proj\ \morp\ in $\gsmcat{G}$ and $\L \in \picard{G}{X'}$ is a \nice\ invertible sheaf, then 
$$f_* \circ c(f^*\L) = c(\L) \circ f_*$$
as \morp s from $\cobg{G}{}{X} \text{ to } \cobg{G}{}{X'}.$
\end{lemma}

\begin{proof}
Let $[Y \to X]$ be an element in $\cobg{G}{}{X}$ and $\gps : X' \to \P^n$ be a \morp\ in $\gsmcat{G}$ such that $\gps^*\O(1) \cong \L$. Then,
\begin{eqnarray}
c(\L) \circ f_* [Y \to X] & = & c(\L) [Y \to X'] \nonumber\\
                          & = & [Y \x_{\P^n} H \to X'] \nonumber\\
                          &   & (\text{fiber product via the map } Y \to X \to X' \to \P^n). \nonumber
\end{eqnarray}
On the other hand,
\begin{eqnarray}
f_* \circ c(f^* \L) [Y \to X] & = & f_* [Y \x_{\P^n} H \to X] \nonumber\\
                              &   & (\text{fiber product via the map } Y \to X \to X' \to \P^n) \nonumber\\
                              & = & [Y \x_{\P^n} H \to X']. \nonumber
\end{eqnarray}
\end{proof} 

We are now ready to prove the \fgl\ property \textbf{(FGL)} of the Chern class operator for \nice\ \glin{G}\ invertible sheaves. As mentioned before, the \fgl\ is the power series $$F(u,v) = \sum_{i,j \geq 0} a_{ij} u^i v^j \in \lazard[[u,v]].$$ 
For \nice\ sheaves $\L, \cat{M} \in \picard{G}{X}$, we consider $F(c(\L),c(\cat{M}))$ as a \morp\ from $\cobglow{G}{*}{X}$ to $\cobglow{G}{*-1}{X}$ given by 
$$\sum_{i,j \geq 0} a_{ij} c(\L)^i \circ c(\cat{M})^j$$
where $a_{ij}$ are considered as elements in $\cobg{G}{}{\pt}$ via the maps
$$\lazard \cong \go(\pt) \cong \cobg{\{1\}}{}{\pt} \stackrel{\gPH_{\gg}}{\longto} \cobg{G}{}{\pt}$$
where $\gPH_{\gg}$ is induced by the group scheme \homo\ $\gg : G \to \{1\}$ (See definition of $\gPH_{\gg}$ in subsection \ref{freeactionisosubsection}. We will see that this is a ring embedding in Corollary \ref{embedlazardcor}). As in the non-\equi\ theory, the Chern class operator decreases the homological grading by one. Since we have $\cobglow{G}{i}{X} = 0$ when $i < 0$, the power series $\sum_{i,j \geq 0} a_{ij} c(\L)^i \circ c(\cat{M})^j$ indeed acts as a finite sum for any given element in $\cobglow{G}{}{X}$.

\begin{prop}
\label{fglonnicesheaves}
If $X$ is an object in $\gsmcat{G}$ and $\L$, $\cat{M} \in \picard{G}{X}$ are both \nice, then 
$$c(\L \otimes \cat{M}) = F(c(\L), c(\cat{M}))$$
as \morp s from $\cobglow{G}{*}{X} \text{ to } \cobglow{G}{*-1}{X}.$
\end{prop}

\begin{proof}
Since $[f : Y \to X] = f_* [\id_Y]$, by Lemma \ref{candpushcomm}, it is enough to prove the statement on the element $[\id_X]$ such that $X \in \gsmcat{G}$ is \equidim.

Let $\gps_1 : X \to \P^n$ and $\gps_2 : X \to \P^m$ be the maps such that $\gps_1^*\O(1) \cong \L$ and $\gps_2^*\O(1) \cong \cat{M}$. Let $\gps : X \to \P^n \x \P^m$ be the map defined by $\gps_1$ and $\gps_2$. Then,
\begin{eqnarray}
c(\L)[\id_X] & = & c(\gps^* \O(1,0)) \circ \gps^* [\id_{\P^n \x \P^m}] \nonumber\\
           & = & \gps^* \circ c(\O(1,0)) [\id_{\P^n \x \P^m}] \nonumber\\
           &   & \text{(by Lemma \ref{candpullcomm})}. \nonumber
\end{eqnarray}
The same holds for $\cat{M}$ and $\L \otimes \cat{M}$. Hence, we have 
$$c(\L \otimes \cat{M})[\id_X] = \gps^* \circ c(\O(1,1)) [\id_{\P^n \x \P^m}]$$
and
$$F(c(\L), c(\cat{M})) [\id_X] = \gps^* \circ F( c(\O(1,0)), c(\O(0,1)) ) [\id_{\P^n \x \P^m}].$$

Thus, without loss of generality, we can assume $X = \P^n \x \P^m$, $\L = \O(1,0)$ and $\cat{M} = \O(0,1).$ Notice that the $G$-actions on $X$,  $\L$, $\cat{M}$ and $\id_X$ are all trivial now. Let 
$$\gPH_{\gg} : \go(\P^n \x \P^m) \cong \cobg{\{1\}}{}{\P^n \x \P^m} \to \cobg{G}{}{\P^n \x \P^m}$$ 
be the abelian groups \homo\ induced by the group scheme \homo\ $\gg : G \to \{1\}$. By Proposition 9.4 in \cite{LePa}, \textbf{(FGL)} holds in the non-\equi\ theory $\go_*$. In particular,
\begin{eqnarray} 
c(\O(1,1))[\id_{\P^n \x \P^m}] = F( c(\O(1,0)), c(\O(0,1)) ) [\id_{\P^n \x \P^m}] \label{eqn8}
\end{eqnarray}
as elements in $\go(\P^n \x \P^m)$. Observe that, for $\L = \O(1,0)$, $\O(0,1)$ or $\O(1,1)$, we have
$$\gPH_{\gg} \circ c(\L)[\id_{\P^n \x \P^m}] = [H_s \embed \P^n \x \P^m] = c(\L)[\id_{\P^n \x \P^m}]$$ 
where $s \in {\rm H}^0(\P^n \x \P^m, \L)$ is a global section such that $H_s$ is a \sm\ divisor on $\P^n \x \P^m$. By applying $\gPH_{\gg}$ on equation (\ref{eqn8}), the same equality holds in $\cobg{G}{}{\P^n \x \P^m}$.
\end{proof}

\vtab

\subsection{Extending the definition}

In order to extend our definition to arbitrary \glin{G}\ invertible sheaves, we need to first consider the sheaf $\O(1) \in \picard{G}{\P(V)}$ for arbitrary $G$-\repn\ $V$. In the case when $G$ is a finite abelian group with exponent $e$, it turns out the only way to define $c(\O(1))$, so that the property \textbf{(FGL)} still holds, will force us to invert the element $e \in \Z$. Hence, we introduce the notation 

\begin{center}
$\coblocal{G}{X}{e} \defeq \cobg{G}{}{X} \otimes_{\Z} \Z[1/e].$
\end{center}

\begin{rmks}
\rm{
We will explain why we cannot expect a more general definition of $c(\L)$ that satisfies the \textbf{(FGL)} without inverting the exponent of the group. Let us consider the following example. Suppose $G$ is a cyclic group of order $p$ (prime) and the ground field $k$ contains a primitive $p$-th \rou\ $\gx$. Let $V \defeq \span{k}{x,y}$ with action $\ga \cdot x = \gx x$ and $\ga \cdot y = y$ where $\ga$ is a generator of $G$. Let $X \defeq \P(V)$. 

Suppose we have defined $c(\O(1)) : \cobg{G}{}{X} \to \cobg{G}{}{X}$ such that \textbf{(FGL)} holds. Then, we will have
$$c(\O(p))[\id_X] = c(\O(1)^{\otimes p})[\id_X] = F^p(c(\O(1))) [\id_X].$$
Notice that 
$$F(c(\O(1)), c(\O(1))) [\id_X] = 2\,c(\O(1))[\id_X] + a_{11} c(\O(1))^2[\id_X] + \cdots.$$ 
For any $i \geq 2$, the element $c(\O(1))^i [\id_X]$ lies in $\cobglow{G}{1-i}{X}$, which is zero because the dimension of $X$ is one. So, we have $F(c(\O(1)), c(\O(1))) [\id_X] = 2\,c(\O(1))[\id_X]$. Inductively, we get $F^p(c(\O(1))) [\id_X] = p\,c(\O(1))[\id_X]$. 

On the other hand, consider the $G$-\equi\ map $\gps : X \to \P^1$ (with trivial action on $\P^1$) given by $(x;y) \mapsto (x^p;y^p)$. Then, $\O_X(p) \cong \gps^*\O_{\P^1}(1)$. Hence, $\O_X(p)$ is \nice. By the definition of the Chern class operator for \nice\ \glin{G}\ invertible sheaves, 
$$c(\O(p))[\id_X] = [H_p \embed \P(V)] = [G \embed \P(V)]$$
where $H_p \cong G$ (the $k$-scheme of $p$ points with free $G$-action). Hence, by pushing down both equalities to $\cobg{G}{}{\pt}$, we obtain
\begin{eqnarray}
[G] = p\,a \label{eqn11}
\end{eqnarray}
where $a \defeq {\gp_k}_*( c(\O(1))[\id_X] )$ and $\gp_k : X \to \pt$. 

Let $[Z_1] - [Z_2]$ be a representative of $a \in \cobg{G}{}{\pt}$. Consider the group scheme homomorphism $\{1\} \to G$. It induces an abelian groups \homo
$$\gPH : \cobglow{G}{0}{\pt} \to \cobglow{\{1\}}{0}{\pt} \cong \go_0(\pt) \cong \Z.$$ 
That implies 
$$p\,(\gPH[Z_1] - \gPH[Z_2]) = \gPH(p\,a) = \gPH[G] = p$$
as elements in $\go_0(\pt)$. Since there is no torsion in $\go_0(\pt) \cong \Z$, we conclude that $\gPH[Z_1] - \gPH[Z_2] = 1$. On the other hand, since the order of the group $G$ is a prime and the dimension of $Z_1$ is zero, $Z_1 \cong \spec{A_t} \disjoint \spec{A_f}$ where the action on $A_t$ is trivial and the action on $A_f$ is free. Moreover, $A_t$ can be written as the product of $K_{t,i}$, where $K_{t,i}$ are finite field extensions of $k$. Similarly, $Z_2 \cong \spec{B_t} \disjoint \spec{B_f}$ and $B_t = \prod_j L_{t,j}$. By Lemma 2.3.4 in \cite{LeMo}, we have $[\spec{K}] = [K : k][\id_{\spec{k}}]$ as elements in $\go(\pt)$, where $[K : k]$ denotes the degree of the field extension. Hence, 
\begin{eqnarray}
1 = \gPH[Z_1] - \gPH[Z_2] = \sum_i\ [K_{t,i} : k] + \gPH[\spec{A_f}] - \sum_j\ [L_{t,j} : k] - \gPH[\spec{B_f}]. \label{eqn9}
\end{eqnarray}

Let us consider an \girred\ component $W$ of $\spec{A_f}$. It can either be $\spec{K}$ with free action, or the disjoint union of $p$ copies of $\spec{K}$ with $G$ permuting them. In the first case, 
$$\gPH[W] = \gPH[\spec{K}] = [K : k] = [K : K^G][K^G : k] = p[K^G:k].$$
In the second case,
\begin{center}
$\gPH[W] = \gPH[\spec{\prod_{i=1}^p K}] = p[K : k].$
\end{center}
\noindent Either case, $p$ divides $\gPH[W]$. Hence, $p$ divides $\gPH[\spec{A_f}]$. Similarly, $p$ divides $\gPH[\spec{B_f}]$.

Now, if we apply the fixed point map $\cat{F} : \cobg{G}{}{\pt} \to \go(\pt)$ on equation (\ref{eqn11}) (see section \ref{fixedpointmapsection} for details), we obtain
\begin{eqnarray}
0 &=& \cat{F}[G] \nonumber\\
&=& \cat{F}( p\,([Z_1]- [Z_2]) ) \nonumber\\
&=& p\,( \cat{F}([Z_1])- \cat{F}([Z_2]) ) \nonumber\\
&=& p\,([\spec{A_t}]- [\spec{B_t}]) \nonumber\\
&=& p\,(\sum_i\ [K_{t,i} : k] - \sum_j\ [L_{t,j} : k]). \nonumber
\end{eqnarray}
That implies
\begin{eqnarray}
0 = \sum_i\ [K_{t,i} : k] - \sum_j\ [L_{t,j} : k]. \label{eqn10}
\end{eqnarray}
Combining equations (\ref{eqn9}) and (\ref{eqn10}) and the fact that $p$ divides $\gPH[\spec{A_f}]$ and $\gPH[\spec{B_f}]$, we get a contradiction. 

Hence, it is impossible to define $c(\O(1))$ as an operator on $\cobg{G}{}{X}$ such that \textbf{(FGL)} holds. It can also be seen in this example that the natural definition of $c(\O(1))[\id_X]$ should be $(1/p) [H_p \embed X]$, as an element in $\coblocal{G}{X}{p}$.
}
\end{rmks}
\vtab

In order to simplify the calculation, we need a condition on $G$ and $k$ such that any irreducible $G$-\repn\ will be of dimension 1.

\begin{defn}
{\rm
We will say that the pair $(G,k)$ is split, if the group $G$ is finite abelian with exponent $e$ and the field $k$ contains a primitive $e$-th \rou.
}\end{defn}

\begin{lemma}
\label{pairsplit}
If the pair $(G,k)$ is split, then any irreducible $G$-\repn\ has dimension one.
\end{lemma}

\begin{proof}
Recall that we are assuming $\char{k} = 0$. We can easily see that when $(G,k)$ is split, we have $k[G] \cong \prod k$. The result then follows.
\end{proof}

\vtab

For the rest of this subsection, we assume that the pair $(G,k)$ is split. In this case, we can extend our definition of the Chern class operator to arbitrary \glin{G}\ invertible sheaves. In order to preserve the \textbf{(FGL)} property, we would like to define $c(\L)$ by the following formula :
$$\fgldiv{e}( F^-(c(\L^e \otimes \cat{M}), c(\cat{M})) )$$
where $\cat{M}$ is in $\picard{G}{X}$ such that $\L^e \otimes \cat{M}$ and $\cat{M}$ are both \nice\ (recall that $\L^e$ means $\L^{\otimes e}$ and $\fgldiv{e}(u)$ is the operation ``division by $e$'' in \fgl, see subsection \ref{secondapproachsubsection} for details). We need the following two Lemmas for its well-definedness.

\begin{lemma}
For any $\L \in \picard{G}{X}$, there exists an invertible sheaf $\cat{M} \in \picard{G}{X}$ such that $\L^e \otimes \cat{M}$ and $\cat{M}$ are both \nice.
\end{lemma}

\begin{proof}
Let us first consider the case when $X = \P(V)$ where $V$ is a $G$-\repn\ and 

\noindent $\L = \O(1)$. By Lemma \ref{pairsplit}, $X \cong {\rm Proj}\ k[x_0, \ldots, x_n]$ such that, for all $i$, $\span{k}{x_i}$ is a 1-dimensional $G$-\repn. Let $Y \defeq {\rm Proj}\ k[y_0, \ldots, y_n]$ with trivial action and 
$$\gps : X = {\rm Proj}\ k[x_0, \ldots, x_n] \to {\rm Proj}\ k[y_0, \ldots, y_n] = Y$$ 
be the morphism corresponding to the $k$-algebra \homo\ $k[y_0, \ldots, y_n] \to k[x_0, \ldots, x_n]$ defined by $y_i \mapsto x_i^e$. Since $e$ is the exponent of $G$, the map $\gps$ is $G$-\equi. Observe that this map can also be considered as an $e$-uple embedding followed by a linear projection on some $G$-\inv\ open subscheme. Hence, we have $\gps^* \O_Y(1) \cong \O_X(e)$. In other words, the sheaf $\O_X(e)$ is \nice.

For general $X \in \gsmcat{G}$ and $\L \in \picard{G}{X}$, by Proposition \ref{equiembedwithsmclosure}, there exists an \equi\ immersion $\gps : X \embed \P(V)$. For large enough $m$, the sheaf $\L \otimes \gps^* \O(m)$ will be very ample. By embedding $\P(V)$ into some larger $\P(V')$, we can assume $m = 1$. Since $\L \otimes \gps^* \O(1)$ is very ample and \glin{G}, by Proposition 1.7 in \cite{Mu}, there exists an \equi\ immersion $\gps' : X \embed \P(V'')$ such that $\gps'^* \O(1) \cong \L \otimes \gps^* \O(1)$. Hence, we have $\gps'^* \O(e) \cong \L^e \otimes \gps^* \O(e)$. Then, the result follows because $\gps^* \O(e)$ and $\gps'^* \O(e)$ are both \nice.
\end{proof}

\begin{lemma}
For any two sheaves $\cat{M}$, $\cat{M'} \in \picard{G}{X}$ such that $\cat{M}$, $\cat{M}'$, $\L^e \otimes \cat{M}$ and $\L^e \otimes \cat{M}'$ are all \nice, we have
$$\fgldiv{e}( F^-(c(\L^e \otimes \cat{M}), c(\cat{M})) ) = \fgldiv{e}( F^-(c(\L^e \otimes \cat{M}'), c(\cat{M}')) )$$
as homo\morp s from $\coblocal{G}{X}{e}$ to $\coblocal{G}{X}{e}$.
\end{lemma}

\begin{proof}
By the fact that all sheaves involved are \nice\ and Proposition \ref{fglonnicesheaves}, we have
$$F(c(\L^e \otimes \cat{M}), c(\cat{M'})) = c(\L^e \otimes \cat{M} \otimes \cat{M'}) = F(c(\L^e \otimes \cat{M'}), c(\cat{M})).$$
That implies
\begin{eqnarray}
F^-(c(\L^e \otimes \cat{M}), c(\cat{M})) & = & F^-(c(\L^e \otimes \cat{M'}), c(\cat{M'})) \nonumber\\
\fgldiv{e}( F^-(c(\L^e \otimes \cat{M}), c(\cat{M})) ) & = & \fgldiv{e}( F^-(c(\L^e \otimes \cat{M'}), c(\cat{M'})) ). \nonumber
\end{eqnarray}
\end{proof}

\begin{defn}
{\rm
Assume the pair $(G,k)$ is split. Suppose $X$ is in $\gsmcat{G}$ and $\L$ is in $\picard{G}{X}$. We define the abelian group homomorphism $c(\L) : \cobglow{G}{*}{X}[1/e] \to \cobglow{G}{*-1}{X}[1/e]$ by the following formula :
$$c(\L) \defeq \fgldiv{e}( F^-(c(\L^e \otimes \cat{M}), c(\cat{M})) )$$
where $\cat{M}$ is in $\picard{G}{X}$ such that $\L^e \otimes \cat{M}$, $\cat{M}$ are both \nice.
}\end{defn}

\begin{rmk}
\rm{
Suppose $\L \in \picard{G}{X}$ is \nice. In this new definition, we can pick $\cat{M}$ to be $\L$. Then, 
$$\fgldiv{e}( F^-(c(\L^e \otimes \L), c(\L)) ) = \fgldiv{e}(c(\L^e)) = c(\L).$$ 
That means the new definition is indeed a generalization of the definition of the Chern class operator for \nice\ \glin{G}\ invertible sheaves.
}
\end{rmk}

Suppose $X$ is an object in $\cat{D}$ and $\L \in \picard{G}{X}$. Then we have two definitions of the Chern class operator (as operators on $\coblocal{G}{X}{e}$), given by the first and second approach. The last part of this section is to show that they agree. 

\begin{lemma}
\label{lbascentdescent}
Suppose $X$ is an object in $\cat{D}$ and $\L$, $\cat{M}$ are sheaves in $\picard{G}{X}$. Let 

\noindent $\gp : X \to X/G$ be the quotient map. Then, we have
$$\gp_*(\L \otimes \cat{M})^G \cong \gp_*\L^G \otimes \gp_*\cat{M}^G.$$
For any two sheaves $\L$, $\cat{M} \in \picard{}{X/G}$, we have
$$\gp^*(\L \otimes \cat{M}) \cong (\gp^*\L) \otimes (\gp^*\cat{M}).$$
In other words, descent and ascent both commutes with tensor product.
\end{lemma}

\begin{proof}
The second statement follows from a basic property of pull-back.
For descent,
\begin{eqnarray}
&& \gp_*\L^G \otimes \gp_*\cat{M}^G \nonumber\\
& \cong & \gp_* (\gp^* (\gp_*\L^G \otimes \gp_*\cat{M}^G) )^G \nonumber\\
&& (\text{Since $X \to X/G$ is a principle $G$-bundle, there is a one-to-one correspondence} \nonumber\\
&& \text{between $\picard{G}{X}$ and $\picard{}{X/G}$ given by $\gp^*$ and $\gp_*(-)^G$.}) \nonumber\\
                                     & \cong & \gp_*( (\gp^* \gp_*\L^G) \otimes (\gp^* \gp_*\cat{M}^G) )^G \nonumber\\ 
                                     & \cong & \gp_*(\L \otimes \cat{M})^G. \nonumber
\end{eqnarray}
\end{proof}

Suppose the pair $(G,k)$ is split, $X$ is an object in $\cat{D}$ and $\L \in \picard{G}{X}$. Denote the corresponding Chern class operator defined by the first approach by $c'(\L)$, i.e.
$$c'(\L) = \gPS \circ c(\gp_* \L^G) \circ \gPS^{-1}$$
from $\coblocal{G}{X}{e}$ to $\coblocal{G}{X}{e}$. Also denote the corresponding Chern class operator defined by the second approach by $c''(\L)$, i.e.
$$c''(\L)[Y \to X] = [Y \x_{\P^n} H \to X]$$
when $\L$ is \nice\ (see subsection \ref{secondapproachsubsection} for details), and for general $\L \in \picard{G}{X}$,
$$c''(\L) = \fgldiv{e}( F^-(c''(\L^e \otimes \cat{M}), c''(\cat{M})) )$$
from $\coblocal{G}{X}{e}$ to $\coblocal{G}{X}{e}$ 
where $\cat{M}$ is in $\picard{G}{X}$ such that $\L^e \otimes \cat{M}$, $\cat{M}$ are both \nice.

\begin{prop}
For any $X \in \cat{D}$ and $\L \in \picard{G}{X}$, we have
$$c'(\L) = c''(\L)$$
as group homo\morp s from $\coblocal{G}{X}{e} \text{ to } \coblocal{G}{X}{e}.$
\end{prop}

\begin{proof}
If $\L \in \picard{G}{X}$ is \nice, then there is an \equi\ \morp\ $\gps : X \to \P^n$ such that $\gps^*\O(1) \cong \L$. By definition,
$$c''(\L)[f : Y \to X] = [Y \x_{\P^n} H \to X]$$ 
where $H$ is a hyperplance in $\P^n$ such that $Y \x_{\P^n} H$ is an invariant \sm\ divisor on $Y$. Let $s \in {\rm H}^0(\P^n, \O(1))$ be the global section that cuts out $H$. Then, $Y \x_{\P^n} H$ is cut out by the invariant global section $(\gps \circ f)^*s \in {\rm H}^0(Y, f^*\L)^G$. On the other hand, by remark \ref{rmkchernclass}, $c'(\L)[Y \to X]$ can also be given by the divisor cut out by any invariant global section $s' \in {\rm H}^0(Y, f^*\L)^G$ as long as the divisor is \sm. Hence, $c'(\L) = c''(\L)$ when $\L$ is \nice.

For general $\L \in \picard{G}{X}$, let $\fgldiv{e}(u) \defeq \sum_{i \geq 1} b_i\ u^i$ and $F^-(u,v) \defeq \sum_{j,k \geq 0} c_{jk}\ u^j v^k.$ Then, we have
\begin{eqnarray}
c''(\L) & = & \fgldiv{e}( F^-(c''(\L^e \otimes \cat{M}), c''(\cat{M})) ) \nonumber\\
& = & \sum_i\ b_i\, ( \sum_{j,k} c_{jk} c''(\L^e \otimes \cat{M})^j c''(\cat{M})^k )^i \nonumber\\
& = & \sum_i\ b_i\, ( \sum_{j,k} c_{jk} \gPS \circ c(\gp_* (\L^e \otimes \cat{M})^G)^j \circ c( \gp_*\cat{M}^G)^k \circ \gPS^{-1} )^i \nonumber\\
&   & (\text{the two definitions agree for \nice\ sheaves}) \nonumber
\end{eqnarray}
\begin{eqnarray}
& = & \gPS \circ ( \sum_i b_i\, ( \sum_{j,k}\ c_{jk}\ c(\gp_*(\L^e \otimes \cat{M})^G)^j\, c(\gp_*\cat{M}^G)^k\,)^i ) \circ \gPS^{-1} \nonumber\\
& = & \gPS \circ \fgldiv{e}( F^-(c(\gp_*(\L^e \otimes \cat{M})^G), c(\gp_*\cat{M}^G) ) ) \circ \gPS^{-1} \nonumber\\
& = & \gPS \circ \fgldiv{e}( F^-( c((\gp_*\L^G)^e \otimes \gp_*\cat{M}^G) , c(\gp_*\cat{M}^G) ) ) \circ \gPS^{-1} \nonumber\\
&   & (\text{by Lemma \ref{lbascentdescent}}) \nonumber\\
& = & \gPS \circ c(\gp_*\L^G) \circ \gPS^{-1} \nonumber\\
&   & (\text{\,because \textbf{(FGL)} holds in } \go(X/G)\,) \nonumber\\
& = & c'(\L). \nonumber
\end{eqnarray}
\end{proof}

\bigskip

\bigskip

\section{More properties for $\cob{G}$}

In this section, we will state and prove some more basic properties in our \equi\ algebraic cobordism theory $\cob{G}$, equipped with the Chern class operator for \nice\ \glin{G}\ invertible sheaves. Some properties are related to the Chern class operator. In that case, we will also prove them in the theory $\coblocaltheory{G}{e} \defeq \cob{G} \otimes_{\Z} \Z[1/e]$ for arbitrary \glin{G}\ invertible sheaves assuming that the pair $(G,k)$ is split (recall that $e$ is the exponent of $G$). The non-\equi\ version of these properties can be found in \cite{LePa}.

At this stage, we have established \proj\ push-forward \textbf{(D1)}, \sm\ pull-back \textbf{(D2)}, Chern class operator \textbf{(D3)} and external product \textbf{(D4)}. For convenience, we will briefly recall here some of the properties already shown in section \ref{basicpropertysection}.

\vtab

\textbf{(A1)} If $f : X \to X'$ and $g : X' \to X''$ are both \sm\ and $X$, $X'$, $X''$ are all \equidim, then
$$(g \circ f)^* = f^* \circ g^*.$$
Moreover, $\id^*$ is the identity \homo. 

\vtab

\textbf{(A2)} If $f : X \to Z$ is \proj\ and $g : Y \to Z$ is \sm\ such that $X$, $Y$, $Z$ are all \equidim, then we have $g^*f_* = f'_*g'^*$ in the pull-back square
\squarediagramword{X \x_Z Y}{X}{Y}{Z}{g'}{f'}{f}{g}

\vtab

\textbf{(A3)} If $f : X \to X'$ is \proj\ and $\L \in \picard{G}{X'}$ is \nice, then
$$f_* \circ c(f^*\L) = c(\L) \circ f_*$$
in the theory $\cob{G}$. Moreover, if the pair $(G, k)$ is split, then the same statement holds in the theory $\coblocaltheory{G}{e}$ for arbitrary $\L \in \picard{G}{X'}$.

\begin{proof}
The first part of the statement follows from Lemma \ref{candpushcomm}. For the second part,
\begin{eqnarray}
f_* \circ c(f^*\L) & = & f_* \circ \fgldiv{e}( F^-( c(f^*\L^e \otimes f^*\cat{M}), c(f^*\cat{M}) ) ) \nonumber\\
                   &   & (\text{for some } \cat{M} \in \picard{G}{X'} \text{ such that } \L^e \otimes \cat{M},\ \cat{M} \text{ are both \nice}) \nonumber\\
                   & = & f_* \circ \sum_i b_i\,( \sum_{j,k} c_{jk}\, c(f^*\L^e \otimes f^*\cat{M})^j\, c(f^*\cat{M})^k )^i \nonumber\\
                   &   & \text{where $b_i$, $c_{jk}$ are coefficients for $\fgldiv{e}(u)$, $F^-(u,v)$ respectively} \nonumber\\
& = & (\sum_i b_i\,( \sum_{j,k} c_{jk}\, c(\L^e \otimes \cat{M})^j\, c(\cat{M})^k )^i) \circ f_* \nonumber\\
&   & \text{(by Lemma \ref{candpushcomm} and the fact that $\L^e \otimes \cat{M}$, $\cat{M}$ are \nice).} \nonumber
\end{eqnarray}
Hence, 
$$f_* \circ c(f^*\L) = \fgldiv{e}( F^-( c(\L^e \otimes \cat{M}), c(\cat{M})) ) \circ f_* = c(\L) \circ f_*.$$
\end{proof}

\vtab

\textbf{(A4)} If $f : X \to X'$ is \sm, $X$, $X'$ are both \equidim\ and $\L \in \picard{G}{X'}$ is \nice, then
$$f^* \circ c(\L) = c(f^*\L) \circ f^*$$
in the theory $\cob{G}$. Moreover, if the pair $(G, k)$ is split, then the same statement holds in the theory $\coblocaltheory{G}{e}$ for arbitrary $\L \in \picard{G}{X'}$.

\begin{proof}
Suppose that $\gps : X' \to \P^n$ is a \morp\ in $\gsmcat{G}$ such that $\L \cong \gps^* \O(1)$. Let $[Y \to X']$ be an element in $\cobg{G}{}{X'}$ and $H$ be a hyperplane in $\P^n$ such that $Y \x_{\P^n} H$ is a \sm\ \inv\ divisor on $Y$. Then,
\begin{eqnarray}
f^* \circ c(\L) [Y \to X'] & = & f^* [Y \x_{\P^n} H \to X'] \nonumber\\
                           & = & [X \x_{X'} (Y \x_{\P^n} H) \to X]. \nonumber
\end{eqnarray}
On the other hand,
\begin{eqnarray}
c(f^*\L) \circ f^* [Y \to X'] & = & c(f^*\L) [X \x_{X'} Y \to X] \nonumber\\
                              & = & [(X \x_{X'} Y) \x_{\P^n} H \to X]. \nonumber
\end{eqnarray}
Hence, they agree. The proof for arbitrary $\L$ is similar to the proof of the similar statement of \textbf{(A3)}.
\end{proof}

\vtab

\textbf{(A5)} If $\L, \L' \in \picard{G}{X}$ are both \nice, then
$$c(\L) \circ c(\L') = c(\L') \circ c(\L)$$
in the theory $\cob{G}$. Moreover, if the pair $(G, k)$ is split, then the same statement holds in the theory $\coblocaltheory{G}{e}$ for arbitrary $\L,\L' \in \picard{G}{X}$.

\begin{proof}
Suppose that $\L$, $\L'$ are \nice\ and let $\gps : X \to \P^n$ and $\gps' : X \to \P^m$ be the corresponding maps for $\L$ and $\L'$ respectively. Then, for some appropriately chosen hyperplanes $H \subset \P^n$ and $H' \subset \P^m$,
\begin{eqnarray}
c(\L) \circ c(\L')\ [Y \to X] & = & c(\L)\ [Y \x_{\P^m} H' \to X] \nonumber\\
                              & = & [(Y \x_{\P^m} H') \x_{\P^n} H \to X] \nonumber\\
                              & = & [(Y \x_{\P^n} H) \x_{\P^m} H' \to X] \nonumber\\
                              & = & c(\L') \circ c(\L)\ [Y \to X]. \nonumber
\end{eqnarray}
The statement for arbitrary $\L$, $\L' \in \picard{G}{X}$ can be shown by a similar argument as before.
\end{proof}

\vtab

\textbf{(A6)} If $f,g$ are \proj, then
$$\x \circ (f_* \x g_*) = (f \x g)_* \circ \x.$$

\vtab

\textbf{(A7)} If $f,g$ are \sm\ with \equidim\ domains and codomains, then
$$\x \circ (f^* \x g^*) = (f \x g)^* \circ \x.$$

\vtab

\textbf{(A8)} Let $a$, $b$ be elements in $\cobg{G}{}{X}$, $\cobg{G}{}{X'}$ respectively and let $\L \in \picard{G}{X}$ be a \nice\ invertible sheaf. Then we have
$$c(\L)(a) \x b = c(\gp_1^*\L)(a \x b).$$
Moreover, if the pair $(G, k)$ is split, then the same statement holds in the theory $\coblocaltheory{G}{e}$ for arbitrary $\L \in \picard{G}{X}$.

\begin{proof}
Suppose that $\L$ is \nice. W\withoutlog, we can assume $a = [Y \to X]$ and $b = [Y' \to X']$. Let $\gps : X \to \P^n$ be the map corresponding to $\L$. Then, for some $H \subset \P^n$,
\begin{eqnarray}
(c(\L) [Y \to X]) \x [Y' \to X'] & = & [Y \x_{\P^n} H \to X] \x [Y' \to X'] \nonumber\\
                                & = & [(Y \x_{\P^n} H) \x Y' \to X \x X'] \nonumber\\
                                & = & [(Y \x Y') \x_{\P^n} H \to X \x X'] \nonumber\\
                                && (\text{via the map $Y \x Y' \to X \x X' \to X \to \P^n$}) \nonumber\\
                                & = & c(\gp_1^*\L) [Y \x Y' \to X \x X']. \nonumber
\end{eqnarray}
For arbitrary $\L \in \picard{G}{X}$, 
\begin{eqnarray}
c(\L)(a) \x b & = & \fgldiv{e}( F^-( c(\L^e \otimes \cat{M}), c(\cat{M}) ) )(a) \x b \nonumber\\
              & = & ( \sum_i b_i\,( \sum_{j,k} c_{jk}\, c(\L^e \otimes \cat{M})^j\, c(\cat{M})^k )^i (a)) \x b \nonumber\\
              & = & ( \sum_{j,k} d_{jk}\, c(\L^e \otimes \cat{M})^j\, c(\cat{M})^k (a) ) \x b \nonumber\\
              && (\text{expand the series out and denote the coefficients by $d_{jk}$}) \nonumber\\
& = & \sum_{j,k} d_{jk}\, c(\gp_1^*\L^e \otimes \gp_1^*\cat{M})^j\, c(\gp_1^*\cat{M})^k\, (a \x b) \nonumber\\
& = & \sum_i b_i\,( \sum_{j,k} c_{jk}\, c(\gp_1^*\L^e \otimes \gp_1^*\cat{M})^j\, c(\gp_1^*\cat{M})^k )^i\, (a \x b) \nonumber\\
              & = & \fgldiv{e}( F^-( c(\gp_1^*\L^e \otimes \gp_1^*\cat{M}), c(\gp_1^*\cat{M}) ) )\, (a \x b) \nonumber\\
              & = & c(\gp_1^* \L)(a \x b). \nonumber
\end{eqnarray}
\end{proof}

\vtab

\textbf{(Dim)} If $\listing{\L_1}{\L_2}{\L_r} \in \picard{G}{X}$ are \nice\ invertible sheaves and $r > \dim X$, then
$$c(\L_1) \circ c(\L_2) \circ \cdots \circ c(\L_r) [\id_X] = 0.$$
Moreover, if the pair $(G, k)$ is split, then the same statement holds in the theory $\coblocaltheory{G}{e}$ for arbitrary $\listing{\L_1}{\L_2}{\L_r} \in \picard{G}{X}$.

\begin{proof}
It follows from the fact that $c(\L) : \cobglow{G}{*}{X} \to \cobglow{G}{*-1}{X}$ and $\cobglow{G}{<0}{X} = 0$.
\end{proof}

\vtab

\textbf{(FGL)} If $\L$, $\L' \in \picard{G}{X}$ are \nice\ invertible sheaves, then
$$c(\L \otimes \L') = F(c(\L), c(\L'))$$
in the theory $\cob{G}$. Moreover, if the pair $(G, k)$ is split, then the same statement holds in the theory $\coblocaltheory{G}{e}$ for arbitrary $\L$, $\L' \in \picard{G}{X}$.

\begin{proof} 
The statement for \nice\ $\L$, $\L'$ was proved in section \ref{chernclasssection}. For arbitrary $\L$, $\L'$,
\begin{eqnarray}
F(c(\L), c(\L')) & = & F(\ \fgldiv{e}( F^-(c(\L^e \otimes \cat{M}), c(\cat{M})) )\ ,\ \fgldiv{e}( F^-(c(\L'^{\,e} \otimes \cat{M'}), c(\cat{M'})) )\ ) \nonumber\\
                 & = & \fgldiv{e}(F(\ F^-(c(\L^e \otimes \cat{M}), c(\cat{M}))\ ,\ F^-(c(\L'^{\,e} \otimes \cat{M'}), c(\cat{M'}))\ )) \nonumber\\
                 & = & \fgldiv{e}(F^-(\ F(c(\L^e \otimes \cat{M}), c(\L'^{\,e} \otimes \cat{M'}))\ ,\ F(c(\cat{M}), c(\cat{M'}))\ )) \nonumber\\
                 & = & \fgldiv{e}(F^-(\ c(\L^e \otimes \cat{M} \otimes \L'^{\,e} \otimes \cat{M'})\ ,\ c(\cat{M} \otimes \cat{M'})\ )) \nonumber\\
                 & = & c(\L \otimes \L') \nonumber\\
                 && \text{because $(\L \otimes \L')^e \otimes (\cat{M} \otimes \cat{M'})$ and $\cat{M} \otimes \cat{M'}$ are both \nice.} \nonumber
\end{eqnarray}
\end{proof}

\bigskip

\bigskip

\section{Generators for the \equi\ algebraic cobordism ring}
\label{setofgeneratorsection}

The main objective of this section is to prove Theorem \ref{setofgenerator}, which gives a set of generators of the \equi\ algebraic cobordism ring $\cobg{G}{}{\pt}$. To achieve this, we need to use a different version of splitting principle. We will assume the pair $(G,k)$ is split in this section.

\subsection{Splitting principle by blowing up along invariant \sm\ centers}

In this subsection, for a sheaf $\cat{E}$ over $Y$ and a map $f : X \to Y$, we will denote $f^* \cat{E}$ by $\cat{E}_X$ if there is no confusion. Suppose $X$ is a scheme over $k$ and $Z$ is a closed subscheme of $X$. We will denote the blow up of $X$ along $Z$ by $\blowup{X}{Z}$.

The main result in this subsection is similar to the \equi\ analog of Theorem 4.7 in \cite{splittingprinciple}.

\vtab

Let $S \in \gsmcat{G}$ be a ground scheme. Suppose $\cat{N}$ is a \glin{G}\ \lfsheaf{N}\ over $S$ and $\cat{A} \embed \cat{N}$ is a rank 1 \glin{G}\ locally free subsheaf. Recall the definition in section 2.1 \cite{splittingprinciple}.

The scheme $\gs_{1,n}(\cat{A}, \cat{N})$ is defined as the closed subscheme of Grassmannian $Gr_n(\cat{N})$ satisfying the following. A point $(s, \cat{H}) \in Gr_n(\cat{N})$ (i.e. $s \in S$ and $\cat{H}$ is a $n$-quotient of $\cat{N}|_s$) is inside $\gs_{1,n}(\cat{A}, \cat{N})$ if the composition $\cat{A}|_s \to \cat{N}|_s \to \cat{H}$ is zero. 

Also recall the following definition in section 3.1 in \cite{splittingprinciple}.

Suppose $X$ is in $\gsmcat{G}$ and $\cat{N}$ is a \glin{G}\ \lfsheaf{N}\ over $\pt$. An \equi\ immersion $X \embed Gr_r (\cat{N})$ is called twisted if it is the Segre product of an \equi\ map $X \to Gr_r(\cat{N}_1)$ and an \equi\ immersion $X \embed \P(\cat{N}_2)$ for some \glin{G}\ locally free sheaves $\cat{N}_1, \cat{N}_2$ over $\pt$.

\begin{prop}
\label{bertiniprop}
Suppose $X \in \gsmcat{G}$ is \girred\ with dimension $d$ and there is a twisted \equi\ immersion 
$$X \embed Gr_r(\cat{N}) \defeq Y$$ 
for some \glin{G}\ locally free sheaf $\cat{N}$ of rank $N$ over $\pt$ ($1 \leq r < N$). Moreover, there is a 1-dimensional character $\gps$ such that the dimension of the $\gps$ component ${\rm H}^0(\pt, \cat{N})_{\gps}$ is greater than $r$. Let $Z \defeq Gr_{N-1}(\cat{N})$ and $\cat{A}$ be the universal subbundle over $Z$ ($\cat{A} \embed \cat{N}_Z$ with rank 1). Then, there exists a closed point $z$ of the fixed point locus $Z^G$, with residue field $k(z) \cong k$, such that the closed subscheme $\gs_{1,r}(\cat{A}|_z, \cat{N}) \subset Y$ is \sm\ with codimension $r$ and the dimension of $X \cap \gs_{1,r}(\cat{A}|_z, \cat{N})$ is $d-r$.
\end{prop}

\begin{proof}
This statement is similar to Theorem 3.3 in \cite{splittingprinciple}. First of all, notice that 
$$X \x Z \embed Y \x Z = Gr_r(\cat{N}) \x Z \cong Gr_r(\cat{N}_Z).$$ 
On the other hand, the subsheaf $\cat{A} \embed \cat{N}_Z$ induces $\gs_{1,r}(\cat{A}, \cat{N}_Z)$, which is a closed subscheme of $Gr_r(\cat{N}_Z)$. So, we will consider $\gs_{1,r}(\cat{A}|_z, \cat{N})$ and $X \cap \gs_{1,r}(\cat{A}|_z, \cat{N})$ as fibers of 
$$\gs_{1,r}(\cat{A}, \cat{N}_Z) \embed Y \x Z \to Z$$ 
and 
$$(X \x Z) \cap \gs_{1,r}(\cat{A}, \cat{N}_Z) \embed Y \x Z \to Z$$
respectively.

Suppose the $G$-\repn\ corresponding to $\cat{N}$ is given by a $k$-basis $\{e_1, e_2, \ldots, e_N\}$ such that each $e_i$ defines a 1-dimensional $G$-\repn. Let $U_N$ be the \inv\ affine open subscheme of $Z$ corresponding to $e_1, \ldots, e_{N-1}$. Then, $U_N = \spec{k[s_1, \ldots, s_{N-1}]}$. Since $Z = Gr_{N-1}(\cat{N}) \cong \P(\dual{\cat{N}})$ and $\dim {\rm H}^0(\pt, \cat{N})_{\gps} \geq r +1$, w\withoutlog, we may assume $G$ acts on the coordinates $s_1, \ldots, s_r$ trivially. In addition, it can be shown that $\cat{A} \embed \cat{N}_{Z}$ is defined by 
$$f \defeq \left(\sum_{i = 1}^{N-1}\,s_i e_i\right) - e_N$$
over $U_N$. 

Let $U_{1,2,\ldots,r}$ be the affine open subscheme of $Y$ corresponding to $e_1, \ldots, e_r$. Then, we have $U_{1,2,\ldots,r} = \spec{k[t_{i,j}]}$ where 
$1 \leq i \leq r$ and $1 \leq j \leq N-r$. Let $(\cat{N}/\cat{G},z) = (t_{i,j},s_k)$ be a closed point in 
$$\spec{k[t_{i,j}, s_k]} = U_{1,2,\ldots, r} \x U_N \subset Y \x Z = Gr_r \cat{N}_Z.$$ 
Then, the map $\cat{A}|_z \to \cat{N} \to \cat{N}/\cat{G}$ at this point corresponds to
$$\span{k}{f} \embed \oplus_{i =1}^N \span{k}{e_i} \to (\oplus_{i =1}^N \span{k}{e_i})\, /\, \span{k}{g_1, \ldots, g_{N-r}}$$
where 
$$g_i \defeq \left(\sum_{j=1}^r\,t_{j,i}e_j\right) - e_{r+i}$$ 
for $1 \leq i \leq N-r$. The composition being zero is equivalent to $f \in \span{k}{g_1, \ldots, g_{N-r}}$, which is equivalent to 
$$h_i \defeq s_i + \left(\sum_{j=1}^{N-r-1}\ s_{j + r} t_{i,j}\right) - t_{i,(N-r)} = 0$$
for $1 \leq i \leq r$. So, $\gs_{1,r}(\cat{A}, \cat{N}_Z)$ is cut out by the equations $h_1, \ldots, h_r$ inside $U_{1,2,\ldots, r} \x U_N$.

Let $z = (q_1, \ldots, q_{N-1})$ be a closed point in $U_N$. Then, when restricted on the fiber of $U_{1,2,\ldots, r} \x U_N \to U_N$ over $z$, the closed subscheme $\gs_{1,r}(\cat{A}|_z, \cat{N}) \cap U_{1,2,\ldots, r}$ is cut out by $r$ linear equations :
$$h_i = q_i + \left(\sum_{j=1}^{N-r-1}\ q_{j + r} t_{i,j}\right) - t_{i,(N-r)} = 0,$$
where $1 \leq i \leq r$. So, $\gs_{1,r}(\cat{A}|_z, \cat{N}) \cap U_{1,2,\ldots, r}$ is \sm\ and of codimension $r$. Moreover, since $X \embed Gr_r(\cat{N})$ is a twisted immersion and $\gs_{1,r}(\cat{A}|_z, \cat{N}) \cap U_{1,2,\ldots, r}$ is given by $r$ linear equations $\{h_i=0\}$, the scheme $X \cap \gs_{1,r}(\cat{A}|_z, \cat{N}) \cap U_{1,2,\ldots, r}$ is of dimension $d-r$ (See the proof of Theorem 3.3 in \cite{splittingprinciple} for details).

Because of the symmetry of $f$, the only other affine open subscheme of $Y$ we need to consider is $U_{1, \ldots, r-1,N}$. In this case, the map $\cat{A}|_z \to \cat{N} \to \cat{N}/\cat{G}$ corresponds to 
$$\span{k}{f} \embed \oplus_{i =1}^N \span{k}{e_i} \to (\oplus_{i =1}^N \span{k}{e_i})\, /\, \span{k}{g_1, \ldots, g_{N-r}}$$
where 
$$g_i \defeq \left(\sum_{j=1}^{r-1}\,t_{j,i}e_j\right) + t_{r,i}e_N - e_{r+i}$$ 
for $1 \leq i \leq N-r-1$ and 
$$g_{N-r} \defeq \left(\sum_{j=1}^{r-1}\,t_{j,N-r}e_j\right) + t_{r,N-r}e_N - e_r.$$ 
Hence, the equations that cut $\gs_{1,r}(\cat{A}, \cat{N}_Z)$ out are 
$$h_i \defeq s_i + \left(\sum_{j=1}^{N-r-1}\,s_{j + r} t_{i,j}\right) + t_{i,N-r}s_r = 0$$
for $1 \leq i \leq r-1$ and
$$h_r \defeq -1 + \left(\sum_{j=1}^{N-r-1}\,s_{j + r} t_{r,j}\right) + t_{r,N-r}s_r = 0.$$

Let $B$ be the closed subscheme of $U_N$ defined by the equations $s_r = s_{r + 1} = \cdots = s_{N-1} = 0$ and $z = (q_1, \ldots, q_{N-1})$ be a closed point in $U_N - B$. Then, in the fiber of $U_{1, \ldots, r-1,N} \x U_N \to U_N$ over $z$, the closed subscheme $\gs_{1,r}(\cat{A}|_z, \cat{N})$ is cut out by $r$ linear equations
$$h_i = q_i + \left(\sum_{j=1}^{N-r-1}\,q_{j + r} t_{i,j}\right) + t_{i,N-r}q_r = 0$$
for $1 \leq i \leq r-1$ and
$$h_r = -1 + \left(\sum_{j=1}^{N-r-1}\,q_{j + r} t_{r,j}\right) + t_{r,N-r}q_r = 0.$$
Since at least one of $q_r, \ldots, q_{N-1}$ is non-zero, the linear equations $\{h_i\ |\ 1 \leq i \leq r\}$ are linearly independent. Hence, by the same reason, $\gs_{1,r}(\cat{A}|_z, \cat{N}) \cap U_{1, \ldots, r-1,N}$ is \sm\ with codimension $r$ and $X \cap \gs_{1,r}(\cat{A}|_z, \cat{N}) \cap U_{1, \ldots, r-1,N}$ is of dimension $d-r$.

For a different affine open subscheme $U_{i_1, \ldots, i_{r-1},N}$ of $Y$, there is a corresponding ``bad'' closed subscheme $B$ of $U_N$ defined by the set of equations $\{s_j = 0\}$ where $j \notin \{i_1, \ldots, i_{r-1}\}$. Hence, the result follows by picking $z = (q_1, \ldots, q_r, 0 , \ldots, 0)$ such that $q_1, \ldots, q_r$ are all non-zero.
\end{proof}

Suppose $\cat{A} \embed \cat{N}$ are \glin{G}\ locally free sheaves of rank $1$, $N$ respectively, over $\pt$. Let $Y \defeq Gr_{r-1}(\cat{N}/\cat{A})$ and $\cat{Q}^Y$ be its universal quotient. Let $\cat{K}$ be the kernel of the composition $\cat{N}_Y \to (\cat{N}/\cat{A})_Y \to \cat{Q}^Y$. Define a map $g : Gr_1(\cat{K}) \to Gr_r(\cat{N})$ as the following.

For a point $(y,\cat{H})$ in $Gr_1(\cat{K})$, we get an exact sequence
$$0 \to \cat{G} \to \cat{K}|_y \to \cat{H} \to 0$$
where the rank of $\cat{G}$ will be $N-r$. Since $\cat{K}|_y \embed \cat{N}|_y = \cat{N}$, we can consider $\cat{N}/\cat{G}$, which is of rank $r$. Thus, we define
$$g(y,\cat{H}) \defeq \cat{N}/\cat{G}.$$

\begin{prop}
\label{splittingprinciplelemma}
The map $g : Gr_1(\cat{K}) \to Gr_r(\cat{N})$ constructed above is \equi ly isomorphic to the map corresponding to the blow up of $Gr_r(\cat{N})$ along $\gs_{1,r}(\cat{A},\cat{N})$.
\end{prop}

\begin{proof}
This is the analog of Theorem 4.4 in \cite{splittingprinciple}. First of all, it is not hard to see that $g$ is \equi. Let $X \defeq Gr_r(\cat{N})$, $Y \defeq Gr_1(\cat{K})$ and $\tilde{X} \defeq \blowup{Gr_r(\cat{N})}{\gs_{1,r}(\cat{A},\cat{N})}$. Also denote the blow up map from $\tilde{X}$ to $X$ by $\gp$. By Theorem 4.4 in \cite{splittingprinciple}, there exists an isomorphism $\gm : \tilde{X} \to Y$ such that $g \circ \gm = \gp$. So, it is enough to show $\gm$ is \equi. Take an invariant open subscheme $U \subset X$ such that $g|_U$ and $\gp|_U$ are both isomorphisms. Since $g|_U,\gp|_U$ are both \equi, the map $\gm|_U = (g|_U)^{-1} \circ \gp|_U$ is also \equi. Now, a map being \equi\ is a closed condition. Hence, $\gm$ is \equi. 
\end{proof}

\begin{thm}
\label{splittingprinciple}
Suppose $X \in \gsmcat{G}$ has dimension $d$ and $\cat{E}$ is a \glin{G}\ \lfsheaf{r}\ over $X$. Then, there exists an \equi\ \morp\ $f : \tilde{X} \to X$, which is the composition of a series of blow ups along invariant \sm\ centers with dimensions $\leq d-r$, and a \glin{G}\ invertible subsheaf $\L \embed f^* \cat{E}$ over $\tilde{X}$ such that the sequence
$$0 \to \L \to f^* \cat{E} \to (f^* \cat{E})/\L \to 0$$
is exact and $(f^* \cat{E})/\L$ is locally free with rank $r-1$.
\end{thm}

\begin{proof}
W\withoutlog, we may assume $X$ is \girred. The result is trivially true if $d=0$, so we may assume $d \geq 1$. By Proposition \ref{equiprojmorp}, we can embed $X$ into $\P(\cat{N}_2)$ for some \glin{G}\ locally free sheaf $\cat{N}_2$ over $\pt$. Denote $\cat{E} \otimes \O_X(m)$ by $\cat{E}(m)$ for simplicity. Assume $X$ is \proj\ first. Let $\cat{N}_1$ be the \glin{G}\ locally free sheaf over $\pt$ corresponding to ${\rm H}^0(X, \cat{E}(m))$. For a sufficiently large $m$, we can assume the induced map $(\cat{N}_1)_X \to \cat{E}(m)$ is surjective and defines an \equi\ immersion $X \embed Gr_r(\cat{N}_1)$, which sends $x$ to $\cat{E}(m)|_x$. Then, we define a twisted \equi\ immersion $i : X \embed Gr_r(\cat{N}) \defeq Y$ as the Segre product of $X \embed Gr_r(\cat{N}_1)$ and $X \embed \P(\cat{N}_2)$. In particular, $\cat{N} \cong \cat{N}_1 \otimes \cat{N}_2$.

By construction, $i^*\cat{Q}^Y \cong \cat{E}(m+1)$ where $\cat{Q}^Y$ is the universal quotient of $Y$. Since dim ${\rm H}^0(X, \cat{E}(m))$ is a polynomial of $m$ with degree $d$, we may assume there is a 1-dimensional character $\gps$ such that the $\gps$ component ${\rm H}^0(X, \cat{E}(m))_{\gps}$ has dimension much larger than $r$. 

If $X$ is not \proj, we can pick $\cat{N}_1$ to be a sheaf corresponding to some finite dimensional $G$-\repn\ inside ${\rm H}^0(X, \cat{E}(m))$ and construct $i : X \embed Y$ in the same manner.

Let $\cat{A}$ be the universal subbundle of $Gr_{N-1}(\cat{N})$. Let $V_1$, $V_2$ and $V$ be the $G$-\repn s corresponding to $\cat{N}_1$, $\cat{N}_2$ and $\cat{N}$ respectively. Then, the dimension of the $\gps$ component of $V_1$ is much larger than $r$ by construction. Thus, there is a 1-dimensional character $\gps'$ such that the dimension of the $\gps'$ component of $V$ is much larger than $r$. Hence, by Proposition \ref{bertiniprop}, there exists a closed point $z$ of the fixed point locus of $Gr_{N-1}(\cat{N})$, with residue field $k(z) \cong k$, such that $\gs_{1,r}(\cat{A}|_z, \cat{N}) \subset Y$ is \sm\ with codimension $r$ and $X \cap \gs_{1,r}(\cat{A}|_z, \cat{N})$ has dimension $d-r$.  

For such $z$, denote $\gs_{1,r}(\cat{A}|_z, \cat{N})$ by $\gs$ for simplicity. Then, we have \sm\ invariant closed subschemes $X$, $\gs$ of $Y$ with dimension $d$ and $\dim Y-r$ respectively. Moreover, $X \cap \gs$ has dimension $d-r$. By applying the embedded desingularization theorem in \cite{embeddeddesingularization} on $X \cup \gs \embed Y$, we obtain the following commutative diagram :
\squarediagramword{\tilde{X}}{Y'}{X}{Y}{i'}{f}{p}{i}
where $p : Y' \to Y$ is the composition of a series of blow ups along \sm\ invariant centers with dimensions bounded above by $\dim \singular{X \cup \gs}$ and $f : \tilde{X} \to X$ is the map corresponding to the strict transform of $X$. Moreover, $\tilde{X} \cup \stricttransform{\gs}$ (denote the strict transform by $\stricttransform{\ }$) is \sm\ and if $E$ is the sum of the exceptional divisors on $Y'$, then $\tilde{X}$, $\stricttransform{\gs}$ and $E$ will intersect transversely. Notice that since $\singular{X \cup \gs} = X \cap \gs$, the dimensions of the \sm\ \inv\ centers are all bounded above by $d-r$. In other words, $f$ is the composition of a series of blow ups along \sm\ invariant centers with dimensions $\leq d-r$.

Observe that $\tilde{X}$ and $\stricttransform{\gs}$ are disjoint because $\tilde{X} \cup \stricttransform{\gs}$ is \sm. In addition, 
$$i'^{-1} \circ p^{-1}(\gs) = i'^{-1} (\stricttransform{\gs} \cup E) = \tilde{X} \cap ( \stricttransform{\gs} \cup E ) = \tilde{X} \cap E,$$
which is an \inv\ divisor on $\tilde{X}$. By the universal property of blow up, there is a unique map $j : \tilde{X} \to \blowup{Y}{\gs}$ such that the following diagram commutes.
\squarediagramword{\tilde{X}}{\blowup{Y}{\gs}}{Y'}{Y}{j}{i'}{q}{p}

Since $z$ is a fixed point with $k(z) \cong k$, the sheaf $\cat{A}|_z$ is a \glin{G}\ \lfsheaf{1}\ over $\pt$ and it is naturally embedded inside $\cat{N}$. Following the construction before. Let $Y_1 \defeq Gr_{r-1}(\cat{N}/\cat{A}|_z)$, $\cat{Q}^{Y_1}$ be its universal quotient, $\cat{K}$ be the kernel of $\cat{N}_{Y_1} \to \cat{Q}^{Y_1}$ and $\tilde{Y} \defeq Gr_1(\cat{K})$. By Proposition \ref{splittingprinciplelemma}, the \equi\ map $g : \tilde{Y} \to Y$ is \equi ly isomorphic to $q : \blowup{Y}{\gs} \to Y$. Moreover, as pointed out in (4.1) in \cite{splittingprinciple}, there is an exact sequence
\begin{eqnarray}
\label{eqn5}
0 \to \L' \defeq \cat{Q}^{\tilde{Y}} \to g^*\cat{Q}^Y \to (\cat{Q}^{Y_1})_{\tilde{Y}} \to 0
\end{eqnarray}
of \glin{G}\ locally free sheaves over $\tilde{Y}$ where $\L'$ is of rank 1.

Consider the following commutative diagram :

\begin{center}
$\begin{CD}
\tilde{X} @>{j}>> \blowup{Y}{\gs} @= \blowup{Y}{\gs} \\
@V{f}VV @V{q}VV @V{\gm}VV \\
X @>{i}>> Y @<{g}<< \tilde{Y}
\end{CD}$
\end{center}

\noindent On one hand, $f^* i^*Q^Y \cong f^* \cat{E}(m+1)$. On the other hand, if we pull back the exact sequence (\ref{eqn5}) by $\gm$ and then $j$. We got an exact sequence of \glin{G}\ locally free sheaves over $\tilde{X}$
$$0 \to j^*\gm^*\L' \to f^* \cat{E}(m+1) \cong j^*\gm^*g^*\cat{Q}^Y \to j^*\gm^*(\cat{Q}^{Y_1})_{\tilde{Y}} \to 0.$$
The result then follows by twisting the whole sequence by $f^*\O_X(-m-1)$.

\end{proof}

\vtab

\vtab

\subsection{Basic structure of \glin{G}\ invertible sheaves}

In this subsection, we will state and prove some results about the structure of \glin{G}\ invertible sheaves over some $X \in \gsmcat{G}$.

\begin{lemma}
\label{kernelofforgetful}
For any $X \in \gsmcat{G}$, we have
$$\kernel{\{ \picard{G}{X} \to \picard{}{X} \}} = \gp_k^*\, \picard{G}{\pt}$$ 
where $\picard{G}{X} \to \picard{}{X}$ is the forgetful map.
\end{lemma}

\begin{proof}
Finding the kernel of the forgetful map is the same as asking how many $G$-linearizations can $\O_X$ have. A $G$-linearization of $\O_X$ can be described by a set of isomorphisms 
$$\{\ga^* : \O_X\iso\O_X \ |\ \ga \in G\}.$$ 
Each iso\morp\ $\ga^*$ induces an isomorphism 
$$\ga^* : {\rm H}^0(X, \O_X)\iso{\rm H}^0(X, \O_X)$$ 
which sends 1 to some element $a_{\ga} \in {\rm H}^0(X, \O_X)$. Since $a_{\ga}^e = 1$ ($e$ is the exponent of $G$) and the pair $(G,k)$ is split, $a_{\ga}$ is in $k^*$. In other words, there exists a 1-dimensional character $\gc$ such that $\ga^*(1) = \gc(\ga)$ for all $\ga \in G$. Then, the result follows from the one to one correspondence between the set of 1-dimensional characters and $\picard{G}{\pt}$.
\end{proof}

\begin{prop}
\label{lbstructure}
Suppose $X \in \gsmcat{G}$ is \girred\ and $\L$ is a \glin{G}\ invertible sheaf over $X$. Then, there exist an invariant divisor $D$ on $X$ and a sheaf $\cat{N} \in \picard{G}{\pt}$ such that
$$\L \cong \O_X(D) \otimes \gp_k^* \cat{N}.$$
\end{prop}

\begin{proof}
W\withoutlog, we may assume the action on $X$ is faithful. Let $U$ be a non-empty, invariant open subscheme of $X$ such that the action on $U$ is free. By Theorem 1 in section 7 of \cite{Mu2}, the geometric quotient $U/G$ exists as a variety over $k$ and $\gp : U \to U/G$ is an \etale\ \morp. By picking a smaller $U$, we may further assume $U/G$ to be \sm. Let $D_1, \ldots, D_n$ be some invariant divisors on $X$ such that $D_i \subset X - U$ for all $i$ and the codimension of $X - U - \cup_i\, D_i$ in $X$ is at least 2.

\vtab

Claim 1 : The kernel of the restriction map $\picard{G}{X} \to \picard{G}{U}$ is generated by $\{\O_X(D_i)\}$ and $\picard{G}{\pt}$.

Consider the following commutative diagram :

\begin{center}
$\begin{CD}
     @.   \picard{G}{\pt} @.                 @. \\
@.        @V{\gp_k^*}VV                 @. @.\\
\Z^n @>{a}>> \picard{G}{X}   @>{c}>> \picard{G}{U} @. \\
@V{\id}VV      @V{b}VV                 @V{b}VV               @. \\
\Z^n @>{a}>> \picard{}{X}    @>{c}>> \picard{}{U}  @>>> 0
\end{CD}$
\end{center}

\noindent where $a$ sends ``1'' in the $i$-th position to $\O_X(D_i)$, $b$ is the forgetful map and $c$ is the restriction map.

Clearly, the third row is exact. Moreover, by Lemma \ref{kernelofforgetful}, the second column is also exact. Then, the result follows from some diagram chasing. \claimend

\vtab

Since the action on $U$ is free, according to Proposition 2 in section 7 in \cite{Mu2}, there is a one-to-one correspondence between $\picard{G}{U}$ and $\picard{}{U/G}$. In particular, $\gp_* (\L|_U)^G$ is an invertible sheaf over $U/G$. Since $U/G$ is \sm, there is a divisor $D'$ on $U/G$ such that $\gp_* (\L|_U)^G \cong \O_{U/G}(D').$ Thus, we have
\begin{eqnarray}
\L|_U &\cong& \gp^* (\gp_* (\L|_U)^G) \nonumber\\
&\cong& \gp^* \O_{U/G}(D') \nonumber\\
&\cong& \O_U(\gp^* D') \nonumber\\
&& \text{($\gp : U \to U/G$ is \etale).} \nonumber
\end{eqnarray}
Consider the sheaf $\O_X(D'') \in \picard{G}{X}$ where $D''$ is the invariant divisor on $X$ given by the closure of $\gp^* D'$ in $X$. Hence, $\L \otimes \O_X(-D'')$ will be in the kernel of the restriction map $\picard{G}{X} \to \picard{G}{U}.$ By claim 1, there are integers $\{m_i\}$ and a sheaf $\cat{N} \in \picard{G}{\pt}$ such that
$$\L \otimes \O_X(-D'') \cong \O_X(\sum_i m_i D_i) \otimes \gp_k^* \cat{N}.$$
The result then follows by defining $D \defeq D'' + \sum_i m_i D_i$.
\end{proof}

\vtab

\subsection{Reduction of towers}
\label{reductionoftowersubsection}

Next, we will define the notion of \qadtower\ and \adtower\ and prove we can reduce an \qadtower\ into something much simplier. This subsection is an analog of section 7 in \cite{LePa}.

\begin{defn}
{\rm
Suppose $Y$ is an object in $\gsmcat{G}$. A \morp\ $\P \to Y$ in $\gsmcat{G}$ is called a \qadtower\ over $Y$ with length $n$ if it can be factored into
$$\P = \P_n \to \P_{n-1} \to \cdots \to \P_1 \to \P_0 = Y$$  
such that, for all $0 \leq i \leq n-1$, $\P_{i+1} = \P(\cat{E}_i)$ where $\cat{E}_i$ is the direct sum of sheaves which is either the pull-back of a \glin{G}\ locally free sheaves over $Y$, or the pull back of $\O_{\P_j}(m)$ for some integer $m$ and $1 \leq j \leq i$. 
}\end{defn}

In this subsection, for an object $Y \in \gsmcat{G}$, an invariant divisor $D$ on $Y$ and a \glin{G}\ locally free sheaf $\cat{E}$ over $Y$, we will denote $\cat{E} \otimes \O_Y(D)$ by $\cat{E}(D)$ for simplicity. Moreover, if $\P \to Y$ is a \qadtower, then we will denote the pull-back of $\cat{E}$ as a sheaf over $\P_i$ by $\cat{E}$ if there is no confusion.

\begin{defn}
{\rm
Suppose $Y$ is an object in $\gsmcat{G}$. We will call a sheaf $\L \in \picard{G}{Y}$ admissible if there exist invariant \sm\ divisors $D_1, \ldots, D_k$ on $Y$ and a sheaf $\cat{N} \in \picard{G}{\pt}$ such that

\begin{center}
$\L \cong \O_Y(\sum_{i = 1}^k m_i D_i) \otimes \gp_k^* \cat{N}$
\end{center}

\noindent for some integers $\{m_i\}$. Denote the subgroup of $\picard{G}{Y}$ generated by \adinvsh ves by $\apicard{G}{Y}$. Also, define the group of \adinvsh ves over $\P_i$ by
$$\apicard{G}{\P_i} \defeq \apicard{G}{Y} + \Z \O_{\P_1}(1) + \cdots + \Z \O_{\P_i}(1).$$ 
Then, a \qadtower\ $\P \to Y$ is called admissible if all sheaves involved in the construction are \adinvsh ves.
}\end{defn}

\begin{rmk}
\label{invshimplyqadtower}
\rm{
If all the \glin{G}\ locally free sheaves involved in the construction of a tower $\P \to Y$ are invertible, then it is a \qadtower.

\begin{proof}
Since
$$\picard{}{\P_i} = \picard{}{Y} + \Z \O_{\P_1}(1) + \cdots + \Z \O_{\P_i}(1)$$
and, by Lemma \ref{kernelofforgetful}, the kernel of the forgetful map $\picard{G}{\P_i} \to \picard{}{\P_i}$ is given by $\picard{G}{\pt}$, we have
$$\picard{G}{\P_i} = \picard{G}{Y} + \Z \O_{\P_1}(1) + \cdots + \Z \O_{\P_i}(1).$$
Then, $\P \to Y$ is a \qadtower\ by definition.
\end{proof}
}
\end{rmk}

\begin{lemma}
\label{divoftwisting}
Suppose $Y \in \gsmcat{G}$ is \girred\ and $\L$ is a sheaf in $\picard{G}{Y}$. Moreover, $\cat{E}$ is the direct sum of a finite number of invertible sheaves in $\picard{G}{Y}$ and $D$ is an invariant \sm\ divisor on $Y$. Let 
\begin{eqnarray}
A & \defeq & \P(\cat{E} \oplus \L \oplus \L(D))|_D, \nonumber\\
B & \defeq & \P(\cat{E} \oplus \L), \nonumber\\
C & \defeq & \P(\cat{E} \oplus \L(D)), \nonumber\\
\P & \defeq & \P(\cat{E} \oplus \L \oplus \L(D)). \nonumber
\end{eqnarray}
Then $A$, $B$, $C$ are invariant \sm\ divisors on $\P$, the sum of them is a \rsncd, $A + B \sim C$ and
\begin{eqnarray}
\O_{\P}(A) & \cong & \O_{\P}(\gp^* D) \nonumber\\
\O_{\P}(B) & \cong & \dual{(\gp^* \L(D))} \otimes \O_{\P}(1) \nonumber\\
\O_{\P}(C) & \cong & \dual{(\gp^* \L)} \otimes \O_{\P}(1) \nonumber
\end{eqnarray}
where $\gp$ is the projection $\P \to Y$.
\end{lemma}

\begin{proof}
The fact that $A$, $B$, $C$ are \sm\ divisors on $\P$ and the sum of them is a \rsncd\ was stated in section 7.2 in \cite{LePa}. They are obviously invariant. Since $\gp$ is \sm, $O_{\P}(A) \cong \O_{\P}(\gp^*D).$ Moreover, as in the proof of Lemma 7.1 in \cite{LePa}, $\P(\cat{E} \oplus \L) \subset \P(\cat{E} \oplus \L \oplus \L(D))$ is given by the vanishing of the composition of \equi\ \morp s
$$\gp^* \L(D) \to \gp^*(\cat{E} \oplus \L \oplus \L(D)) \to \O_{\P}(1).$$
Hence, 
$$O_{\P}(B) = O_{\P}(\P(\cat{E} \oplus \L)) \cong \dual{(\gp^* \L(D))} \otimes \O_{\P}(1).$$ 
Similarly, 
$$O_{\P}(C) = O_{\P}(\P(\cat{E} \oplus \L(D))) \cong \dual{(\gp^* \L)} \otimes \O_{\P}(1).$$
Then, we have
\begin{eqnarray}
\O_{\P}(A) \otimes \O_{\P}(B) & \cong & \O_{\P}(\gp^*D) \otimes \dual{(\gp^* \L(D))} \otimes \O_{\P}(1) \nonumber\\
& \cong & \gp^*\O_Y(D) \otimes \gp^* \dual{\L(D)} \otimes \O_{\P}(1) \nonumber\\
& \cong & \gp^* \dual{\L} \otimes \O_{\P}(1) \nonumber\\
& \cong & \O_{\P}(C) \nonumber
\end{eqnarray}
By remark \ref{divisopicrmk}, that implies $A + B \sim C$.
\end{proof}

\begin{lemma}
\label{towertwisting}
Suppose $Y$ is \girred\ and $D$ is an invariant \sm\ divisor on $Y$. If $\P \to Y$ is an \adtower\ with length $n$ and $\P_{i+1} = \P(\oplus_{j=1}^r \L_j)$, then there exist an \adtower\ $\P' \to Y$ of length $n$ and \qadtower s $Q_0$, $Q_1$, $Q_2$, $Q_3 \to D$ such that
$$\P' = \P_n' \to \cdots \to \P_{i+1}' \to \P_i \to \cdots \to \P_0 = Y$$
where $\P_{i+1}' = \P( (\oplus_{j=1}^{r-1} \L_j) \oplus \L_r(D))$ and we have the following equality in $\cobg{G}{}{Y}$ :
$$[\P' \to Y] - [\P \to Y] = [Q_0 \to D \to Y] - [Q_1 \to D \to Y] + [Q_2 \to D \to Y] - [Q_3 \to D \to Y].$$
\end{lemma}

\begin{proof}
Let $\hat{\P}_{i+1} \defeq \P((\oplus^r_{j=1} \L_j) \oplus \L_r(D))$. Then, we have $\hat{\P}_{i+1} \to \P_i$ and $\P_{i+1} \embed \hat{\P}_{i+1}$.

We will first construct an \adtower
$$\hat{\P} = \hat{\P}_n \to \cdots \to \hat{\P}_{i+1} \to \P_i \to \cdots \to \P_0 = Y$$ 
such that $\P_k = \P_{i+1} \x_{\hat{\P}_{i+1}} \hat{\P}_k$ for all $k > i$. Since 
$$\apicard{G}{\P_{i+1}} = \apicard{G}{Y} + \Z \O_{\P_1}(1) + \cdots + \Z \O_{\P_{i+1}}(1)$$
$$\apicard{G}{\hat{\P}_{i+1}} = \apicard{G}{Y} + \Z \O_{\P_1}(1) + \cdots + \Z \O_{\P_i}(1) + \Z \O_{\hat{\P}_{i+1}}(1)$$
and the restriction map $\picard{G}{\hat{\P}_{i+1}} \to \picard{G}{\P_{i+1}}$ sends $\O_{\hat{\P}_{i+1}}(1)$ to $\O_{\P_{i+1}}(1)$, if we write $\P_{i+2} = \P(\oplus \L'_{j'})$ for some $\L'_{j'} \in \apicard{G}{\P_{i+1}}$, then we can define $\hat{\P}_{i+2} \defeq \P(\oplus \L'_{j'})$ by considering $\L'_{j'}$ as in $\apicard{G}{\hat{\P}_{i+1}}$. Similarly, for higher levels, $\apicard{G}{\hat{\P}_k} \to \apicard{G}{\P_k}$ is surjective and $\hat{\P}_{k+1}$ can be constructed.

Next, we will construct the \adtower\ $\P' \to Y$ and \qadtower\ $Q_0 \to Y$. As in the statement, $\P_{i+1}' \defeq \P( (\oplus^{r-1}_{j=1} \L_j) \oplus \L_r(D))$, which can be naturally embedded inside $\hat{\P}_{i+1}$. Then, we define $\P_k' \defeq \P_{i+1}' \x_{\hat{\P}_{i+1}} \hat{\P}_k$ for all $k > i+1$, which is clearly admissible. The \qadtower\ $Q_0$ are defined by pull-back, i.e. ${(Q_0)}_j \defeq D \x_Y \hat{\P}_j$ for all $0 \leq j \leq n$.

By Lemma \ref{divoftwisting}, $\P_{i+1}$, $\P_{i+1}'$ and $(Q_0)_{i+1}$ are all invariant \sm\ divisors on $\hat{\P}_{i+1}$, the sum of them is a \rsncd\ and $(Q_0)_{i+1} + \P_{i+1} \sim \P'_{i+1}$. Pull them back to the top level, we have $Q_0 + \P \sim \P'$ as invariant \sm\  divisors on $\hat{\P}$. By $GDPR(2,1)$, we have
\begin{eqnarray}
\label{eqn12}
[\P' \embed \hat{\P}] &=& [Q_0 \embed \hat{\P}] + [\P \embed \hat{\P}] \\
&& -\ [(Q_0 \cap \P) \x_{\hat{\P}} P^1 \to \hat{\P}] \nonumber\\
&& +\ [(Q_0 \cap \P \cap \P') \x_{\hat{\P}} P^2 \to \hat{\P}] \nonumber\\
&& -\ [(Q_0 \cap \P \cap \P') \x_{\hat{\P}} P^3 \to \hat{\P}]  \nonumber
\end{eqnarray}
as elements in $\cobg{G}{}{\hat{\P}}$, where 
\begin{eqnarray}
P^1 &\defeq& \P(\O \oplus \O(Q_0)) \nonumber\\
P^2 &\defeq& \P(\O \oplus \O(1)) \to \P(\O(-\P) \oplus \O(-\P')) \nonumber\\
P^3 &\defeq& \P(\O \oplus \O(-\P) \oplus \O(-\P')). \nonumber
\end{eqnarray}
We then denote $(Q_0 \cap \P) \x_{\hat{\P}} P^1$, $(Q_0 \cap \P \cap \P') \x_{\hat{\P}} P^2$ and $(Q_0 \cap \P \cap \P') \x_{\hat{\P}} P^3$ by $Q_1$, $Q_2$ and $Q_3$ respectively. They all clearly lie over $D$. Since the towers $Q_1$, $Q_2$, $Q_3 \to D$ are all constructed by \glin{G}\ invertible sheaves, by Remark \ref{invshimplyqadtower}, they are all \qadtower s. Hence, the result follows by pushing down equality (\ref{eqn12}) to $\cobg{G}{}{Y}$.
\end{proof}

\begin{rmk}
\label{onlyleveltwisted}
\rm{
Notice that $\P'_{j} = \P_{j}$ for all $j < i+1$. For $j > i+1$, if we identify the \adinvsh ves over $\P_{j-1}$ that comes from $Y$ to those over $\P'_{j-1}$ and also the sheaves of the form $\O(m)$ for some integer $m$, then $\P'_{j}$ is defined by the exact same set of \adinvsh ves as $\P_{j}$. 
}
\end{rmk}

\begin{lemma}
\label{splittingisenough}
Suppose $Y$ is an object in $\gsmcat{G}$ and $\cat{E}$ is a \glin{G}\ \lfsheaf{r}\ over $Y$. Furthermore, there exists an exact sequence of \glin{G}\ sheaves over $Y$
$$0 \to \L \to \cat{E} \to \cat{E}/\L \to 0$$
such that $\L$ and $\cat{E}/L$ are locally free of rank $1$, $r-1$ respectively. Then, 
$$\P(\cat{E}) \sim \P((\cat{E}/\L) \oplus \L)$$ 
as invariant \sm\ divisors on $\P(\cat{E} \oplus \L)$ and they intersect transversely.
\end{lemma}

\begin{proof}
W\withoutlog, we may assume $Y$ is \girred. $\P(\cat{E})$ and $\P((\cat{E}/\L) \oplus \L)$ are obviously invariant \sm\ divisors on $\P(\cat{E} \oplus \L)$ and their intersection is $\P(\cat{E}/\L)$. So, we only need to prove they are \equi ly linearly equivalent.

Ignore the $G$-action first. Locally, over an affine open subscheme $U_i$, we have $\cat{E} \cong R e_{i1} \oplus \cdots \oplus R e_{ir}$ where $R \defeq \O_Y(U_i)$. Similarly, $\L \cong R f_i$. Let $\gph : \L \embed \cat{E}$ be the embedding of sheaves as in the statement. For simplicity, denote $\P(\cat{E} \oplus \L)$ by $\P$, $\P(\cat{E})$ by $A$ and $\P((\cat{E}/\L) \oplus \L)$ by $B$. Locally, $\P = {\rm Proj}\ R[e_{i1}, \ldots, e_{ir}, f_i]$, $A$ is defined by $f_i = 0$ and $B$ is defined by $\gph(f_i) = 0$. So, it is enough to show $g \defeq f_i / \gph(f_i) \in \cat{K}(\P)^*$ is independent of $i$, namely, $f_i / \gph(f_i) = f_j / \gph(f_j)$.

On the intersection $U_i \cap U_j$, we can consider the ratio $f_i / f_j \defeq \gs_{ij} \in \O(U_i \cap U_j)^*$, which defines the transition function of $\L$. On the other hand, since $\gph : \L \embed \cat{E}$ is a \morp\ between sheaves, we can also consider the ratio $\gph(f_i) / \gph(f_j)$ and it should be $\gs_{ij}$ too. That means 
$$\frac{f_i}{f_j} = \gs_{ij} = \frac{\gph(f_i)}{\gph(f_j)}.$$
Hence, $g$ is independent of $i$. Finally,
$$\ga \cdot g = \ga \cdot \frac{f_i}{\gph(f_i)} = \frac{\ga \cdot f_i}{\ga \cdot \gph(f_i)} = \frac{\ga \cdot f_i}{\gph(\ga \cdot f_i)} = g.$$
\end{proof}

The following result is an analog of Lemma 5.1 in \cite{LePa}.

\begin{lemma}
\label{blowuprelation}
If $X$ is in $\gsmcat{G}$ and $Z$ is an invariant \sm\ closed subscheme of $X$, then, as elements in $\cobg{G}{}{X}$,
$$[\blowup{X}{Z} \to X] - [\id_X] = - [\P_1 \to Z \embed X] + [\P_2 \to Z \embed X]$$
for some \proj\ \morp s $\P_1$, $\P_2 \to Z$ in $\gsmcat{G}$.
\end{lemma}

\begin{proof}
W\withoutlog, $X$ is \girred. Let $Y \defeq \blowup{(X \x \P^1)}{Z \x 0}$ (trivial action on $\P^1$). Consider the \proj\ map $Y \to X \x \P^1$. For any closed point $\gx \neq 0$ in $\P^1$, we have $[Y_{\gx} \to X] = [\id_X]$, where $Y_{\gx}$ denotes the fiber of $Y$ over $\gx$ as before. Consider the fiber of $Y$ over 0, we have $Y_0 = A \cup B$ where $A \defeq \P(\O_Z \oplus \dual{\nbundle{Z}{X}})$ (the exceptional divisor) and $B \defeq \blowup{X}{Z}$ (the strict transform of $X$). In addition, $A \cap B = \P(\dual{\nbundle{Z}{X}})$. Hence, $Y_{\gx}$, $A$, $B$ are all invariant \sm\ divisors on $Y$ and $A,B$ intersect transversely. In other words, $Y \to X \x \P^1$ defines an \equi\ DPR. By Proposition \ref{equidprholds}, we have
\begin{eqnarray}
&&[\id_X] \nonumber\\
&=& [\P(\O_Z \oplus \dual{\nbundle{Z}{X}}) \to X] + [\blowup{X}{Z} \to X] - [\P(\O_{A \cap B} \oplus \O_{A \cap B}(A)) \to A \cap B \to X] \nonumber\\
&=& [\P(\O_Z \oplus \dual{\nbundle{Z}{X}}) \to Z \embed X] + [\blowup{X}{Z} \to X] - [\P(\O_{A \cap B} \oplus \O_{A \cap B}(A)) \to Z \embed X]. \nonumber
\end{eqnarray}
Then, the result follows from defining $\P_1 \defeq \P(\O_Z \oplus \dual{\nbundle{Z}{X}})$ and $\P_2 \defeq \P(\O_{A \cap B} \oplus \O_{A \cap B}(A))$ and the fact that $A \cap B = \P(\dual{\nbundle{Z}{X}})$ is \proj\ over $Z$.
\end{proof}

\begin{rmk}
\label{blowuprelationrmk}
\rm{
We can express $\P_2$ in a different way. Consider the following commutative diagram :
\vtab
\squarediagram{\P(\O_{A \cap B} \oplus \O_{A \cap B}(A))}{\P(\dual{\nbundle{Z}{X}}) = A \cap B}{\P(\O_A \oplus \O_A(A))}{\P(\O_Z \oplus \dual{\nbundle{Z}{X}}) = A}
\vtab
Since $A$ is the exceptional divisor of the blowup $Y \to X \x \P^1$, we have $\O_A(A) \cong \O_A(-1)$. Thus, $\O_{A \cap B}(A) \cong \O_{A \cap B}(-1)$. Hence,
$$\P_1 = \P(\O \oplus \dual{\nbundle{Z}{X}}) \to Z$$
$$\P_2 = \P(\O \oplus \O(-1)) \to \P(\dual{\nbundle{Z}{X}}) \to Z.$$
}
\end{rmk}

\begin{defn}
{\rm
Define $\cobg{G}{}{\pt}'$ to be the abelian subgroup of $\cobg{G}{}{\pt}$ generated by \adtower s over $\pt$.
}\end{defn}

\begin{rmks}
\rm{
If $\P \to \pt$ and $\P' \to \pt$ are two \adtower s, then the product $\P \x \P' \to \P \to \pt$ is also an \adtower\ over $\pt$. In other words, $\cobg{G}{}{\pt}'$ is a subring of $\cobg{G}{}{\pt}$.
}
\end{rmks}

\begin{prop}
\label{towerreduction}
For any \qadtower\ $\P \to Y$ where $Y$ is \girred, there exist elements $a_i \in \cobg{G}{}{\pt}'$ and maps $Y'_i \to Y$ in $\gsmcat{G}$ with $\dim Y'_i \leq \dim Y$ such that
$$[\P \to Y] = \sum_i a_i\, [Y'_i \to Y]$$
as elements in $\cobg{G}{}{Y}$.
\end{prop}

\begin{proof}
We will prove the statement by induction on dimension of $Y$. We will handle the induction step first. Suppose $\dim Y \geq 1$. Let $\cobg{G}{}{Y}'$ be the subgroup of $\cobg{G}{}{Y}$ generated by elements of the form $[\P \to Y' \to Y]$ where $Y' \in \gsmcat{G}$ is \girred\ with dimension less than $\dim Y$ and $\P \to Y'$ is a \qadtower. So, elements in $\cobg{G}{}{Y}'$ will be handled by the induction assumption. Let $\P \to Y$ be a \qadtower. If the length of the tower $n$ is 0, then we are done. Suppose $n \geq 1$. 

\vtab

Step 1 : Reduction to a \qadtower\ constructed only by \glin{G}\ invertible sheaves.

Define the integer ``total rank'' as the sum of ranks of all sheaves involved in all levels. Also, define the integer ``number of sheaves'' as the number of sheaves in all levels. For example, the tower $\P(\cat{E}_1) \to \P(\cat{E}_2 \oplus \cat{E}_3) \to Y$ has total rank = rank $\cat{E}_1$ + rank $\cat{E}_2$ + rank $\cat{E}_3$ and number of sheaves 3.

Assume that, for the tower $\P \to Y$, number of sheaves is less than total rank. Then, there exists a sheaf $\cat{E}$, which is used in the construction of some level $\P_i$, has rank greater than 1. Notice that $\cat{E}$ has to come from $Y$ because the tower is quasi-admissible. Let $\P_i \defeq \P((\oplus_j \cat{E}_j) \oplus \cat{E})$. By Theorem \ref{splittingprinciple}, there exists a map $\gp : \tilde{Y} \to Y$, which is the composition of a series of blow ups along invariant \sm\ centers with dimensions less than $\dim Y$, and a \glin{G}\ invertible sheaf $\L$ over $\tilde{Y}$ such that the sequence of \glin{G}\ sheaves
$$0 \to \L \to \gp^* \cat{E} \to (\gp^* \cat{E})/\L \to 0$$
is exact and $(\gp^* \cat{E})/\L$ is locally free with rank $r-1$. 

Define the tower $\tilde{\P} \to \tilde{Y}$ by pulling back each level, namely $\tilde{\P}_i = \P_i \x_Y \tilde{Y}$. Then, the sheaves in the construction at each level of $\tilde{\P}$ is exactly the same as $\P$ if we identify $\gp^* \cat{M}$ and $\cat{M}$. Thus, $\tilde{\P} \to \tilde{Y}$ is a \qadtower\ with the same total rank and number of sheaves. 

\vtab

Claim 1 : $\gp_*[\tilde{\P} \to \tilde{Y}] - [\P \to Y]$ lies in $\cobg{G}{}{Y}'.$

First, assume $\gp$ is given by a single blow up along some invariant \sm\ center $Z \subset Y$. Observe that $\tilde{\P}$ can be considered as the blow up of $\P$ along $\P|_Z$. By Lemma \ref{blowuprelation}, we obtain the equality
$$[\tilde{\P} \to \P] - [\id_{\P}] = - [Q_1 \to \P|_Z \to \P] + [Q_2 \to \P|_Z \to \P].$$
Pushing them down to $Y$, we get
$$\gp_*[\tilde{\P} \to \tilde{Y}] - [\P \to Y] = - [Q_1 \to \P|_Z \to Z \to Y] + [Q_2 \to \P|_Z \to Z \to Y].$$
Notice that the tower $\P|_Z \to Z$ is trivially quasi-admissible and, by Remark \ref{blowuprelationrmk}, the sheaves involved in the construction of $Q_1$, $Q_2 \to \P|_Z$ are either of the form $\O(m)$ or $\dual{\nbundle{\P|_Z}{\P}} \cong \dual{\nbundle{Z}{Y}}$ in our notation. That implies $Q_1 \to Z$ and $Q_2 \to Z$ are both \qadtower s. The result then follows from the fact that $\dim Z < \dim X$. The general case with more blow ups follows easily from the fact that $\gp_*\, \cobg{G}{}{\tilde{Y}}' \subset \cobg{G}{}{Y}'$. \claimend

\vtab

Hence, w\withoutlog, we may assume the splitting 
$$0 \to \L \to \cat{E} \to \cat{E}/\L \to 0$$
happens in the original tower $\P \to Y$. Next, we will construct towers $\hat{\P}$, $\P' \to Y$ in a similar manner as in the proof of Lemma \ref{towertwisting}. Define $$\hat{\P}_i \defeq \P((\oplus_j \cat{E}_j) \oplus \cat{E} \oplus \L) \text{\tab and \tab} \P_i' \defeq \P((\oplus_j \cat{E}_j) \oplus (\cat{E}/\L) \oplus \L).$$ Then, by Lemma \ref{splittingisenough}, $\P_i$ and $\P_i'$ are \equi ly linearly equivalent invariant \sm\ divisors on $\hat{\P}_i$ and they intersect transversely. For each level $k > i$, we construct $\hat{\P}_k$ by the same set of sheaves used in $\P_k$ to form a tower 
$$\hat{\P} \defeq \hat{\P}_n \to \cdots \to \hat{\P}_i \to \P_{i-1} \to \cdots \to Y.$$ 
Also, for each level $k > i$, we construct $\P_k'$ by fiber product, namely, $\P_k' \defeq \hat{\P}_k \x_{\hat{\P}_i} \P_i'$ to form another tower 
$$\P' \defeq \P_n' \to \cdots \to \P_i' \to \P_{i-1} \to \cdots \to Y.$$

In this case, $\P \sim \P'$ as invariant \sm\ divisors on $\hat{\P}$ and they intersect transversely. By $GDPR(1,1)$, we have $[\P \embed \hat{\P}] = [\P' \embed \hat{\P}]$ and hence, 
$$[\P \to Y] = [\P' \to Y]$$
as elements in $\cobg{G}{}{Y}$. Observe that, for each level $k \neq i$, the set of sheaves involved in the construction of $\P_k'$ is exactly the same as those of $\P_k$ in our notation. For level $i$, by definition, $\P_i' = \P((\oplus_j \cat{E}_j) \oplus (\cat{E}/\L) \oplus \L)$. Hence, $\P' \to Y$ is a \qadtower\ with the same total rank as $\P \to Y$ and one higher number of sheaves. By repeating this procedure, we will obtain the highest number of sheaves possible : the number of sheaves is equal to the total rank. That means all sheaves involved in the construction of the \qadtower\ are \glin{G}\ invertible sheaves. 

\vtab

Step 2 : Reduction to an \adtower.

By step 1, we may assume $\P \to Y$ is a \qadtower\ constructed by \glin{G}\ invertible sheaves only. For each $\L \in \picard{G}{Y}$ used in the construction, there is an invariant divisor $D_{\L}$ on $Y$ and a sheaf $\cat{N}_{\L} \in \picard{G}{\pt}$ such that
$$\L \cong \O_Y(D_{\L}) \otimes \gp_k^* \cat{N}_{\L}$$
by Proposition \ref{lbstructure}. We can then represent such a (Weil) divisor as a linear combination of prime divisors $\{D_{\L,k}\}$ on $Y$. Let
$$\{D_1, \ldots, D_N\} \defeq \{D_{\L,k} \text{ where $\L$ is used in the construction of $\P \to Y$} \}.$$
Consider $\cup_{k = 1}^N D_k$ as a reduced closed subscheme of $Y$. Apply the embedded desingularization Theorem in \cite{embeddeddesingularization} on $\cup_{k = 1}^N D_k \embed Y$, we obtain a map $\gp : \tilde{Y} \to Y$, which is the composition of a series of blow ups along invariant \sm\ centers such that $\stricttransform{\cup_{k = 1}^N D_k}$ is \sm. Let $\{E_l\}$ be the set of exceptional divisors. Since $\stricttransform{\cup_{k = 1}^N D_k} = \cup_{k = 1}^N \stricttransform{D_k}$ is \sm, the strict transforms $\{\stricttransform{D_k}\}$ are disjoint invariant \sm\ divisors on $\tilde{Y}$. Moreover, we have
$$\gp^* \O_Y(D_k) \cong \O_{\tilde{Y}}(\stricttransform{D_k} + \sum_l m_l E_l)$$
for some integers $m_l$ and all invariant divisors involved are \sm. Hence, $\gp^* \L$ are all admissible and the tower $\tilde{\P} \to \tilde{X}$ defined by $\tilde{\P}_i \defeq \P_i \x_X \tilde{X}$ becomes admissible. By claim 1, we reduce to the case when $\P \to Y$ is an \adtower.

\vtab

Step 3 : Reduction to an \adtower\ with $\P_1 = \P(\gp_k^* \cat{E}_1)$ where $\cat{E}_1$ is a \glin{G}\ locally free sheaf over $\pt$.

By step 2, we may assume $\P \to Y$ is an \adtower. Consider the first level $\P_1 = \P(\oplus^r_{j=1} \L_j)$. Since the sheaves $\L_j$ are admissible, we have $\L_j \cong \O_Y(\sum_k \pm D_{jk}) \otimes \gp_k^* \cat{N}_j$ for some invariant \sm\ divisors $D_{jk}$ on $Y$ and some $\cat{N}_j \in \picard{G}{\pt}$. By lemma \ref{towertwisting}, we can twist $\P \to Y$ to $\P' \to Y$ so that $\P_1' = \P( (\oplus_{j \neq p}\ \L_j) \oplus \L_p(D))$ and the difference will be given by \qadtower s $Q \to D$. Notice that
$$[Q \to D \embed Y] = \sum_i\ [Q_i \to D_i \embed Y]$$
where $\{D_i\}$ are the $G$-components of $D$ and $Q_i \defeq Q \x_D D_i$ defines a \qadtower\ over $D_i$. So, $[\P \to Y] - [\P' \to Y]$ lie in $\cobg{G}{}{Y}'$. Hence, by twisting each $\L_j$ by suitable choices of $D$, we may assume there exists a sheaf $\L' \in \apicard{G}{Y}$ such that
$$\L_j \cong \L' \otimes \gp_k^* \cat{N}_j$$
for all $j$. In other words, 
$$\P_1 = \P(\L' \otimes \gp_k^* \cat{E}_1)$$
where $\cat{E}_1 \defeq \oplus_j\ \cat{N}_j$ is a \glin{G}\ locally free sheaf over $\pt$. Notice that $\P(\L' \otimes \gp_k^*\cat{E}_1)$ is isomorphic to $\P(\gp_k^*\cat{E}_1)$ as \equi\ \proj\ bundles over $Y$. If we define $\P_1' \defeq \P(\gp_k^*\cat{E}_1)$ and $\P_i' \defeq \P_i \x_{\P_1} \P_1'$ for all $2 \leq i \leq n$, then we obtain a tower $\P' \to Y$ which is isomorphic to $\P \to Y$. Since all the sheaves involved in the construction of $\P'$ are invertible, by Remark \ref{invshimplyqadtower}, $\P' \to Y$ is a \qadtower. By applying step 2 on $\P' \to Y$, we obtain an \adtower\ $\tilde{\P} \to \tilde{Y}$. Then, the result follows from claim 1 and the fact that 
$$\tilde{\P}_1 = \P_1' \x_Y \tilde{Y} \cong \P(\gp_k^* \cat{E}_1).$$

\vtab

Step 4 : Finish the induction step.

By step 3, it is enough to prove the statement in the case when $\P \to Y$ is an \adtower\ with $\P_1 = \P(\gp_k^* \cat{E}_1)$. Consider the second level $\P_2 = \P(\oplus^r_{j=1} \L_j)$. Since the sheaves $\L_j$ are admissible and 
$$\apicard{G}{\P_1} = \apicard{G}{Y} + \Z \O_{\P_1}(1),$$
by the same trick as in step 3, we can twist $\P \to Y$ until there exists a sheaf $\L' \in \apicard{G}{Y}$ such that
$$\L_j \cong \L' \otimes \O_{\P_1}(m_j) \otimes \gp_k^* \cat{N}_j$$
for all $j$. By Remark \ref{onlyleveltwisted}, the twisting will not affect $\P_1$. By defining 
$$\cat{E}_2 \defeq \oplus_j (\O_{\P(\cat{E}_1)}(m_j) \otimes \cat{N}_j)$$
and $p_1 : \P_1 = \P(\gp_k^* \cat{E}_1) \to \P(\cat{E}_1)$, we obtain an isomorphism 
$$\P_2 = \P(\L' \otimes p_1^* \cat{E}_2) \cong \P(p_1^* \cat{E}_2).$$
Simiarly, we get an isomorphic \qadtower\ $\P' \to Y$ and then, an \adtower\ $\tilde{P} \to \tilde{Y}$ by blow ups. Thus, we have the following commutative diagram :
\begin{center}
$\begin{CD}
\P(\cat{E}_2) @<<< \P(p_1^* \cat{E}_2) = \P'_2 @<<< \P(q_1^* p_1^* \cat{E}_2) = \tilde{\P}_2 \\
@VVV @VVV @VVV \\
\P(\cat{E}_1) @<{p_1}<< \P(\gp_k^* \cat{E}_1) = \P'_1 @<{q_1}<< \P(\gp_k^* \cat{E}_1) = \tilde{\P}_1 \\
@VVV @VVV @VVV \\
\pt @<<< Y @<<< \tilde{Y}
\end{CD}$
\end{center}
\noindent That handles the second level. By repeating the process until level $n$, we obtain an \adtower
$$Q = Q_n = \P(\cat{E}_n) \to \cdots \to \P(\cat{E}_1) \to Q_0 = \pt$$ 
such that 
$$[\P \to Y] = [Y \x Q \to Y] = [Q \to \pt][Y \to Y].$$ 

\vtab

Step 5 : $\dim Y = 0$ case.

In this case, any \glin{G}\ locally free sheaf $\cat{E}$ over $Y$ splits into the direct sum of \glin{G}\ invertible sheaves (by direct calculation or Theorem \ref{splittingprinciple}). Moreover, if $\L$ is a sheaf in $\picard{G}{Y}$, then, by Proposition \ref{lbstructure}, we have $\L \cong \O_Y(D) \otimes \gp_k^* \cat{N} \cong \gp_k^* \cat{N}.$ That means 
$$\P_1 = \P(\oplus \L_j) = \P(\oplus \gp_k^* \cat{N}_j) = Q_1 \x Y$$
where $Q_1 \defeq \P(\oplus \cat{N}_j)$. The same argument applies to higher levels. Hence, 
$$[\P \to Y] = [Q \to \pt] [Y \to Y]$$ 
with \adtower\ $Q \to \pt$.
\end{proof}

\vtab

\subsection{Generators for $\cobg{G}{}{\pt}$}

We are now in position to prove the generators Theorem. First of all, we will prove that any two birational objects $Y$, $Y' \in \gsmcat{G}$ agree in some truncated theory.

\begin{defn}
{\rm
For any $X \in \gsmcat{G}$, we define the abelian group $\overline{\cobglow{G}{}{X}}$ as the quotient of $\cobglow{G}{}{X}$ by the subgroup generated by elements of the form $[Z][Y \to X]$ where $[Z]$ is in $\cobglow{G}{\geq 1}{\pt}'$ and $[Y \to X]$ is in $\cobglow{G}{}{X}$, i.e.
$$\overline{\cobglow{G}{}{X}} \defeq \cobglow{G}{}{X}\, /\, \cobglow{G}{\geq 1}{\pt}'\, \cobglow{G}{}{X}.$$
}\end{defn}

\begin{rmk}
\label{additivefgl}
\rm{
$\overline{\cob{G}}$ can be considered as a theory on $\gsmcat{G}$ with \proj\ push-forward, \sm\ pull-back, Chern class operator (for \nice\ invertible sheaves) and external product. In this truncated theory, the \fgl\ becomes additive, i.e. 
$$c(\L \otimes \cat{M}) = c(\L) + c(\cat{M}).$$

\begin{proof}
In section 7.3 in \cite{LePa}, the abelian group $\go(\pt)'$ is defined as the subgroup of $\go(\pt)$ generated by \adtower s (Without group action, the notions of ``\adtower'' in \cite{LePa} and in our paper are equivalent). By Corollary 7.5 and equation 8.1 in \cite{LePa}, the coefficients $a_{ij}$ used in the \fgl\ in the theory $\go$ are all inside $\go_{\geq 1}(\pt)'$. Then, the result follows from the fact that the \fgl\ in the theory $\go$ and the \fgl\ in our theory $\cob{G}$ share the same set of coefficients $a_{ij}$ if we consider $\go(\pt) \embed \cobg{G}{}{\pt}$.
\end{proof}
}
\end{rmk}

\begin{prop}
\label{birationalimplyequal}
Suppose $Y$, $Y' \in \gsmcat{G}$ are both \proj\ and \girred. If they are \equi ly birational, then $[Y] = [Y']$ as elements in $\overline{\cobg{G}{}{\pt}}$.
\end{prop}

\begin{proof}
By the \equi\ weak factorization theorem (Theorem 0.3.1) in \cite{AbKaMaWl}, there exists a sequence of blowups and blowdowns along \sm\ invariant centers to go from $Y$ to $Y'$. So, it is enough to consider a single blowup. By Lemma \ref{blowuprelation}, 
$$[Bl_Z\ Y \to Y] - [\id_Y] = -[\P_1 \to Z \to Y] + [\P_2 \to Z \to Y]$$
as elements in $\cobg{G}{}{Y}.$ Pushing them down to $\cobg{G}{}{\pt}$ gives 
$$[Bl_Z\ Y] - [Y] = - [\P_1 \to Z \to \pt] + [\P_2 \to Z \to \pt]$$
as elements in $\cobg{G}{}{\pt}.$ For simplicity, assume $Z$ is \girred. By Remark \ref{blowuprelationrmk}, $\P_1$, $\P_2 \to Z$ are both \qadtower s. By Proposition \ref{towerreduction}, $[\P_i \to Z] = \sum a\,[Z' \to Z]$ for some $a \in \cobg{G}{}{\pt}'$ and $Z' \in \gsmcat{G}$ such that $\dim Z' \leq \dim Z$. Since $\dim \P_i = \dim Y > \dim Z,$ the elements $\{a\}$ are all in $\cobg{\geq 1}{G}{\pt}'$. Hence, the element $[\P_i \to Z \to \pt]$ vanishes in $\overline{\cobg{G}{}{\pt}}$.
\end{proof}

Finally, we are ready to prove our main Theorem. The generators of our \equi\ algebraic cobordims ring $\cobg{G}{}{\pt}$, as a $\lazard$-algebra, will be \adtower s over $\pt$ and some ``exceptional objects''. For an integer $n \geq 0$ and a pair of subgroups $G \supset H \supset H'$, since $G$ is abelian, we can write
$$H/H' \cong H_1 \x \cdots \x H_a$$
where $H_i$ is a cyclic group of order $M_i \defeq p_i^{m_i}$ for some prime $p_i.$ Let $\ga_i$ be a generator of $H_i$. Define a $(H/H')$-action on ${\rm Proj}\ k[x_0, \ldots, x_n,v_1, \ldots, v_a]$ as the following. First, $H/H'$ acts on $x_0, \ldots, x_n$ trivially. Then, for all $i$, the subgroup $H_i$ acts on $v_i$ by $\ga_i \cdot v_i = \gx_i v_i$ for some primitive $M_i$-th \rou\ $\gx_i$. For all $j \neq i$, the subgroup $H_j$ acts on $v_i$ trivially. 

\begin{lemma}
\label{exceptionalobjectwd}
There exist homogeneous polynomials $g_1, \ldots, g_a \in k[x_0, \ldots, x_n]$ with degrees $M_1, \ldots, M_a$ respectively, such that the \proj\ variety
$${\rm Proj}\ k[x_0, \ldots, x_n,v_1, \ldots, v_a]\, / \, (v_1^{M_1} - g_1, \ldots, v_a^{M_a} - g_a),$$
is \sm\ and has dimension $n$.
\end{lemma}

\begin{proof}
Let $U$ be the open subscheme $\cup_{i=0}^n D(x_i)$ of ${\rm Proj}\ k[x_0, \ldots, x_n,v_1, \ldots, v_a]$. For $1 \leq i \leq a$, let $\gps_i : U \to \P^{N_i}$ be the $(H/H')$-\equi\ map sending $(x_0 ; \ldots ; x_n ; v_1 ; \ldots ; v_a)$ to 
$$(x_0^{M_i} ; x_0^{M_i-1}x_1 ; \ldots ; x_n^{M_i} ; v_i^{M_i})$$
(the first $N_i-1$ coordinates run though all degree $M_i$ monomials given by $x_0, \ldots, x_n$). By Lemma \ref{bertinilemma}, there exist hyperplanes $H_i \subset \P^{N_i}$ such that $U \x_{\P^{N_1}} H_1 \x_{\P^{N_2}} H_2 \x_{\P^{N_3}} \cdots \x_{\P^{N_a}} H_a$ is \sm\ and has dimension $n$. The result then follows by observing each $H_i$ defines a homogeneous polynomial $g_i$ with degree $M_i$ and 
$$U \x_{\P^{N_1}} H_1 \x_{\P^{N_2}} H_2 \x_{\P^{N_3}} \cdots \x_{\P^{N_a}} H_a = {\rm Proj}\ k[x_0, \ldots, x_n,v_1, \ldots, v_a]\, / \, (v_1^{M_1} - g_1, \ldots, v_a^{M_a} - g_a).$$
\end{proof}

Pick $g_1, \ldots, g_a$ as in Lemma \ref{exceptionalobjectwd}. Let $X$ be the \proj\ variety
$${\rm Proj}\ k[x_0, \ldots, x_n,v_1, \ldots, v_a]\, / \, (v_1^{M_1} - g_1, \ldots, v_a^{M_a} - g_a).$$
Then, $X$ is in $(H/H')$-\textit{Sm}. Fix a set of representatives $\{\gb_j\}$ of $G/H$. The exceptional object $E_{n,H,H'}$ is defined as $G/H \x X$ such that for all $\ga \in G$ and $(\gb_j, y) \in E_{n,H,H'}$,
$$\ga \cdot (\gb_j, y) \defeq (\gb_k,\ \gg \cdot y)$$
where $\gb_k \in G/H$ and $\gg \in H$ are uniquely determined by the equality $\ga \gb_j = \gb_k \gg$. We will see that the element $[E_{n,H,H'}] \in \cobg{G}{}{\pt}$ is independent of the choice of $\{g_i\}$.

\begin{thm}
\label{setofgenerator}
If the pair $(G,k)$ is split, then $\cobg{G}{}{\pt}$ is generated by the set of exceptional objects $\{E_{n, H, H'}\ |\ n \geq 0 \text{ and } G \supset H \supset H'\}$ and the set of \adtower s over $\pt$ as a $\lazard$-algebra.
\end{thm}

\begin{proof}
Let $S$ be the set of generators mentioned in the statement, i.e. $S \defeq \{ [E_{n,H,H'}],\ [\P] \}$. Consider the following diagram of abelian groups :

\begin{center}
$\begin{array}{ccc}
\cobglow{G}{n}{\pt} & &  \\
\downarrow & \stackrel{P_n}{\searrow} &  \\
\overline{\cobglow{G}{n}{\pt}} & \stackrel{\overline{P_n}}{\dashrightarrow} & \cobglow{G}{}{\pt}\, / \,\lazard[S]
\end{array}$
\end{center}

\noindent Our goal is to prove $S$ gives a set of generator of $\cobg{G}{}{\pt}$ as $\lazard$-algebra. It is obviously enough to show that $P_n = 0$ for all $n$. Suppose we have shown that $P_0 = P_1 = \cdots = P_{n-1} = 0.$ Then, since $\cobglow{G}{}{\pt}'$ is a subgroup of $\lazard[S]$ and
$$( \cobglow{G}{\geq 1}{\pt}'\, \cobglow{G}{}{\pt} ) \cap \cobglow{G}{n}{\pt} = \sum_{i = 1}^n \cobglow{G}{i}{\pt}'\, \cobglow{G}{n-i}{\pt},$$
the homo\morp\ $\overline{P_n}$ is well-defined and the diagram is commutative. In addition, $\overline{P_n} = 0$ will imply $P_n = 0$ and $P_0$, $\overline{P_0}$ agree. So, it is enough to show that $\overline{P_n} = 0$ for all $n$.

Suppose $n \geq 0$ and $[Y] \in \overline{\cobglow{G}{n}{\pt}}$ is \girred. Assume $Y$ is irreducible and the $G$-action is faithful first. Let $G = G_1 \x \cdots \x G_a$ where $G_i$ is a cyclic group of order $M_i \defeq p_i^{m_i}$ for some prime $p_i$ and $\ga_i$ be a generator of $G_i$. 

\vtab

Claim 1 : 
$$k(Y) \cong k(x_1, \ldots, x_n)[x_{n+1}, v_1, \ldots, v_a]\, / \,(f, v_1^{M_1} - g_1, \ldots, v_a^{M_a} - g_a)$$ 
for some $f$, $g_i \in k[x_1, \ldots, x_{n+1}]$ such that the $G$-action on $x_i$ is trivial, $G_j$ acts on $v_i$ trivially if $i \neq j$ and $\ga_i \cdot v_i = \gx_i v_i$ where $\gx \in k$ is a primitive $M_i$-th \rou.

Denote the function field $k(Y)$ by $K$. Since the $G$-action on $Y$ is faithful, the degree of the extension $K/K^G$ is equal to the order of $G$ and it is a Galois extension (separable because $\char{k} = 0$). Let $K_i$ be the subfield of $K$ consists of elements fixed by $\prod_{j \neq i} G_j$. Then, $K = K_1 \cdots K_a$, the intersection of any $K_i \neq K_j$ is $K^G$ and the extension $K_i/K^G$ is Galois with ${\rm Gal}(K_i/K^G) \cong G_i$. 

Since the dimension of the scheme $Y/G$ is $n$, the field $k(Y/G) \cong K^G$ has transcendence degree $n$ over $k$. So, there is an element $f \in k[x_1, \ldots, x_{n+1}]$ such that $K^G \cong k(x_1, \ldots, x_n)[x_{n+1}] / (f)$. Consider $K_i$ as a $G_i$-\repn\ over $K^G$. Since the pair $(G_i, K^G)$ is split, $K_i$ can be written as direct sum of 1-dimensional $G_i$-\repn s over $K^G$. Then, the action of $\ga_i$ on at least one of the $G_i$-\repn s is given by $\gx_i$. Let $b_i$ be a generator of such \repn. Since $b_i^{M_i}$ is fixed by $G_i$, it is in $K^G$. Denote it by $g_i$. W\withoutlog, $g_i \in k[x_1, \ldots, x_{n+1}]$. Consider the polynomial
$$v^{M_i} - g_i = \prod_{j = 0}^{M_i - 1} (v - \gx^j b_i) \in K^G[v].$$
It is irreducible because if $j < M_i$, then $\ga$ does not fix $b_i^j$, hence $b_i^j \notin K^G$. Since $v^{M_i} - g_i$ has degree $M_i$, the field $K_i$ has to be generated by $b_i$. In other words, 
$$K_i \cong k(x_1, \ldots, x_n)[x_{n+1}, v_i]\, / \,(f, v_i^{M_i} - g_i).$$
Also, the $G$-action on $v_i$ corresponds to the $G$-action on $b_i$, which is exactly as the one described in the statement.  \claimend

\vtab

Let 
$$Y' \defeq {\rm Proj}\ k[x_0, \ldots, x_{n+1}, v_1, \ldots, v_a]\, / \,(f, v_1^{M_1} - g_1, \ldots, v_a^{M_a} - g_a)$$ 
and 
$$P' \defeq {\rm Proj}\ k[x_0, \ldots, x_{n+1}, v_1, \ldots, v_a]$$ 
where $f$, $g_i \in k[x_0, \ldots, x_{n+1}]$ are homogeneous polynomials with degree $d$ and $M_i$ respectively, the $G$-action on $x_i$ is trivial and the $G$-action on $v_j$ are the one described in claim 1. For simplicity, we will denote $P'$ simply as ${\rm Proj}\ k[x,v]$. By claim 1, $Y'$ is \equi ly birational to $Y$. By applying the embedded desingularization theorem \cite{embeddeddesingularization} on $Y' \embed P'$, there is a commutative diagram
\squarediagram{\stricttransform{Y'}}{P}{Y'}{P'}
where $\stricttransform{Y'}$, $P$ are both in $\gsmcat{G}$. Since $\stricttransform{Y'}$ is \sm\ and \equi ly birational to $Y'$, by Proposition \ref{birationalimplyequal}, we may assume $\stricttransform{Y'} = Y$. Moreover, since $P \to P'$ is \proj, by Proposition \ref{equiprojmorp}, there exist free variables $y_0, \ldots, y_m$ with $G$-actions and a set of polynomials $\{h\} \subset k[x, v, y_0, \ldots, y_m]$  which are bihomogeneous \wrt\ $(x,v)$ and $y$ such that 
$$P \cong {\rm BiProj}\ k[x, v][y] / (h)$$ 
and
$$Y = {\rm BiProj}\ k[x, v][y]\, / \,(h) + (f, v_1^{M_1} - g_1, \ldots, v_a^{M_a} - g_a).$$ 

Define a set of indices
$$J \defeq \{ \text{monomial in $k[x]$ with degree $d$} \}\ \disjoint\ \coprod_i \{ \text{monomial in $k[x]$ with degree $M_i$} \}.$$
Let $C \defeq k[\{c_j\ |\ j \in J\}]$ be the polynomial ring generated by free variables indexed by $J$. Then, $f(x)$ can be considered as $f(c_0,x)$ for some $c_0 \in \spec{C}$ and similarly for $g_i$. Let 
$$T' \defeq {\rm Proj}\ C[x,v]\, / \,( f(c,x), v_1^{M_1} - g_1(c,x), \ldots, v_a^{M_a} - g_a(c,x) ).$$
If we assign a trivial $G$-action to $\spec{C}$, then there is an \equi, \proj, surjective map $\gph' : T' \to \spec{C}$ with fiber $T'_{c_0} \cong Y'$ and $T'$ is a closed subscheme of $\spec{C} \x P' \cong {\rm Proj}\ C[x,v].$ Also let
$$T \defeq {\rm BiProj}\ C[x, v][y]\, / \,(h) + ( f(c,x), v_1^{M_1} - g_1(c,x), \ldots, v_a^{M_a} - g_a(c,x) ).$$
Similarly, there is an \equi, \proj, surjective map $\gph : T \to \spec{C}$ with fiber $T_{c_0} \cong Y$ and $T$ is a closed subscheme of $\spec{C} \x P \cong {\rm BiProj}\ C[x, v][y] / (h).$

\vtab

Claim 2 : $T$ is in $\gsmcat{G}$ and has dimension $\dim \spec{C} + n$. 

W\withoutlog, $k$ is algebraically closed. Notice that $T$ is cut out from $\spec{C} \x P$, which is \sm\ and has relative dimension $n + a + 1$ over $\spec{C}$, by the equations $f(c,x)$ and $v_i^{M_i} - g_i(c,x)$. We will show that the gradients $\{ \gradient{f(c,x)}, \gradient{( v_i^{M_i} - g_i(c,x)) } \}$ are linearly independent and they are also linearly independent to any $\gradient{h}$ at any closed point in $T$. Since $h$ is in $k[x, v][y]$, $\gph_* \gradient{h} = 0$. So, it will be enough to show that the vectors
$$\{ \gph_* \gradient{f(c,x)},\ \gph_* \gradient{(v_i^{M_i} - g_i(c,x))} \}$$ 
are linearly independent. Over $D(x_j)$, if we denote $x_k / x_j$ by $t_k$, then we have
$$\gph_* \gradient{f(c,x)} = (t_0^d, t_0^{d-1}t_1, \ldots, t_{n+1}^d, 0, \ldots, 0)$$
and 
$$\gph_* \gradient{(v_i^{M_i} - g_i(c,x))} = -(0, \ldots, 0, t_0^{M_i}, t_0^{M_i-1}t_1, \ldots, t_{n+1}^{M_i}, 0, \ldots, 0)$$
(zero except coordinates corresponding to the coefficients of $g_i$). Moreover, over $D(v_j)$, if we denote $x_k / v_j$ by $t_k$, then we obtain the same equations for the vectors $\gph_* \gradient{f(c,x)}$ and $\gph_* \gradient{(v_i^{M_i} - g_i(c,x))}$. Thus, they are linearly independent as long as $(x_0; \cdots ; x_{n+1}) \neq 0$. Suppose $x = (x_0 ; \cdots ; x_{n+1}) = 0$ for a certain closed point in $T$. Then, $v_1, \ldots, v_a$ are all zero too. So, the coordinate of this point is $(c,0;0,y) \in T \subset C \x P$. We then get a contradiction by realizing that the map $C \x P \to P \to P'$ will send $(c,0;0,y)$ to $(0;0) \in P'$. \claimend

\vtab

The same argument also shows that $T'$ is in $\gsmcat{G}$ and has dimension $\dim \spec{C} + n$. Notice that $T$ and $\spec{C}$ are both \sm, the map $\gph$ is \proj, surjective and has relative dimension $n$ and the fiber $T_{c_0}$ is \sm\ with dimension $n$. So, the map $\gph$ is \sm\ if restricted in an open neighborhood of $c_0$ (because the point $c_0$ is not in the image of $\{\text{critical point}\}$, which is closed). Call such a neighborhood $U_0$. Pick a point $c_1 = (c_{1j}) \in \spec{C}$ such that the fiber $T'_{c_1}$ is in $\gsmcat{G}$ with dimension $n$ (such point exists by Lemma \ref{exceptionalobjectwd}). Similarly, the map $\gph' : T' \to \spec{C}$ is \sm\ if restricted in an open neighborhood of $c_1$. Call such a neighborhood $U_1$. 

\vtab

Claim 3 : There exists an \equi, \proj, birational map $\gm : T \to T'$ of schemes over $\spec{C}$. 

The map $\gm$ is given by the restriction of the map $\spec{C} \x P \to \spec{C} \x P'$ which sends $(c,x;v,y)$ to $(c,x;v)$. So, it is clearly \equi\ and \proj. Notice that $P \to P'$ is birational. That means if $\ge_1$ is the generic point of $P'$, then $P_{\ge_1} \to \spec{\ge_1}$ is an isomorphism, i.e. ${\rm Proj}\ k(x_*, v_*)[y] / (h_*) \iso \spec{k(x_*, v_*)}$ where $x_*$, $v_*$ and $h_*$ are the dehomogenizations of $x$, $v$ and $h$ \wrt\ $x_0$, respectively. Let $\ge_2$ be the generic point of $T'$, as scheme over $\spec{C}$. Then,
\begin{eqnarray}
T_{\ge_2} &\cong& {\rm Proj}\ C(x_*, v_*)[y]\, / \,(h_*) + ( f_*, {v_1}_*^{M_1} - {g_1}_*, \ldots, {v_a}_*^{M_a} - {g_a}_* ) \nonumber\\
&\cong& \spec{C(x_*, v_*)\, / \, ( f_*, {v_1}_*^{M_1} - {g_1}_*, \ldots, {v_a}_*^{M_a} - {g_a}_* ) } \nonumber\\
&\cong& \spec{\ge_2}. \nonumber
\end{eqnarray}
That means $\gm$ is birational, as a \morp\ of schemes over $\spec{C}$. \claimend

\vtab

Denote the open subscheme $U_0 \cap U_1 \subset \spec{C}$ by $U$. Then, $\gph : T|_U \to U$ and $\gph' : T'|_U \to U$ are both \sm\ and $\gm : T|_U \to T'|_U$ has birational fibers (over $U$). Also, denote the affine line in $\spec{C}$ connecting $c_0$ and $c_1$ by $L$ and pick a closed point $c_2 \in U \cap L$. Consider the \equi, \proj\ map $\gph : T|_L \to L$. It is \sm\ over $U_0 \cap L$. That means $\singular{T|_L}$ is disjoint from the fibers $T_{c_0}$ and $T_{c_2}$. By resolution of singularities (Theorem 1.6 in \cite{embeddeddesingularization}), we can assume $T|_L$ is \sm\ (The blow ups will not affect the two fibers). Now, $T|_L$ has fibers $T_{c_0}$ and $T_{c_2}$ which are both \sm\ invariant divisors. By Proposition \ref{equiembedwithsmclosure}, we can extend $T|_L \to L$ to some \equi, \proj\ map $\overline{T} \to \P^1$ where $\overline{T}$ is in $\gsmcat{G}$. Then, $GDPR(1,1)$ will imply 
$$[T_{c_0} \embed \overline{T}] = [T_{c_2} \embed \overline{T}]$$
as elements in $\cobg{G}{}{\overline{T}}$. Push them down to $\pt$, we got $[T_{c_0}] = [T_{c_2}]$. By applying the same argument on $\gph' : T'|_L \to L$, we got $[T'_{c_1}] = [T'_{c_2}]$. Hence,
$$[Y] = [T_{c_0}] = [T_{c_2}] = [T'_{c_2}] = [T'_{c_1}]$$
as elements in $\overline{\cobg{G}{}{\pt}}$ by Proposition \ref{birationalimplyequal} and the fact that $T_{c_2}, T'_{c_2}$ are birational and are both \sm.

Because of the freedom of choice of $c_1 = (c_{1j})$, we can assume 
$$Y \cong {\rm Proj}\ k[x, v]\, / \,(f, v_1^{M_1} - g_1, \ldots, v_a^{M_a} - g_a)$$
for any choice of $f(x)$, $g_i(x)$ as long as the degrees of $f$, $g_i$ are $d$, $M_i$ respectively and $Y$ is \sm. Consider the \equi\ map 
$$\gps : W \defeq {\rm Proj}\ k[x, v]\, / \,(v_1^{M_1} - g_1, \ldots, v_a^{M_a} - g_a) \to {\rm Proj}\ k[x] \cong \P^{n+1}.$$ 
Then, $Y$ can be considered as the preimage of a generic degree $d$ hypersurface. More precisely, as elements in $\overline{\cobg{G}{}{W}}$,
$$[Y \embed W] = c(\gps^* \O(d)) [\id_W] = d \,c(\gps^*\O(1)) [\id_W]$$
because $\gps^*\O(1)$ is \nice\ and \fgl\ becomes additive by Remark \ref{additivefgl}. In other words, it is enough to consider the case when $d = 1$. W\withoutlog, we may assume $f(x) = x_{n+1}$. Hence, we have 
$$Y \cong {\rm Proj}\ k[x_0, \ldots, x_n, v_1, \ldots, v_a]\, / \,(v_1^{M_1} - g_1(x), \ldots, v_a^{M_a} - g_a(x)),$$
which is the exceptional object $E_{n,G,\{1\}}$. So, $\overline{P_n}[Y] = 0$. That proves the case when $Y$ is irreducible with faithful $G$-action.

If the $G$-action on $Y \in \gsmcat{G}$ is faithful, but $Y$ is reducible, then $Y \cong G/H \x X$ for some subgroup $H \subset G$ and some irreducible $X \in \gsmcat{H}$. By applying claim 1 on $X$ with $H$-action, we can define $G/H \x X'$, $G/H \x X$, $G/H \x P'$ and $G/H \x P$ in the same manner to obtain the following commutative diagram :

\squarediagram{G/H \x X}{G/H \x P}{G/H \x X'}{G/H \x P'}

\noindent We can also define the polynomial ring $C$ and the $C$-schemes $G/H \x T'$ and $G/H \x T$. We will also have $G$-\equi, \proj\ maps  $\gph : G/H \x T \to \spec{C}$ and $\gph' : G/H \x T' \to \spec{C}$ such that $\gph$ is \sm\ around $c_0$ and $\gph'$ is \sm\ around some $c_1 = (c_{1j})$. Similarly, the natural map $\gm : G/H \x T \to G/H \x T'$ will also be $G$-\equi, \proj\ and has birational fibers over $\spec{C}$. Hence, as before,
$$[Y] = [G/H \x X] = [(G/H \x T)_{c_0}] = [(G/H \x T)_{c_2}] = [(G/H \x T')_{c_2}] = [(G/H \x T')_{c_1}].$$
In other words, we may assume 
$$X \cong {\rm Proj}\ k[x,v]\, / \,(f,v_1^{M_1} - g_1, \ldots, v_a^{M_a} - g_a).$$ 
Define $\gps : G/H \x W \to \P^{n+1}$ similarly to get the same reduction on $f$. We may further assume 
$$X \cong {\rm Proj}\ k[x,v]\, / \,(v_1^{M_1} - g_1, \ldots, v_a^{M_a} - g_a).$$
Hence, we have $Y \cong G/H \x X \cong E_{n,H,\{1\}}.$

In general, if we have subgroups $G \supset H \supset H'$ such that the $(G/H')$-action on $Y$ is faithful and $Y \cong G/H \x X$ for some irreducible $X \in \gsmcat{(H/H')}$, then we may assume
$$X \cong {\rm Proj}\ k[x_0, \ldots, x_n, v_1, \ldots, v_a]\, / \,(v_1^{M_1} - g_1(x), \ldots, v_a^{M_a} - g_a(x))$$
for some generic $g_1, \ldots, g_a$ where $v_1, \ldots, v_a$ are given by $H/H'$. Hence, $Y \cong E_{n,H,H'}$. That finishes the proof.
\end{proof}

\begin{rmk}
\label{redundantexceptionalobj}{\rm
Notice that we did not use the full power of the g\gdpr\ in our proof of Theorem \ref{setofgenerator}. More precisely, if we define our \equi\ algebraic cobordism theory by imposing the e\edpr\ $GDPR(2,1)$ alone, the same set of generators will still generate the \equi\ algebraic cobordism ring. But, with the aid of the g\gdpr, we can actually simplify the exceptional objects further.

Suppose the dimension of an exceptional object $E_{n,H,H'}$ is greater than the order of the group $H/H'$. Let 
$$W \defeq G/H \x {\rm Proj}\ k[x_0, \ldots, x_n,v_1, \ldots, v_a]\, / \, (v_1^{M_1} - g_1, \ldots, v_{a-1}^{M_{a-1}} - g_{a-1}).$$
Then, the invariant \sm\ divisor $G/H \x \{v_a^{M_a} = g_a\} = E_{n,H,H'}$ is \equi ly linearly equivalent to the sum of invariant \sm\ divisors $G/H \x \{x_i = 0\}$ where $i$ runs from $0$ to $M_a - 1$. Moreover, by the freedom of choice of $\{g_i\}$, we can assume 
$$E_{n,H,H'} + \sum_{i=0}^{M_a-1} G/H \x \{x_i = 0\}$$ 
is a \rsncd. Thus, by the g\gdpr\ $GDPR(M_a,1)$,
$$[E_{n,H,H'} \embed W] = \sum_{i = 0}^{M_a-1}\,[G/H \x \{x_i = 0\} \embed W]$$
as elements in $\overline{\cobg{G}{}{W}}$ (``extra terms'' are always of the form $[\P \to Z \embed W]$ where $\P \to Z$ is a \qadtower\ and $\dim Z < \dim \P = n$). In other words, it is enough to consider objects of the form 
$$G/H \x {\rm Proj}\ k[x_0, \ldots, x_{n-1},v_1, \ldots, v_a]\, / \, (v_1^{M_1} - g_1, \ldots, v_{a-1}^{M_{a-1}} - g_{a-1})$$
instead. Similarly, we can apply the same argument to reduce $E_{n,H,H'}$ into
$$G/H \x {\rm Proj}\ k[x_0, \ldots, x_{n-a},v_1, \ldots, v_a] = G/H \x \P(V)$$
for some $(H/H')$-\repn\ $V$. In particular, if the group $G$ is a cyclic group of prime order, then 
$$[G/H \x \P(V)] = [G/H]\,[\P(V)]$$
where $V$ is some $G$-\repn\ and $H$ can be either $G$ or the trivial group. Notice that $E_{0, \{1\}, \{1\}} \cong G$ and $\P(V)$ is an \adtower\ over $\pt$. Hence, only finite number of exceptional objects are needed to generate $\cobg{G}{}{\pt}$ in this case.
}\end{rmk}

\bigskip

\bigskip

\section{Fixed point map}
\label{fixedpointmapsection}

In this section, we will assume the ground field $k$ has characteristic 0 and the group $G$ is either finite or reductive. Notice that, as mentioned in section 1 of Ch I of \cite{Mu}, the notion linearly reductive and reductive are equivalent over a field of characteristic 0.

Before proving the main Theorem, we need the following Proposition.

\begin{prop}
\label{smoothfixedpointlocus}
For any object $X \in \gsmcat{G},$ the fixed point locus $X^G$ is \sm. Moreover, if $x \in X^G$ is a closed point, then there is no non-zero normal vector in $\nbundle{X^G}{X}|_x$ which has trivial $G$-action.
\end{prop}

\begin{proof}
By Proposition 3.4 in \cite{Ed}, the fixed point locus $X^G$ is \sm\ if $G$ is finite. In the case when $G$ is linearly reductive, let $x \in X^G$ be a closed point and $C(X,x)$ be the tangent cone of $X$ at $x$. By the Theorem 5.2 in \cite{Fo}, we have $C(X,x)^G = C(X^G,x)$. Since $X$ is \sm\ at $x$, the tangent cone $C(X,x)$ is isomorphic to $\A^d_{k(x)} = \A^d \x \spec{k(x)}$ where $d$ is the dimension of $\O_{X,x}$ ($G$ acts on $k(x)$ trivially). The group $G$ acts on $C(X,x)$ linearly, so $C(X,x)^G = C(X^G,x)$ is a linear subspace. Hence, the fixed point locus $X^G$ is \sm\ at $x$.

It remains to show the statement about the normal vector. For finite group $G$, by Proposition 3.2 in \cite{Ed}, we have $\cat{T}X^G|_x \cong (\cat{T}X|_x)^G$. Moreover, the following exact sequence of $G$-\repn s splits :
$$0 \to \cat{T}X^G|_x \to \cat{T}X|_x \to \nbundle{X^G}{X}|_x \to 0.$$
Hence, there is no non-zero normal vector of $X^G$ which is fixed by $G$.

When $G$ is reductive, 
$$\cat{T}X^G|_x \cong \cat{T} C(X^G,x)|_0 \cong \cat{T} C(X,x)^G|_0 \cong (\cat{T} C(X,x)|_0)^G \cong (\cat{T}X|_x)^G.$$
Then the result follows similarly.
\end{proof}

\begin{thm}
\label{fixedpointmap}
Suppose $X$ is an object in $\gsmcat{G}$ and $\{Z\}$ is the set of irreducible components of its fixed point locus $X^G$. Then, sending $[Y \to X]$ to $\sum_Z [Y^G \x_{X^G} Z \to Z]$ defines an abelian group homomorphism :
$$\fixptmap : \cobg{G}{}{X} \to \bigoplus_Z \go(Z).$$
\end{thm}

\begin{proof}
By Proposition \ref{smoothfixedpointlocus}, $Z$ is \sm\ and sending $[Y \to X]$ to $[Y^G \x_{X^G} Z \to Z]$ is well-defined at the level of $M_G(X)^+ \to M(Z)^+$. If $X^G$ is the empty set, then $\oplus_Z\ \go(Z) = 0$ and there is nothing to prove. So, we can assume $X^G$ is non-empty. The strategy of this proof is very similar to that of the Proposition \ref{gdprholdprop}. 

First of all, it is clearly enough to show the well-definedness of $\cat{F}$ \wrt\ one fixed component $Z$, i.e. $\cat{F}_Z[Y \to X] = [Y^G \x_{X^G} Z \to Z].$ Consider a g\gdpr\ setup given by $\gph : Y \to X$ with $\gdprdivisorsequi$ on $Y$. Let $\cat{G} : \cat{R} \to M_G(X)^+$ be the corresponding map. What we need to show is 
$$\cat{F}_Z \circ \cat{G}(G^X_{n,m}) = \cat{F}_Z \circ \cat{G}(G^Y_{m,n})$$ 
as elements in $\go(Z)$. 

For a general term $X_i \cdots U^p_k \cdots$ in $\cat{R}$,
\begin{eqnarray}
\cat{F}_Z \circ \cat{G}(X_i \cdots U^p_k \cdots) &=& \cat{F}_Z [A_i \x_Y \cdots \x_Y P^p_k \x_Y \cdots \to Y \to X] \nonumber\\
&=& [(A_i \x_Y \cdots \x_Y P^p_k \x_Y \cdots)^G \x_{X^G} Z \to Y^G \x_{X^G} Z \to Z]. \nonumber
\end{eqnarray}
If $Y^G \x_{X^G} Z$ is empty, then $\cat{F}_Z \circ \cat{G}(G^X_{n,m}) = \cat{F}_Z \circ \cat{G}(G^Y_{m,n}) = 0$. So, we may assume $Y^G \x_{X^G} Z$ is non-empty. Let $\{W\}$ be the set of irreducible components of $Y^G \x_{X^G} Z$ and $\gp_W : W \to Z$ be the natural \proj\ map. Let $\cat{G}' : \cat{R} \to M_G(Y)^+$ be the map corresponding to the GDPR setup given by $\id : Y \to Y$ with the same set of divisors on $Y$. Then,
\begin{eqnarray}
{\gp_W}_* \circ \cat{F}_W \circ \cat{G}'(X_i \cdots U^p_k \cdots) &=& {\gp_W}_* \circ \cat{F}_W [A_i \x_Y \cdots \x_Y P^p_k \x_Y \cdots \to Y] \nonumber\\
&=& {\gp_W}_* [(A_i \x_Y \cdots \x_Y P^p_k \x_Y \cdots)^G \x_{Y^G} W \to W] \nonumber\\
&=& [(A_i \x_Y \cdots \x_Y P^p_k \x_Y \cdots)^G \x_{Y^G} W \to W \to Z]. \nonumber
\end{eqnarray}
Hence, 
$$\cat{F}_Z \circ \cat{G} = \sum_W\ {\gp_W}_* \circ \cat{F}_W \circ \cat{G}'.$$
That means it is enough to prove 
$$\cat{F}_W \circ \cat{G}'(G^X_{n,m}) = \cat{F}_W \circ \cat{G}'(G^Y_{m,n})$$ 
as elements in $\go(W)$. In other words, we may assume $\gph = \id_X$. In particular, $X$ is \equidim. For simplicity, we will denote $\cat{F}_Z$ by $\cat{F}$.

Within this proof, we will call a \glin{G}\ invertible sheaf $\L$ over $X$ ``good'' if $\L|_Z$ has trivial $G$-action. Otherwise, we will call it ``bad''. We will also call an invariant divisor $D$ on $X$ ``good'' (``bad'') if the corresponding \glin{G}\ invertible sheaf $\O_X(D)$ is ``good''(``bad''). 

For a set of invariant divisors \gdprdivisors\ on $X$ such that $\gdprdivisorsequi,$ we define a ring homomorphism $\cat{F}'$ from $\cat{R}$ to $\Endo{\go(Z)}$ by the following rules :

\begin{eqnarray}
X_i & \mapsto & \left\{ 
\begin{array}{cl} 
c(\O(A_i)) & \text{ if $A_i$ is good } \\ 
1 & \text{ if $A_i$ is bad } 
\end{array} \right. \nonumber\\
 & & \nonumber\\
U^1_k & \mapsto & \left\{ 
\begin{array}{cl} 
{(p^1_D)}_*{(p^1_D)}^* & \text{ if $D \defeq A_1 + \cdots + A_k$ is good }\\ 
2 & \text{ if $D$ is bad } 
\end{array} \right. \nonumber\\
&& \text{where $p^1_D : \P(\O \oplus \O(D)) \to Z$} \nonumber\\
& & \nonumber
\end{eqnarray}
\begin{eqnarray}
U^2_k & \mapsto & \left\{ 
\begin{array}{cl} 
{(p^2_k)}_*{(p^2_k)}^* & \text{ if $D$, $A_k$, $D + A_k$ are all good }\\
 & \text{where $D \defeq A_1 + \cdots + A_{k-1}$ } \\
2{(p^1_D)}_*{(p^1_D)}^* & \text{ if $D$ is good but $A_k$, $D + A_k$ are bad }\\ 
2 + {(p^1_{A_k})}_*{(p^1_{A_k})}^* & \text{ if $A_k$ is good but $D$, $D + A_k$ are bad }\\ 
2 + {(p^1_{D + A_k})}_*{(p^1_{D + A_k})}^* & \text{ if $D + A_k$ is good but $A_k$, $D$ are bad }\\ 
4 & \text{ if $D$, $A_k$, $D + A_k$ are all bad } 
\end{array} \right. \nonumber\\
&& \text{where $p^2_k : \P(\O \oplus \O(1)) \to \P(\O(-A_k) \oplus \O(- D - A_k)) \to Z$} \nonumber\\
& & \nonumber\\
U^3_k & \mapsto & \left\{ 
\begin{array}{cl} 
{(p^3_k)}_*{(p^3_k)}^* & \text{ if $D$, $A_k$, $D + A_k$ are all good }\\
 & \text{where $D \defeq A_1 + \cdots + A_{k-1}$ } \\
1 + {(p^1_D)}_*{(p^1_D)}^* & \text{ if $D$ is good but $A_k$, $D + A_k$ are bad }\\ 
1 + {(p^1_{A_k})}_*{(p^1_{A_k})}^* & \text{ if $A_k$ is good but $D$, $D + A_k$ are bad }\\ 
1 + {(p^1_{D + A_k})}_*{(p^1_{D + A_k})}^* & \text{ if $D + A_k$ is good but $A_k$, $D$ are bad }\\ 
3 & \text{ if $D$, $A_k$, $D + A_k$ are all bad } 
\end{array} \right. \nonumber\\
&& \text{where $p^3_k : \P(\O \oplus \O(-A_k) \oplus \O(- D - A_k)) \to Z$} \nonumber
\end{eqnarray}
if $1 \leq i,k \leq n$. Otherwise, send it to zero. Define $\cat{F}'(Y_j)$, $\cat{F}'(V^q_l)$ similarly by replacing ``$A$'' by ``$B$''. As shown in the proof of Proposition \ref{gdprholdprop}, $c(\O(D))$ and $p_* p^*$ commutes with each other. Hence, $\cat{F}'$ is well-defined. Notice that since, in the $U^2_k$, $U^3_k$ cases, we have $D + A_k \sim (A_1 + \cdots + A_k)$, it is impossible to have only one of $A_k$, $D$, $D + A_k$ being bad. Thus, the definition covers all possibilities.

\vtab

Claim 1 : $\cat{F}'(G^X_{n,m}) = \cat{F}'(G^Y_{m,n})$ as elements in $\Endo{\go(Z)}$.

By a similar symbolic cancelation as in the proof of Proposition \ref{gdprholdprop}, it is enough to show the claim in the case when $A + B \sim C$. In this case, 
$$G^X_{2,1} = X_1 + X_2 - X_1X_2U^1_1 + Y_1X_1X_2(U^2_2-U^3_2)$$
and 
$$G^Y_{1,2} = Y_1.$$
We will prove the claim case by case.

Case 1 : $A$, $B$, $C$ are all good.

In this case, 
\begin{eqnarray}
\cat{F}'(G^X_{2,1}) &=& \cat{F}'(X_1 + X_2 - X_1X_2U^1_1 + Y_1X_1X_2(U^2_2-U^3_2)) \nonumber\\
&=& c(\O(A)) + c(\O(B)) - c(\O(A))c(\O(B)){p^1_A}_*{p^1_A}^* \nonumber\\
&& +\ c(\O(C))c(\O(A))c(\O(B))(\, {p^2_2}_*{p^2_2}^* - {p^3_2}_*{p^3_2}^* \,) \nonumber\\
&& \text{where $p^1_A : \P(\O \oplus \O(A)) \to Z$} \nonumber\\
&& \text{$p^2_2 : \P(\O \oplus \O(1)) \to \P(\O(-B) \oplus \O(-C)) \to Z$} \nonumber\\
&& \text{$p^3_2 : \P(\O \oplus \O(-B) \oplus \O(-C)) \to Z$.} \nonumber
\end{eqnarray}
On the other hand,
\begin{eqnarray}
\cat{F}'(G^Y_{1,2}) &=& \cat{F}'(Y_1) \nonumber\\
&=& c(\O(C)). \nonumber
\end{eqnarray}
Thus, the difference $\cat{F}'(G^X_{2,1})-\cat{F}'(G^Y_{1,2})$ is exactly what we defined to be $H(\O(A), \O(B))$ in the proof of Proposition \ref{gdprholdprop}, which was proved to be zero.

Case 2 : $A$ is good but $B$, $C$ are bad.
\begin{eqnarray}
\cat{F}'(G^X_{2,1}) &=& c(\O(A)) + 1 - c(\O(A)){p^1_A}_*{p^1_A}^* \nonumber\\
&& +\ c(\O(A))(\, 2{p^1_A}_*{p^1_A}^* - 1 - {p^1_A}_*{p^1_A}^* \,) \nonumber\\
&=& 1 \nonumber\\
&=& \cat{F}'(G^Y_{1,2}). \nonumber
\end{eqnarray}

Case 3 : $B$ is good but $A$, $C$ are bad.
\begin{eqnarray}
\cat{F}'(G^X_{2,1}) &=& 1 + c(\O(B)) - c(\O(B))(2) \nonumber\\
&& +\ c(\O(B))(\, 2 + {p^1_B}_*{p^1_B}^* - 1 - {p^1_B}_*{p^1_B}^* \,) \nonumber\\
&=& 1 \nonumber\\
&=& \cat{F}'(G^Y_{1,2}). \nonumber
\end{eqnarray}

Case 4 : $C$ is good but $A$, $B$ are bad.
\begin{eqnarray}
\cat{F}'(G^X_{2,1}) &=& 1 + 1 - (1)(2) + c(\O(C))(\, 2 + {p^1_C}_*{p^1_C}^* - 1 - {p^1_C}_*{p^1_C}^* \,) \nonumber\\
&=& c(\O(C)) \nonumber\\
&=& \cat{F}'(G^Y_{1,2}). \nonumber
\end{eqnarray}

Case 5 : $A$, $B$, $C$ are all bad.
\begin{eqnarray}
\cat{F}'(G^X_{2,1}) &=& 1 + 1 - (1)(2) + (1)(4 - 3) \nonumber\\
&=& 1 \nonumber\\
&=& \cat{F}'(G^Y_{1,2}). \nonumber
\end{eqnarray}

That proves the claim. \claimend

\vtab

The next step is to verify the correspondence between $\cat{F}$ and $\cat{F}'$. To be more precise, let $\cat{G} : \cat{R} \to M_G(X)^+$ be the map corresponding to a GDPR setup given by $\gdprdivisorsequi$ on $X$ such that $\gdprdivisorssum$ is a \rsncd\ and let $\cat{F}' : \cat{R} \to \Endo{\go(Z)}$ be the map we just defined corresponding to this setup. Consider the fixed point map $\cat{F}$ as a map from $M_G(X)^+$ to $\go(Z)$. The equation we are going to prove is
\begin{eqnarray}
\cat{F} \circ \cat{G}(s) &=& \cat{F}'(s) [\id_Z] \label{eqn6}
\end{eqnarray}
for any element $s \in \Z\{ X_i \cdots Y_j \cdots U^p_k \cdots V^q_l \cdots\ |\ \text{power of any $X_i$, $Y_j \leq 1$} \}$.

Suppose equation (\ref{eqn6}) is true. Then,
\begin{eqnarray}
\cat{F} \circ \cat{G}(G^X_{n,m}) &=& \cat{F}'(G^X_{n,m}) [\id_Z] \nonumber\\
&=& \cat{F}'(G^Y_{m,n}) [\id_Z] \nonumber\\
&& \text{(by claim 1)} \nonumber\\
&=& \cat{F} \circ \cat{G}(G^Y_{m,n}), \nonumber
\end{eqnarray}
which is what we want. That means it is enough to verify equation (\ref{eqn6}). First of all, we need to understand the meaning of an \inv\ divisor being ``good''.

\vtab

Claim 2 : Suppose $D$ is a \sm\ invariant divisor on $X$. Then, $D$ is good if and only if $D \cap Z$ is a \sm\ divisor on $Z$. Also, $D$ is bad if and only if $D \cap Z = Z$.

First of all, observe that 
$$D \cap Z = D \x_X Z = D \x_X X^G \x_{X^G} Z = D^G \x_{X^G} Z,$$ 
which is always \sm. If $D \cap Z = \emptyset$, then $\O_Z(D) \cong \O_Z$. That means it is good and $D \cap Z$ is the zero divisor.

Suppose $D \cap Z$ is non-empty. Take a closed point $x \in D \cap Z$. Notice that since the action on $Z$ is trivial and $Z$ is irreducible, the action $\O_Z(D)$ is trivial if and only if the action on $\O_Z(D)|_x$ is trivial. Moreover, $\O_Z(D)|_x \cong \O_D(D)|_x \cong \nbundle{D}{X}|_x.$ Hence, the action on $\nbundle{D}{X}|_x$ is trivial if and only if $D$ is good.

Suppose the action on $\nbundle{D}{X}|_x$ is trivial and $D \cap Z$ is not a divisor on $Z$. That means $D \cap Z = Z$, i.e. $Z \subset D$. Thus, we have $\nbundle{Z}{X}|_x \supset \nbundle{D}{X}|_x$. It contradicts with the fact that there is no non-zero vector in $\nbundle{Z}{X}|_x$ which has trivial action (Proposition \ref{smoothfixedpointlocus}).

Suppose $D \cap Z$ is a divisor on $Z$. Then $D$ and $Z$ intersect transversely. That means the tangent space of $X$ at $x$ is spanned by the tangent spaces of $D$ and $Z$ at $x$, namely, $\cat{T}X|_x = \cat{T}D|_x + \cat{T}Z|_x.$ Hence, we have $\nbundle{D}{X}|_x \embed \cat{T}Z|_x$, which has trivial action. \claimend

\vtab

Suppose the \sm\ invariant divisor $A_i$ is good. Then, we have
\begin{eqnarray}
\cat{F} \circ \cat{G}(X_i) &=& \cat{F}[A_i \to X] \nonumber\\
&=& [A_i^G \x_{X^G} Z \to Z] \nonumber\\
&=& [A_i \cap Z \to Z] \nonumber\\
&=& c(\O(A_i))[\id_Z] \nonumber\\
&& \text{(by claim 2 and \textbf{(Sect)} axiom in the theory $\go$)} \nonumber\\
&=& \cat{F}'(X_i)[\id_Z]. \nonumber
\end{eqnarray}
On the other hand, if $A_i$ is bad, then we have $A_i \cap Z = Z$ by claim 2. In this case,
$$\cat{F} \circ \cat{G}(X_i) = [A_i^G \x_{X^G} Z \to Z] = [Z \to Z] = \cat{F}'(X_i)[\id_Z].$$
Hence, equation (\ref{eqn6}) holds for $X_i$ and $Y_j$.

For $U^1_k$, if $D \defeq A_1 + \cdots + A_k$ is good, then $\P(\O_Z \oplus \O_Z(D))$ has trivial action. Thus,
\begin{eqnarray}
\cat{F} \circ \cat{G}(U^1_k) &=& [\P(\O \oplus \O(D))^G \x_{X^G} Z \to Z] \nonumber\\
&=& [\P(\O_Z \oplus \O_Z(D)) \to Z] \nonumber\\
&=& {p^1_D}_* {p^1_D}^*[\id_Z] \nonumber\\
&& \text{where $p^1_D : \P(\O \oplus \O(D)) \to Z$} \nonumber\\
&=& \cat{F}'(U^1_k)[\id_Z]. \nonumber
\end{eqnarray}
If $D$ is bad, then $\P(\O_Z \oplus \O_Z(D))$ has non-trivial fiberwise action. That implies 
$$\P(\O_X \oplus \O_X(D))^G|_Z = \P(\O_Z \oplus \O_Z(D))^G = \P(\O_Z(D)) \disjoint \P(\O_Z).$$ 
Thus,
\begin{eqnarray}
\cat{F} \circ \cat{G}(U^1_k) &=& [\P(\O \oplus \O(D))^G \x_{X^G} Z \to Z] \nonumber\\
&=& [\P(\O_Z(D)) \disjoint \P(\O_Z) \to Z] \nonumber\\
&=& 2[\id_Z] \nonumber\\
&=& \cat{F}'(U^1_k)[\id_Z]. \nonumber
\end{eqnarray}
Hence, equation (\ref{eqn6}) holds for $U^1_k$ and $V^1_l$.

For $U^2_k$, let $D \defeq A_1 + \cdots + A_{k-1}$ as in the definition of $\cat{F}'$. There are five different cases to consider.

Case 1 (Divisors $D$, $A_k$, $D + A_k$ are all good) :

The action on the \proj\ bundle $\P(\O_Z(-A_k) \oplus \O_Z(-D - A_k))$ will be trivial and so is the \proj\ bundle $\P(\O \oplus \O(1))$ above it. Thus,
\begin{eqnarray}
\cat{F} \circ \cat{G}(U^2_k) &=& [\P(\O \oplus \O(1))^G \x_{X^G} Z \to Z] \nonumber\\
&=& [\P(\O \oplus \O(1)) \to \P(\O_Z(-A_k) \oplus \O_Z(-D - A_k)) \to Z] \nonumber\\
&=& {p^2_k}_* {p^2_k}^*[\id_Z] \nonumber\\
&& \text{($p^2_k$ as in the definition of $\cat{F}'$)} \nonumber\\
&=& \cat{F}'(U^2_k)[\id_Z]. \nonumber
\end{eqnarray}

Case 2 (Divisor $D$ is good but $A_k$, $D + A_k$ are bad) :

That means $\O_Z(-A_k)$ and $\O_Z(-D - A_k)$ have the same non-trivial fiberwise action. If we let $F : \picard{G}{Z} \to \picard{}{Z}$ be the forgetful map, $\L_2 \defeq F(\O_Z(-A_k))$ and $\L_3 \defeq F(\O_Z(- D -A_k))$, then there exists a non-trivial one-dimensional character $\gps$ such that $\O_Z(-A_k) \cong \L_2 \otimes \gps$ and $\O_Z(- D - A_k) \cong \L_3 \otimes \gps.$ So, we have
$$\P(\O_Z(-A_k) \oplus \O_Z(-D - A_k)) \cong \P( (\L_2 \otimes \gps) \oplus (\L_3 \otimes \gps) ) \cong \P(\L_2 \oplus \L_3),$$ 
which has trivial action. Moreover, this iso\morp\ takes $\O(1)$ to $\O(1) \otimes \gps$. Hence, the tower
$$\P(\O \oplus \O(1)) \to \P(\O_Z(-A_k) \oplus \O_Z(-D - A_k)) \to Z$$
is isomorphic to
$$\P( \O \oplus (\O(1) \otimes \gps) ) \to \P(\L_2 \oplus \L_3) \to Z.$$
Hence,
\begin{eqnarray}
\cat{F} \circ \cat{G}(U^2_k) &=& [\P(\O(1) \otimes \gps) \disjoint \P(\O) \to \P(\L_2 \oplus \L_3) \to Z] \nonumber\\
&=& 2[\P((\L_2 \otimes \dual{\L_3}) \oplus \O) \to Z] \nonumber\\
&=& 2[\P(\O(D) \oplus \O) \to Z] \nonumber\\
&=& 2{p^1_D}_* {p^1_D}^*[\id_Z] \nonumber\\
&=& \cat{F}'(U^2_k)[\id_Z]. \nonumber
\end{eqnarray}

Case 3 (Divisor $A_k$ is good but $D$, $D + A_k$ are bad) :

Similarly, we define $\L_2 \defeq F(\O_Z(-A_k))$ and $\L_3 \defeq F(\O_Z(- D -A_k))$ and we will have $\O_Z(-A_k) \cong \L_2$ and $\O_Z(- D - A_k) \cong \L_3 \otimes \gps$ for some non-trivial character $\gps$. The \proj\ bundle $\P(\O_Z(-A_k) \oplus \O_Z(-D - A_k))$ is then isomorphic to $\P( \L_2 \oplus (\L_3 \otimes \gps) )$, which has fixed point locus $\P(\L_2) \disjoint \P(\L_3 \otimes \gps)$. Observe that the restriction takes $\O(1)$ over $\P( \L_2 \oplus (\L_3 \otimes \gps) )$ to $\O(1)$ over $\P(\L_2)$. Moreover, the iso\morp\ $\P(\L_2)\iso \P(\O)$ takes $\O(1)$ to $\O(1) \otimes \L_2$. Hence, the tower
$$\P(\O \oplus \O(1)) \to \P(\L_2) \to Z$$
is isomorphic to
$$\P( \O \oplus (\O(1) \otimes \L_2) ) \to \P(\O) \to Z$$
which is simply $\P(\O \oplus \L_2) \to Z$. Similarly, the tower
$$\P(\O \oplus \O(1)) \to \P(\L_3 \otimes \gps) \to Z$$
is isomorphic to
$$\P( \O \oplus (\L_3 \otimes \gps) ) \to Z.$$
Hence,
\begin{eqnarray}
\cat{F} \circ \cat{G}(U^2_k) &=& [\P(\O \oplus \L_2) \to Z] + [\P(\L_3 \otimes \gps) \disjoint \P(\O) \to Z] \nonumber\\
&=& [\P(\O_Z \oplus \O_Z(-A_k)) \to Z] + 2[\id_Z] \nonumber\\
&=& [\P(\O(A_k) \oplus \O) \to Z] + 2[\id_Z] \nonumber\\
&=& ({p^1_{A_k}}_* {p^1_{A_k}}^* + 2) [\id_Z] \nonumber\\
&=& \cat{F}'(U^2_k)[\id_Z]. \nonumber
\end{eqnarray}

Case 4 (Divisor $D + A_k$ is good but $D$, $A_k$ are bad) :
 
Similarly, the fixed point locus of 
$$\P(\O_Z(-A_k) \oplus \O_Z(-D - A_k)) \cong \P( (\L_2 \otimes \gps) \oplus \L_3)$$
is the disjoint union of $\P(\L_2 \otimes \gps)$ and $\P(\L_3)$. Also, the tower corresponding to $\P(\L_2 \otimes \gps)$ is isomorphic to $\P( \O \oplus (\L_2 \otimes \gps) ) \to Z$ and the tower corresponding to $\P(\L_3)$ is isomorphic to $\P(\O \oplus \L_3) \to Z$. Hence,
\begin{eqnarray}
\cat{F} \circ \cat{G}(U^2_k) &=& 2[\id_Z] + [\P(\O(D + A_k) \oplus \O) \to Z] \nonumber\\
&=& (2 + {p^1_{D + A_k}}_* {p^1_{D + A_k}}^*) [\id_Z] \nonumber\\
&=& \cat{F}'(U^2_k)[\id_Z]. \nonumber
\end{eqnarray}

Case 5 (Divisors $D$, $A_k$, $D + A_k$ are all bad) :

In this case, we have $\O_Z(-A_k) \cong \L_2 \otimes \gps_2$ and $\O_Z(- D - A_k) \cong \L_3 \otimes \gps_3$ for distinct non-trivial characters $\gps_2$, $\gps_3$. Thus, the fixed point locus of $\P(\O_Z(-A_k) \oplus \O_Z(-D - A_k))$ will then be the disjoint union of $\P(\L_2 \otimes \gps_2)$ and $\P(\L_3 \otimes \gps_3)$ and towers corresponding to them are $\P( \O \oplus (\L_2 \otimes \gps_2) ) \to Z$ and $\P( \O \oplus (\L_3 \otimes \gps_3) ) \to Z$ respectively. Hence,
\begin{eqnarray}
\cat{F} \circ \cat{G}(U^2_k) &=& 2[\id_Z] + 2[\id_Z] \nonumber\\
&=& \cat{F}'(U^2_k)[\id_Z]. \nonumber
\end{eqnarray}

That proves equation (\ref{eqn6}) holds for $U^2_k$ and similarly for $V^2_l$.

\vtab

For $U^3_k$, similarly, let $D \defeq A_1 + \cdots + A_{k-1}$. In case 1,
$$\cat{F} \circ \cat{G}(U^3_k) = [\P(\O \oplus \O(-A_k) \oplus \O(-D - A_k)) \to Z] = {p^3_k}_*{p^3_k}^*[\id_Z] = \cat{F}'(U^3_k)[\id_Z].$$

In case 2, let $\L_2 \defeq F(\O_Z(-A_k))$ and $\L_3 \defeq F(\O_Z(- D -A_k))$ as before. Then,
\begin{eqnarray}
\cat{F} \circ \cat{G}(U^3_k) &=& [\P(\O) \disjoint \P( \L_2 \oplus \L_3 ) \to Z] \nonumber\\
&=& [\id_Z] + [\P(\O(D) \oplus \O) \to Z] \nonumber\\
&=& (1 + {p^1_D}_*{p^1_D}^*) [\id_Z] \nonumber\\
&=& \cat{F}'(U^3_k)[\id_Z]. \nonumber
\end{eqnarray}

In case 3, 
\begin{eqnarray}
\cat{F} \circ \cat{G}(U^3_k) &=& [\P(\L_3) \disjoint \P(\O \oplus \O(- A_k)) \to Z] \nonumber\\
&=& [\id_Z] + [\P(\O(A_k) \oplus \O) \to Z] \nonumber\\
&=& (1 + {p^1_{A_k}}_* {p^1_{A_k}}^*) [\id_Z] \nonumber\\
&=& \cat{F}'(U^3_k)[\id_Z]. \nonumber
\end{eqnarray}

In case 4, 
\begin{eqnarray}
\cat{F} \circ \cat{G}(U^3_k) &=& [\P(\L_2) \disjoint \P(\O \oplus \O(-D - A_k)) \to Z] \nonumber\\
&=& [\id_Z] + [\P(\O(D + A_k) \oplus \O) \to Z] \nonumber\\
&=& (1 + {p^1_{D + A_k}}_* {p^1_{D + A_k}}^*) [\id_Z] \nonumber\\
&=& \cat{F}'(U^3_k)[\id_Z]. \nonumber
\end{eqnarray}

In case 5, 
$$\cat{F} \circ \cat{G}(U^3_k) = [\P(\O) \disjoint \P(\L_2) \disjoint \P(\L_3) \to Z] = 3[\id_Z] = \cat{F}'(U^3_k)[\id_Z].$$

That proves equation (\ref{eqn6}) holds for $U^3_k$ and similarly for $V^3_l$.

\vtab

Let $s$, $t$ be two terms in 
$$\overline{\cat{R}} \defeq \Z\{ X_i \cdots Y_j \cdots U^p_k \cdots V^q_l \cdots\ |\ \text{power of any $X_i$, $Y_j \leq 1$} \}.$$ 
By definition, the domain of $\cat{G}(st) = \text{the domain of } \cat{G}(s) \x_X \text{the domain of } \cat{G}(t)$. For simplicity, we will focus on domains. By abuse of notation, we will still call it $\cat{G}$. Observe that
\begin{eqnarray}
\cat{F}[Y_1 \x_X Y_2 \to X] &=& [(Y_1 \x_X Y_2)^G \x_{X^G} Z \to Z] \nonumber\\
&=& [Y_1^G \x_{X^G} Y_2^G \x_{X^G} Z \to Z] \nonumber\\
&=& [(Y_1^G \x_{X^G} Z) \x_Z (Y_2^G \x_{X^G} Z) \to Z]. \nonumber
\end{eqnarray}
Hence, $\cat{F}(Y_1 \x_X Y_2) = \cat{F}(Y_1) \x_Z \cat{F}(Y_2)$, by abuse of notation again. Suppose $s \defeq X_i$, $Y_j$, $U^p_k$ or $V^q_l$ and $t$ is a term in $\overline{\cat{R}}$ such that $st$ is also in $\overline{\cat{R}}$. By induction, we assume equation (\ref{eqn6}) holds for $s$ and $t$. In that case,
\begin{eqnarray}
\cat{F} \circ \cat{G}(st) &=& \cat{F} [\cat{G}(st) \to X] \nonumber\\
&=& \cat{F} [\cat{G}(s) \x_X \cat{G}(t) \to X] \nonumber\\
&=& [\cat{F}( \cat{G}(s) \x_X \cat{G}(t) ) \to Z] \nonumber\\
&=& [\cat{F} \circ \cat{G}(s) \x_Z \cat{F} \circ \cat{G}(t) \to Z]. \nonumber
\end{eqnarray}
On the other hand,
$$\cat{F}'(st)[\id_Z] = \cat{F}'(s) \circ \cat{F}'(t)[\id_Z] = \cat{F}'(s)[\cat{F} \circ \cat{G}(t) \to Z]$$
by induction assumption. Denote $\cat{F} \circ \cat{G}(t)$ by $Y$ and $Y \to Z$ by $f$. By the above calculation, 
$$[\cat{F} \circ \cat{G}(s) \to Z] = m_1 [\id_Z] + m_2[\P \to Z] + m_3 [D \cap Z \embed Z]$$ 
for some non-negative integers $m_1, m_2, m_3$, tower $\P$ and good, \sm, \inv\ divisor $D$ on $X$.

\vtab

Claim 3 : The map $\cat{F} \circ \cat{G}(s) \to Z$ is transverse to $f : Y \to Z$.

The claim is clearly true for $[\id_Z]$ and $[\P \to Z]$. So, we only need to consider the map $[D \cap Z \embed Z]$ where $D$ is a \good, \sm, invariant divisor on $X$. Recall that 
$$Y = \cat{F} \circ \cat{G}(X_i \cdots U^p_k \cdots) = \cat{F}(A_i) \x_Z \cdots \x_Z \cat{F}(P^p_k) \x_Z \cdots.$$ 
Since $\cat{F}(P^p_k)$ is the sum of towers and $\cat{F}(A_i) = Z$ when $A_i$ is bad, we may assume it only involves \good\ divisors, i.e.
$$\cat{F}(A_{i_1}) \x_Z \cdots \x_Z \cat{F}(B_{j_1}) \x_Z \cdots = A_{i_1} \cap \cdots \cap B_{j_1} \cap \cdots \cap Z.$$

Notice that since $st$ is in $\overline{\cat{R}}$, the divisor $D$ and the set of divisors $\{A_{i_1}, \cdots, B_{j_1}, \cdots\}$ are all distinct. For simplicity, we will only show the transversality involving \good\ divisors $D$, $D'$. More precisely, we will show if $D$, $D'$ are \good, \sm, invariant divisors on $X$ such that $D + D'$ is a \rsncd, then $D \cap Z + D' \cap Z$ is a \rsncd\ on $Z$. 

Since $X$ is \equidim, $D$ is \equidim. Let $W$ be an irreducible component of $Z \cap D$. Then, $D \in \gsmcat{G}$ is \equidim, $D \cap D'$ is an \inv\ \sm\ divisor on $D$ and $W$ is an irreducible component of the fixed point locus of $D$. Notice that 
$$\O_W(D \cap D') \cong \O_X(D')|_W \cong \O_Z(D')|_W.$$
Thus, $D \cap D'$ is a good divisor on $D$ \wrt\ $W$, for all $W$, because $D'$ is a good divisor on $X$ \wrt\ $Z$. By applying claim 2 with $X$, $D$, $Z$ replaced by $D$, $D \cap D'$, $W$ respectively, $D \cap D' \cap W$ is a \sm\ divisor on $W$. So, $D \cap D' \cap Z$ is a \sm\ divisor on $D \cap Z$. Hence, $D \cap Z$ and $D' \cap Z$ intersect transversely inside $Z$. \claimend

\vtab

Let $\{Y_i\}$ be the irreducible components of $Y$. Notice that $f$ is \proj. So, the push-forward $f_* : \go(Y) \to \go(Z)$ is well-defined. Since $Y$, $Z$ are both \sm\ and \qproj, the map $f$ is a local complete intersection \morp\ (See section 5.1.1 in \cite{LeMo}). In addition, the algebraic cobordism theories $\go$ and $\gO$ are canonically isomorphic (Theorem 1 in \cite{LePa}) and, for any local complete intersection \morp\ $g : X \to X'$ with \equidim\ domain and codomain, the pull-back $g^* : \gO(X') \to \gO(X)$ is well-defined (see definition 6.5.10 in \cite{LeMo}). Hence, $f^* : \go(Z) \to \oplus_i\, \go(Y_i) \cong \go(Y)$ is also well-defined.

Suppose we have shown that 
\begin{eqnarray}
\label{eqn13}
\cat{F}'(s)[f : Y \to Z] = f_*f^* \cat{F}'(s)[\id_Z]. 
\end{eqnarray}
Then, we have
\begin{eqnarray}
\cat{F}'(st)[\id_Z] &=& \cat{F}'(s)[f : Y \to Z] \nonumber\\
&=& f_*f^*\cat{F}'(s)[\id_Z] \nonumber\\
&=& f_*f^*[\cat{F} \circ \cat{G}(s) \to Z] \nonumber\\
&& \text{(by induction assumption)} \nonumber\\
&=& [(\cat{F} \circ \cat{G}(s)) \x_Z (\cat{F} \circ \cat{G}(t)) \to Z] \nonumber\\
&& \text{(by claim 3 and Theorem 6.5.12 in \cite{LeMo})} \nonumber\\
&=& \cat{F} \circ \cat{G}(st). \nonumber
\end{eqnarray}
That means equation (\ref{eqn6}) holds for $st \in \overline{R}$. Hence, it remains to show equation (\ref{eqn13}).

By the previous calculation, 
$$\cat{F}'(s) = m_1 + m_2\, p_*p^* + m_3\, c(\O_Z(D))$$ 
for some non-negative integers $m_1, m_2, m_3$, \sm, \proj\ map $p : \P \to Z$ and good, \sm, \inv\ divisor $D$ on $X$. The equation obviously holds for the identity operator. For $c(\O_Z(D))$,
\begin{eqnarray}
c(\O_Z(D)) [Y \to Z] &=& [(D \cap Z) \x_Z Y \to Z] \nonumber\\
&& \text{(by claim 3 and \textbf{(Sect)} axiom in $\go$)} \nonumber\\
&=& f_*f^* [D \cap Z \to Z] \nonumber\\
&& \text{(by claim 3 and the Theorem 6.5.12 in \cite{LeMo})} \nonumber\\
&=& f_*f^* c(\O_Z(D))\, [\id_Z]. \nonumber
\end{eqnarray}
For $p_*p^*$,
\begin{eqnarray}
p_*p^*[Y \to Z] &=& [\P \x_Z Y \to Z] \nonumber\\
&=& f_*f^* [\P \to Z] \nonumber\\
&& \text{(by Theorem 6.5.12 in \cite{LeMo})} \nonumber\\
&=& f_*f^* p_*p^*[\id_Z]. \nonumber
\end{eqnarray}
That proves equation (\ref{eqn13}) and hence finishes the proof.
\end{proof}

\begin{cor}
\label{fixedpointmap2}
If $X$ is an object in $\gsmcat{G}$, then sending $[Y \to X]$ to $[Y^G \to X^G]$ defines an abelian group homomorphism
$$\fixptmap : \cobg{G}{}{X} \to \go(X^G).$$
\end{cor}

\begin{proof}
Let $\{Z\}$ be the set of irreducible components of the fixed point locus $X^G$. By Theorem \ref{fixedpointmap}, sending $[Y \to X]$ to $\sum_Z\ [Y^G \x_{X^G} Z \to Z]$ defines an abelian group homo\morp\ $\cobg{G}{}{X} \to \oplus_Z\, \go(Z).$ Then, the map $\fixptmap : \cobg{G}{}{X} \to \go(X^G)$ can be considered as the composition
$$\cobg{G}{}{X} \to \oplus_Z\, \go(Z) \to \oplus_Z\, \go(X^G) \to \go(X^G)$$
defined by sending
\begin{eqnarray}
[Y \to X] &\mapsto& \sum_Z\ [Y^G \x_{X^G} Z \to Z] \nonumber\\
&\mapsto& \sum_Z\ [Y^G \x_{X^G} Z \to Z \embed X^G] \nonumber\\
&\mapsto& \sum_Z\ [Y^G \x_{X^G} Z \to Z \embed X^G] = [Y^G \to X^G]. \nonumber
\end{eqnarray}
\end{proof}

\begin{cor}
\label{embedlazardcor}
Suppose $X$ is an object in $\gsmcat{G}$ with trivial $G$-action. Then, the abelian group $\go(X) \cong \cobg{\{1\}}{}{X}$ is a direct summand of $\cobg{G}{}{X}$ via the homo\morp\ 
$$\gPH_{\gg} : \cobg{\{1\}}{}{X} \to \cobg{G}{}{X}$$ 
induced by the group homo\morp\ $\gg : G \to \{1\}$. In particular, the Lazard ring $\lazard$ is naturally a subring of the \equi\ algebraic cobordism ring $\cobg{G}{}{\pt}$.
\end{cor}

\begin{proof}
The fixed point map 
$$\cat{F} : \cobg{G}{}{X} \to \go(X^G) = \go(X) \cong \cobg{\{1\}}{}{X}$$ 
is a left inverse of the homo\morp\ $\cobg{\{1\}}{}{X} \to \cobg{G}{}{X}$. Also, 
$$\gPH_{\gg} : \lazard \cong \cobg{\{1\}}{}{\pt} \to \cobg{G}{}{\pt}$$ 
is a ring \homo. 
\end{proof}

\bigskip

\bigskip

%Bibliography	

\end{document}